\documentclass[reqno]{amsart}
\usepackage{amsfonts,amsmath,amssymb}

\numberwithin{equation}{section}

\def\be{{\beta}}

\def\eps{\epsilon}
\def\ka{\kappa}

\def\si{\sigma}

\def\ka{\kappa}

\def\phii{\widetilde{\varphi}}

\def\al{\alpha}

\newtheorem{theorem}{Theorem}[section]
\newtheorem{lemma}[theorem]{Lemma}

\newtheorem{proposition}[theorem]{Proposition}
\newtheorem{definition}[theorem]{Definition}
\newtheorem{remark}[theorem]{Remark}

\begin{document}

\title[Klein--Gordon systems in 3D]{Global solutions of quasilinear systems of Klein--Gordon equations in 3D}

\author{Alexandru D. Ionescu}
\address{Princeton University}
\email{aionescu@math.princeton.edu}

\author{Benoit Pausader}
\address{LAGA, Universit\'e Paris 13 (UMR 7539)}
\email{pausader@math.u-paris13.fr}

\thanks{The first author was partially supported by a Packard Fellowship and NSF
 grant DMS-1065710. The second author was partially supported by NSF grant DMS-1142293.}

\begin{abstract}
We prove small data global existence and scattering for quasilinear systems of Klein-Gordon equations with different speeds, in dimension three. As an application, we obtain a robust global stability result for the Euler-Maxwell equations for electrons.
\end{abstract}

\maketitle

\tableofcontents

\section{Introduction}\label{intro}

In this paper we consider systems of quasilinear Klein--Gordon equations with different speeds and masses in dimension three. 
Our aim is to prove that small, smooth, and localized initial data lead to global solutions, assuming only certain mild non-degeneracy conditions which are automatically satisfied in our main applications. The method we develop appears to be robust enough to deal with many situations that involve large space-time resonant sets, at least in dimension three.

We will focus on two examples which should be sufficient to illustrate the scope of our method. We first consider quasilinear 
systems of Klein-Gordon type with pointwise quadratic nonlinearities
\begin{equation}\label{SKG}
\left(\partial_{tt}-c_\sigma^2\Delta+b_\sigma^2\right)u_\sigma=F_\sigma, \quad \sigma\in\{1,\ldots,d\},
\end{equation}
satisfying a hyperbolicity condition on the quasilinear term in the nonlinearity. Variations on such systems have been proposed in \cite{Kh} to model bilayer materials. This problem also appears in \cite{Ge} as an important toy model. More specifically, this problem when the speeds are the same has received a lot of attention in low dimensions \cite{De,Ha,Su}.

Our second model case is the Euler-Maxwell system for electrons. This is a simplification of the two fluid Euler-Maxwell system, which is one of the main models in plasma physics. We refer to \cite{Bi} for some physical reference and to \cite{GeMa,Guo} for previous mathematical study of the solutions. The system describes the dynamical evolution of the functions $n_e:\mathbb{R}^3\to\mathbb{R}$, $v_e,E',B':\mathbb{R}^3\to\mathbb{R}^3$, i.e.
\begin{equation}\label{EM}
\begin{split}
&\partial_tn_e+\hbox{div}(n_ev_e)=0,\\
&\partial_tv_e+v_e\cdot\nabla v_e=-\frac{P_e}{m_e}\nabla n_e-\frac{e}{m_e}\left[E'+\frac{v_e}{c}\times B'\right],\\
&\partial_tB'+c\nabla\times E'=0,\\
&\partial_tE'-c\nabla\times B'=4\pi en_ev_e,\\
\end{split}
\end{equation}
together with the elliptic equations
\begin{equation}\label{Ell}
\hbox{div}(B')=0,\quad \hbox{div}(E')=-4\pi e(n_e-n^0).
\end{equation}
Here $e>0$ is the electron charge, $P_e$ is related to the effective electron temperature\footnote{More precisely, $k_BT_e=n^0P_e$, where $k_B$ is the Boltzmann constant.}, $m_e$ is the mass of an electron and $c$
denotes the speed of light. The two equations \eqref{Ell} are propagated by the dynamic flow, provided that they are satisfied at the initial time. In addition, we make the following irrotationality assumption which removes a non decaying component,
\begin{equation}\label{IrrAss}
B'(0)=\frac{m_ec}{e}\nabla\times v_e(0),
\end{equation}
and which is also propagated by the flow and remains valid for all times. 

In the case of the system \eqref{EM}--\eqref{IrrAss} we want to explore the stability of the equilibrium solution $(n^0_e,v^0_e,E^0,B^0)=(n^0,0,0,0)$, $n^0>0$. In the system above, we have chosen a quadratic pressure $p(n_e)=P_en_e^2/2$. This is chosen only to minimize the number of terms in the nonlinearity but does not make the system \eqref{EM} symmetric and in particular, one needs to add a cubic correction to the energy estimates.

In both cases \eqref{SKG} and \eqref{EM}-\eqref{IrrAss}, we prove that small, localized, and smooth initial data lead to global classical solutions that scatter. Below is a precise description of the main results.

\subsection{Statement of the results}

Given a real-valued vector $u=(u_1,\ldots,u_d):\mathbb{R}^3\times[0,T]\to\mathbb{R}^d$, $u\in C([0,T]:H^N_r)\cap C^1([0,T]:H^{N-1}_r)$\footnote{In the paper we let $H^N=H^N_{(m)}$ denote standard $L^2$-based Sobolev spaces of complex vector-valued functions $f:\mathbb{R}^3\to\mathbb{C}^m$, $m=1,2,\ldots$. We let $H^N_r=H^N_{r,(m)}$ denote $L^2$-based Sobolev spaces of real vector-valued functions $f:\mathbb{R}^3\to\mathbb{R}^m$, $m=1,2,\ldots$.}, for some $T\geq 0$, $d\geq 1$, and $N\geq 5$, we consider quadratic nonlinearities of the form
\begin{equation}\label{KGNon1}
F_\mu:=\sum_{j,k=1}^3\sum_{\nu=1}^dG_{\mu\nu}^{jk}\partial_j\partial_ku_\nu+Q_\mu,
\end{equation}
where, with $\partial_0:=\partial_t$,
\begin{equation}\label{KGNon2}
G_{\mu\nu}^{jk}=G_{\mu\nu}^{jk}(u,\nabla_{x,t} u):=\sum_{\sigma=1}^d\left(\sum_{l=0}^3g^{jkl}_{\mu\nu\sigma}\partial_l u_\sigma+h^{jk}_{\mu\nu\sigma}u_\sigma\right),\qquad g^{jkl}_{\mu\nu\sigma},h^{jk}_{\mu\nu\sigma}\in\mathbb{R},
\end{equation}
and $Q_\mu=Q_\mu(u,\nabla_{x,t} u)$ is an arbitrary quadratic form (with real constant coefficients) in $(u_\sigma,\partial_ku_\sigma)$, $\sigma\in\{1,\ldots,d\},k\in\{0,1,2,3\}$. We assume that $G_{\mu\nu}^{jk}$ are symmetric in both $\mu,\nu$ and $j,k$ (the latter not being a restriction in generality), i.e.
\begin{equation}\label{CondForEE}
g^{jkl}_{\mu\nu\sigma}=g^{jkl}_{\nu\mu\sigma}=g^{kjl}_{\mu\nu\sigma},\qquad h^{jk}_{\mu\nu\sigma}=h^{jk}_{\nu\mu\sigma}=h^{kj}_{\mu\nu\sigma},
\end{equation}
for all choices of $j,k,l$ and $\mu,\nu,\sigma$.

We consider general systems of Klein--Gordon equations of the form
\begin{equation*}
(\partial_t^2-c^2_\mu\Delta+b_\mu^2)u_\mu=F_\mu,\qquad\qquad\mu=1,\ldots,d,
\end{equation*}
where the coefficients $b_1,\ldots,b_d,c_1,\ldots,c_d$ satisfy the non-degeneracy conditions \eqref{bc} below and the quadratic nonlinearities $F_\mu$ are as before. Our first main theorem concerns the global stability of the equilibrium solution $u\equiv 0$. More precisely:

\begin{theorem}\label{MainThm1}
Assume $A\geq 1$, $d\geq 1$, and $b_1,\ldots,b_d,c_1,\ldots,c_d\in[1/A,A]$ satisfy the non-resonance conditions
\begin{equation}\label{bc}
\begin{split}
&|b_{\sigma_1}+b_{\sigma_2}-b_{\sigma_3}|\geq 1/A\qquad\qquad\qquad\qquad\qquad\quad\,\text{ for any }\sigma_1,\sigma_2,\sigma_3\in\{1,\ldots,d\},\\
&|c_{\sigma_1}-c_{\sigma_2}|,\,|b_{\sigma_1}-b_{\sigma_2}|\in\{0\}\cup[1/A,\infty)\qquad\qquad\,\text{ for any }\sigma_1,\sigma_2\in\{1,\ldots,d\},\\
&(c_{\sigma_1}-c_{\sigma_2})(c_{\sigma_1}^2b_{\sigma_2}-c_{\sigma_2}^2b_{\sigma_1})\geq 0\qquad\qquad\qquad\quad\,\,\text{ for any }\sigma_1,\sigma_2\in\{1,\ldots,d\}.
\end{split}
\end{equation}
We fix quadratic nonlinearities $(F_{\mu})_{\mu\in\{1,\ldots,d\}}$ as in \eqref{KGNon1}--\eqref{CondForEE}, let $N_0=10^4$, and assume that $v_0,v_1:\mathbb{R}^3\to\mathbb{R}^d$ satisfy the smallness conditions
\begin{equation}\label{maincond1}
\|v_0\|_{H^{N_0+1}_r}+\|v_1\|_{H^{N_0}_r}+\|(1-\Delta)^{1/2}v_0\|_Z+\|v_1\|_Z=\varepsilon_0\leq\overline{\varepsilon},
\end{equation}
where $\overline{\varepsilon}=\overline{\varepsilon}(d,A,F_\mu)>0$ is sufficiently small (depending only on $d$, $A$, and the constants in the definition of the nonlinearities $F_\mu$), and the $Z$ norm is defined in Definition \ref{MainDef}. 

Then there exists a unique global solution $u\in C([0,\infty):H^{N_0+1}_r)\cap C^1([0,\infty):H^{N_0}_r)$ of the system 
\begin{equation}\label{mainsystem}
(\partial_t^2-c^2_\mu\Delta+b_\mu^2)u_\mu=F_\mu,\qquad\qquad\mu=1,\ldots,d,
\end{equation}
with initial data $(u(0),\dot{u}(0))=(v_0,v_1)$. Moreover, with $\beta=1/100$,
\begin{equation}\label{mainconcl1.1}
\begin{split}
\sup_{t\in[0,\infty)}&\big[\|u(t)\|_{H^{N_0+1}_r}+\|\dot{u}(t)\|_{H^{N_0}_r}\big]\\
&+\sup_{t\in[0,\infty)}(1+t)^{1+\beta}\big[\sup_{|\rho|\leq 4}\|D_x^\rho u(t)\|_{L^\infty}+\sup_{|\rho|\leq 3}\|D_x^\rho\dot{u}(t)\|_{L^\infty}\big]\lesssim \varepsilon_0.
\end{split}
\end{equation}
\end{theorem}

\begin{remark}\label{remark1}
(i) The non-degeneracy condition \eqref{bc} is automatically satisfied if the masses are all equal, $b_1=\ldots=b_d$, which is the case in our main application below to the Euler--Maxwell system.

(ii) Qualitatively, our condition on the parameters is 
\begin{equation*}
\begin{split}
&b_1,\ldots,b_d,c_1,\ldots,c_d\in(0,\infty),\\
&|b_{\sigma_1}+b_{\sigma_2}-b_{\sigma_3}|\neq 0\qquad\qquad\qquad\qquad\qquad\quad\,\text{ for any }\sigma_1,\sigma_2,\sigma_3\in\{1,\ldots,d\},\\
&(c_{\sigma_1}-c_{\sigma_2})(c_{\sigma_1}^2b_{\sigma_2}-c_{\sigma_2}^2b_{\sigma_1})\geq 0\qquad\qquad\qquad\,\,\text{ for any }\sigma_1,\sigma_2\in\{1,\ldots,d\}.
\end{split}
\end{equation*}
The point of the quantitative formulation in \eqref{bc}, in terms of the large parameter $A$, is to indicate the exact dependence of the smallness parameter $\overline{\varepsilon}$ in \eqref{maincond1}. 

(iii) The condition \eqref{bc} can certainly be relaxed. We have chosen this condition mostly because it is automatically satisfied in our application to the Euler--Maxwell system, can be explained conceptually in terms of the non-degeneracy of the space-time resonant sets, see subsection \ref{comments}, and reduces the amount of technical work. However, it seems natural to raise the question of whether this condition can be eliminated completely.
\end{remark}

We turn now to the Euler--Maxwell system. Recalling the system \eqref{EM}, we make the changes of variables
\begin{equation*}
\begin{split}
&n_e(x,t)=n^0[1+n(\lambda x,\lambda t)],\\
&v_e(x,t)=v(\lambda x,\lambda t),\\
&E'(x,t)=ZE(\lambda x,\lambda t),\\
&B'(x,t)=cZB(\lambda x,\lambda t),
\end{split}
\end{equation*}
where
\begin{equation*}
\lambda:=\sqrt{\frac{4\pi e^2n^0}{m_e}},\qquad\qquad Z:=\frac{\lambda m_e}{e}=\frac{4\pi en^0}{\lambda}.
\end{equation*}
The system \eqref{EM} becomes
\begin{equation}\label{EM2}
\begin{split}
&\partial_tn+\hbox{div}((1+n)v)=0,\\
&\partial_tv+v\cdot\nabla v+T\nabla n+E+v\times B=0,\\
&\partial_tB+\nabla\times E=0,\\
&\partial_tE-c^2\nabla\times B=(1+n)v,\\
\end{split}
\end{equation}
where\footnote{$\lambda$ is often called the ``electron plasma frequency'', $Z^2$ is the density of mass, and $\sqrt{T}$ is then the thermal velocity.}
\begin{equation*}
T:=\frac{P_en^0}{m_e}>0.
\end{equation*}

For any $N\geq 4$ we define the normed space
\begin{equation}\label{basicspace}
\begin{split}
\widetilde{H}^N:=\{&(n,v,E,B):\mathbb{R}^3\to\mathbb{R}\times\mathbb{R}^3\times\mathbb{R}^3\times\mathbb{R}^3:\\
&\|(n,v,E,B)\|_{\widetilde{H}^N}:=\|n\|_{H^N_r}+\|v\|_{H^N_r}+\|E\|_{H^N_r}+\|B\|_{H^N_r}<\infty\}.
\end{split}
\end{equation}

We can now state our second main theorem.

\begin{theorem}\label{MainThm2}
Let $N_0=10^4$ and assume that $(n_0,v_0,E_0,B_0)\in {\widetilde{H}^{N_0+1}}$ satisfies
\begin{equation}\label{maincond2}
\begin{split}
&\|(n_0,v_0,E_0,B_0)\|_{\widetilde{H}^{N_0+1}}+\|(1-\Delta)^{1/2}E_0\|_Z+\|(1-\Delta)^{1/2}v_0\|_Z=\varepsilon_0\leq\overline{\varepsilon},\\
&n_0=-\hbox{div}(E_0),\qquad B_0=\nabla\times v_0,
\end{split}
\end{equation}
where $\overline{\varepsilon}=\overline{\varepsilon}(c,T)>0$ is sufficiently small, and the $Z$ norm is defined in Definition \ref{MainDef}. Then there exists a unique global solution $(n,v,E,B)\in C([0,\infty):{\widetilde{H}^{N_0+1}})$ of the system \eqref{EM2} with initial data $(n(0),v(0),E(0),B(0))=(n_0,v_0,E_0,B_0)$. Moreover,
\begin{equation}\label{mainconcl2}
n(t)=-\hbox{div}(E(t)),\qquad B(t)=\nabla\times v(t),\qquad\text{ for any }t\in[0,\infty),
\end{equation}
and, with $\beta=1/100$,
\begin{equation}\label{mainconcl2.1}
\sup_{t\in[0,\infty)}\|(n(t),v(t),E(t),B(t))\|_{\widetilde{H}^{N_0+1}}+\sup_{t\in[0,\infty)}\sup_{|\rho|\leq 4}(1+t)^{1+\beta}\big(\|D^\rho_x v(t)\|_{L^\infty}+\|D^\rho_x E(t)\|_{L^\infty}\big)\lesssim \varepsilon_0.
\end{equation}
\end{theorem}

We remark that our restriction $n=-\hbox{div}(E)$, together with the assumptions on $E$, can only be satisfied if $\int_{\mathbb{R}^3}n(t)\,dx=0$, which means that we are only considering electrically neutral perturbations.

\subsection{Comments and plan of the proof}\label{comments}

\subsubsection{Previous results on systems of Klein-Gordon equations.} Systems of wave and Klein-Gordon-type equations have been studied by many authors, as they appear as natural models of physical evolutions. We also refer the reader to the introduction of \cite{Ge} for a review on previous works. 

The scalar case (or the system when all the speeds are equal and all the masses are equal) has been studied extensively. Some key developments include the work of John \cite{Joh} showing that blow-up in finite time {\it can} happen even for small smooth localized initial data of a semilinear wave equation, the introduction of the vector field method by Klainerman \cite{Kl} and of the normal form transformation by Shatah \cite{Sh}, and the understanding of the role of "null structures", starting with the works of Klainerman \cite{Kl2} and Christodoulou \cite{Ch}. Recently, a convenient general framework, which explains all of these results in the constant-coefficient case in terms of the concept of {\it{space-time resonances}}, was introduced independently by Germain-Masmoudi-Shatah \cite{GeMaSh} and Gustasfon-Nakanishi-Tsai \cite{GuNaTs}. We will get back to this later in this subsection.

The case of systems of wave equations with different speeds is well understood, both in the semilinear and the quasilinear case (see \cite{Yo} and \cite{SiTu}), provided that the nonlinearities satisfy appropriate null conditions, similar to those in the scalar case.

The case of Klein--Gordon quasilinear systems with equal speeds, $c_1=\ldots=c_d=1$, and different masses is also well understood both in dimensions two and three. For example, in \cite{De}, the authors show that if $b_{\sigma_1}+b_{\sigma_2}-b_{\sigma_3}\neq 0$ for any $\sigma_1,\sigma_2,\sigma_3$, then one has global existence and scattering in dimension two. If this condition is violated, then the same conclusion holds if the nonlinearity satisfies an appropriate null condition. We refer to \cite{Ha,Su,Ts} for related works.

As pointed out in \cite{Ge} a key new difficulty (the presence of a large set of space-time resonances) arises when the velocities are allowed to be different. In \cite{Ge}, the author studies semilinear systems of two Klein-Gordon equations when the masses are equal, $b_1=b_2$ in dimension {\it three}. Under a less explicit assumption on the velocities that covers most (but excludes some) parameters, he obtains global existence and scattering with a weak decay like $t^{-1/2}$ of the solution as $t\to \infty$.

In \cite{GeMa}, the authors study the Euler-Maxwell system for the electron \eqref{EM}-\eqref{IrrAss} in dimension {\it three} and obtain global existence and scattering  with weak decay by an elaborate iterated energy estimate. The results are conditional on $c$ and $T$ satisfying an implicit relation that holds for most (but not all) values of $T,c$.

With respect to the previous works, we remark that our result in this paper is obtained by a robust method, which yields time-integrability of the solution in $L^\infty$ and holds for {\it all} values of the velocities when the masses are equal. In addition, our smallness assumption is expressed explicitly in terms of the parameters, and the number of the derivatives $N_0$ needed is quantified (although most likely not optimal).

\subsubsection{General strategy.} Systems \eqref{SKG} and \eqref{EM} are hyperbolic systems of conservation laws and no general theory exists yet for such systems, even for the scalar case. Indeed, systems which are remarkably similar to \eqref{SKG} can be shown to have rather opposite behavior, even for small, smooth initial data, from blow up in finite time for all positive solutions of the quadratic wave equation \cite{Joh} to global existence and scattering for the quadratic scalar Klein-Gordon equation \cite{Sh}. The case for systems is even more complicated and only few partial results are known \cite{De,Ge,Ha}.

We follow and extend the analysis started in our previous work \cite{IoPa}. We refer to \cite{De2,GeMaSh,GuNaTs,Kl,Sh} for previous seminal work on dispersive quasilinear systems. The main two challenges we face are:

(i) overcoming the quasilinear nature of the nonlinearity to ensure global existence,

(ii) obtaining decay of the solution to control the asymptotic behavior. 

Fortunately, these two difficulties are complementary provided one obtains sufficiently strong control. Indeed:

(I) the loss of derivative coming from the nonlinearity is overcome by using energy estimates which allow to control high-regularity norms provided a lower-order norm remains small,

(II) the decay estimate, if implying time-integrability precisely propagates the smallness of low regularity norms globally in time. This is obtained from a delicate semilinear analysis assuming that high regularity norms remain bounded. Together, these two ingredients allow a bootstrap in time which yields both global existence and scattering.

The energy estimates come from the conservative structure of the equation and depend on delicate symmetry properties of the nonlinearity. In order to be extended globally, they demand a decay of some norm of at least $1/t$.

This decay is provided by the semilinear analysis of systems of dispersive equations. We use the Fourier transform method. After suitable algebraic manipulations, this is reduced to the study of bilinear operators of the form
\begin{equation}\label{OPT}
\widehat{T[f,g]}(\xi)=\int_{\mathbb{R}}\int_{\mathbb{R}^3}e^{it\Phi(\xi,\eta)}m(\xi,\eta)\widehat{f}(\xi-\eta,t)\widehat{g}(\eta,t) d\eta dt.
\end{equation}
As a first approximation, one may think of $f$, $g$ being smooth bump functions and $m$ being essentially a smooth cut-off, and the main challenge is to estimate efficiently the infinite time integral.  It then becomes clear that a key role is played by the properties of the function $\Phi$ and in particular by the points where it is stationary,
\begin{equation*}
\nabla_{(t,\eta)}[t\Phi(\xi,\eta)]=0.
\end{equation*}
This was already highlighted in \cite{GeMaSh} and forms the basis of the space-time resonance method. In some situations, one has no or few fully stationary points and the task is mainly to propagate enough smoothness of $\hat{f}$, $\hat{g}$ to exploit (non)-stationary phase arguments.

However, this is not the case in the models in this paper and we have to face some unavoidable ``space-time resonances''. Under some conditions that enforce non-degeneracy of the phase at critical points, we perform a robust stationary phase analysis of this case. We believe this forms the main contribution of the present work and we present it below in more details.

\subsubsection{Space-time resonant sets.}
The analysis of operators of the form \eqref{OPT} relies especially on the properties of the phase $\Phi$ which, in our case, is of the form
\begin{equation*}
\Phi=\Lambda_\sigma(\xi)\pm\Lambda_\mu(\xi-\eta)\pm\Lambda_\nu(\xi-\eta),\qquad \Lambda_\rho(\theta)=\sqrt{b_\rho^2+c_\rho^2\vert\theta\vert^2},\,\rho\in\{\sigma,\mu,\nu\}.
\end{equation*}
As in \cite{GeMaSh}, one can define the space-resonant set
\begin{equation*}
\mathcal{R}_x=\{(\xi,\eta):\,\,\nabla_\eta\Phi(\xi,\eta)=0\},
\end{equation*}
the time-resonant set
\begin{equation*}
\mathcal{R}_t=\{(\xi,\eta):\,\,\Phi(\xi,\eta)=0\},
\end{equation*}
and the set of space-time resonances
\begin{equation*}
\mathcal{R}=\mathcal{R}_x\cap \mathcal{R}_t.
\end{equation*}
The absence of {\it any} stationary point corresponds to the condition $\mathcal{R}=\emptyset$. This holds in a certain number of cases and the semilinear analysis can be carried out using integration by parts arguments either in $x$ or in $t$. It is remarkable that the simple condition $\mathcal{R}=\emptyset$ explains essentially many of the classical global regularity results, see the longer discussion in \cite{Ge}. For example the case of scalar Klein--Gordon equations corresponds to $\mathcal{R}_t=\emptyset$, in which case one can perform an integration by parts in $t$ (the normal form method \cite{Sh}).

More generally, one can sometimes adapt the integration by parts semilinear arguments even if the set $\mathcal{R}$ is nontrivial, provided that either the multiplier $m$ in \eqref{OPT} or the $\xi$ gradient $\nabla_\xi\Phi$ vanish suitably on this set. In the case of wave equations, the vanishing of $m$ corresponds precisely to Klainerman's ``null condition'' \cite{Kl2}. See also \cite{GeMaSh,GuNaTs,GeMaSh2,GuPa,IoPa,HaPuSh} for recent results exploiting these ideas.

However, it was observed by Germain \cite{Ge} that the case of Klein--Gordon systems with different speeds is genuinely different, even in the case of a system of two equations with equal masses $b_1=b_2$. In this case one cannot avoid the presence of large sets of space-time resonances and there are no natural ``null conditions''.  In general, the sets of space-time resonances take the form
\begin{equation*}
\mathcal{R}=\{(\xi,\eta)=(re,r'e):e\in\mathbb{S}^2\},
\end{equation*}
for certain values $r,r'\in (0,\infty)$ which depend on the parameters. In other words, the set $\mathcal{R}$ is a 2-dimensional manifold in $\mathbb{R}^3\times\mathbb{R}^3$, which should be thought of as the natural situation, in view of the fact that it is defined by four identities $\Phi(\xi,\eta)=\nabla_\eta\Phi(\xi,\eta)=0$. 

A partial result, which assumes certain separation conditions of the problematic frequencies, was obtained in \cite{Ge} in the semilinear case, and later extended to a quasilinear example in \cite{GeMa}. The results in \cite{Ge} and \cite{GeMa} appear to hold only for ``generic'' sets of parameters, and the required smallness of the perturbation depends implicitly on these parameters. 

Our analysis in this paper can be understood as a robust analysis of the case of non-degenerate resonances $\mathcal{R}\cap\mathcal{D}=\emptyset$, where $\mathcal{D}$ is the degenerate set
\begin{equation}\label{EquationAdded1}
\mathcal{D}=\{(\xi,\eta):\,\,\det[\nabla^2_{\eta,\eta}\Phi(\xi,\eta)]=0\}.
\end{equation}
The analysis seems to be limited to dimension three (and higher), and the method does not appear to extend easily to the two-dimensional case. It is possible, however, that this analysis can be developed further to allow for low-order degeneracy of the phase, thereby removing the condition on the parameters $b_\sigma$, $c_\sigma$ in \eqref{bc}. We note however, that this would require nontrivial change of the norms as it becomes likely that the gap in $xL^2$ integrability would increase between ``weak'' and ``strong'' norms. We note also that our conditions are sufficient to cover our main physical application.

Regarding the precise relations on the parameters in \eqref{bc}, the first condition ensures that $(0,0)$ is not time-resonant and thus this point plays little role. Note that $(0,0)$ is a specific point as all the gradients vanish there. The second condition only reflects a lack of uniformity of the estimates in terms of the gap between like parameters\footnote{As different velocities and masses approach each other, the corresponding spheres of ``space-time resonances'' go off to infinity, see \eqref{DefQ(XI)}. However a slightly more careful analysis would yield the wanted uniformity, at the expense of some clarity of the proof.}. Finally, the third condition is equivalent to asking that there are no degenerate space-time resonant points in $\mathbb{R}^3\times\mathbb{R}^3\setminus(0,0)$. We justify this at the end of this section.

The relevance of \eqref{EquationAdded1} can be illustrated by the fact that, after suitable manipulations and use of Morse lemma, the  study of operators like \eqref{OPT} can be reduced to the study of operators in standard form:
\begin{equation*}
\widehat{T[f,g]}(\xi)=\int_{\mathbb{R}}\int_{\mathbb{R}^3}e^{it\vert \eta-p(\xi)\vert^2}m(\xi,\eta)\widehat{f}(\xi-\eta,t)\widehat{g}(\eta,t) d\eta dt.
\end{equation*}
for some smooth function $p:\mathbb{R}^3\to\mathbb{R}^3$, which allows for precise estimate on the phase.

\subsubsection{Norms.} The choice of the $Z$-norms we use in the semilinear analysis, see Definition \ref{MainDef}, is very important. These norms have to satisfy at least two essential requirements:

(a) they {\it must} yield a $1/t$ decay after we apply the linear flow,

(b) they {\it must} allow for boundedness of the basic interaction bilinear operator \eqref{OPT}.

The simplest energy-type norm compatible with (a) corresponds to $x^{-(1+\varepsilon)}L^2(dx)$. This is, essentially, the ``strong norm'' $B^1_{k,j}$ in \eqref{sec5.2}
\footnote{We prefer, however, to first localize all our functions both in space and frequency. One should think of a function as composed of atoms,
\begin{equation*}
f=\sum_{k,j\in\mathbb{Z},\,k+j\geq 0}f_{k,j}=\sum_{k,j\in\mathbb{Z},\,k+j\geq 0}P_{[k-2,k+2]}(\phii_j^{(k)}\cdot P_kf),
\end{equation*}
where the atoms $f_{k,j}$ are localized essentially at frequency $\approx 2^k$ and distance $\approx 2^j$ from the origin in the physical space. Then we measure appropriately the size of each such atom, and use this to define the $Z$ norm of $f$. This point of view, which was used also in \cite{IoPa}, is convenient to deal with the main difficulty of the paper, namely estimating efficiently bilinear operators such as those in \eqref{OPT}.} and we are able to control most of the interactions in this norm. Unfortunately, certain interactions, corresponding to space-time resonances, are simply not bounded in this norm, even for inputs $f,g$ which are small smooth bump functions of scale $1$. This forces us to add another component to our space, measured in the ``weak-norm'' which has insufficient $xL^2$ integrability. This corresponds to $B^2_{k,j}$ in \eqref{sec5.2}. Fortunately, these only happen on an exceptional set of frequencies and the ``weak norm'' has an additional component that captures the essential two-dimensional nature of the support of these solutions. This smallness on the support then more than compensates for the weaker integrability and yields the all-important $1/t$ decay.

In addition, although fundamental, the gap in $L^2$-integrability between weak and strong norms is sufficiently small to allow us to treat the two norms similarly for most of the easier cases, thereby keeping the computations manageable.

\subsubsection{Condition on the parameters.}

We finish this section with simple computations showing that the condition \eqref{bc} implies the absence of degenerate space-time resonances, i.e. $\mathcal{R}\cap\mathcal{D}=\emptyset$. Let
\begin{equation*}
\begin{split}
&\Lambda_\sigma(\xi)=\sqrt{b_\sigma^2+c_\sigma^2\vert\xi\vert^2},\quad\Lambda_\mu(\xi)=\sqrt{b_\mu^2+c_\mu^2\vert\xi\vert^2},\quad\Lambda_\nu(\xi)=\sqrt{b_\nu^2+c_\nu^2\vert\xi\vert^2},\\
&\Phi(\xi,\eta)=\Lambda_\sigma(\xi)-\eps_1\Lambda_\mu(\xi-\eta)-\eps_1\epsilon\Lambda_\nu(\eta),\qquad\eps_1,\epsilon\in\{-1,1\}.
\end{split}
\end{equation*}
Clearly, $\Phi(0,0)=b_\sigma\pm b_\mu\pm b_\nu$ and therefore, the first equation in \eqref{bc} forces $(\xi,\eta)=(0,0)$ to not be time-resonant. Moreover, clearly any point of the form $(\xi,\eta)=(\xi,0)$, $\xi\in\mathbb{R}^3\setminus\{0\}$, cannot be space-resonant.

We show now that $(\xi,\eta)$ cannot be a degenerate space-time resonant point, provided that \eqref{bc} holds and $\eta\neq 0$. We may assume, without loss of generality, that 
\begin{equation}\label{bnbn0}
c_\mu\geq c_\nu\qquad \text{ and }\qquad b_\nu c_\mu^2\geq b_\mu c_\nu^2. 
\end{equation}
The relation
\begin{equation*}
(\nabla_\eta\Phi)(\xi,\eta)=0
\end{equation*}
is satisfied if and only if $\xi=q(\eta)$, where
\begin{equation}\label{DefQ(XI)}
q(\eta)=\Big[1+\epsilon\frac{b_\mu c_\nu^2}{(b_\nu^2c_\mu^4+c_\mu^4c_\nu^2|\eta|^2-c_\nu^4c_\mu^2|\eta|^2)^{1/2}}\Big]\eta.
\end{equation}

Clearly, $r=\vert q(\eta)\vert$ depends only on $s=\vert\eta\vert$ and
\begin{equation}\label{bnbn}
\frac{dr}{ds}=1+\epsilon\frac{ b_\mu c_\nu^2b_\nu^2c_\mu^4}{(b_\nu^2c_\mu^4+c_\mu^4c_\nu^2s^2-c_\nu^4c_\mu^2s^2)^{3/2}}.
\end{equation}
We claim now that 
\begin{equation}\label{bnbn2}
\frac{dr}{ds}>0\qquad \text{ if }s\in(0,\infty)\text{ and }(q(\eta),\eta)\in\mathcal{R}_t.
\end{equation}
Indeed, this is clear from \eqref{bnbn} if $\eps=1$ or if $\eps=-1$ and either $b_\nu c_\mu^2> b_\mu c_\nu^2$ or $c_\mu>c_\nu$. In the remaining case $\epsilon=-1$, $c_\mu=c_\nu$, $b_\mu=b_\nu$, we have $q(\eta)=0$, so $\Phi(q(\eta),\eta)=\Lambda_\sigma(0)\neq 0$, therefore $(q(\eta),\eta)\notin\mathcal{R}_t$. The conclusion \eqref{bnbn2} follows.

Finally, we show that 
\begin{equation}\label{bnbn3}
\det[(\nabla^2_{\eta,\eta}\Phi)(q(\eta),\eta)]\neq 0\qquad \text{ if }\eta\in\mathbb{R}^3\setminus\{0\}\text{ and }(q(\eta),\eta)\in\mathcal{R}_t.
\end{equation}
Letting $\Xi(\xi,\eta):=(\nabla_\eta\Phi)(\xi,\eta)$, we start from the defining identity $\Xi(q(\eta),\eta)=0$ and differentiate with respect to $\eta$. It follows that
\begin{equation*}
\frac{d\Xi}{d\eta}(q(\eta),\eta)=-\frac{d\Xi}{d\xi}(q(\eta),\eta)\cdot \frac{dq}{d\eta}(\eta).
\end{equation*}
It follows from \eqref{DefQ(XI)} and \eqref{bnbn2} that $\det(\partial q/\partial\eta)\neq 0$. Moreover, using the definition, $\det(\partial\Xi/\partial\xi)=\det(\nabla^2_{\eta,\xi}\Phi)\neq 0$, and the conclusion \eqref{bnbn3} follows.
\medskip

The rest of the paper is organized as follows: in Section \ref{MainSec1}, we prove Theorem \ref{MainThm1} and Theorem \ref{MainThm2} relying on a decay assumption. This is then proved in Section \ref{ContProof} and Section \ref{normproof} where we prove respectively the continuity of the $Z$-norm that captures the decay and a bootstrap result that gives global control of this norm assuming global bounds on high-order energy. Finally, in Section \ref{technical}, we provide some needed technical estimates and we study the relevant sets associated to our phases.

\section{Reductions and proofs of the main theorems}\label{MainSec}

\subsection{Local existence results}\label{MainSec1}

In this subsection we state and prove suitable local regularity results for our equations.

We start with quasilinear systems of Klein--Gordon equations. For $\sigma\in\{1,\ldots,d\}$ assume that $b_\sigma,c_\sigma\in[1/A,A]$ and $F_\sigma$ are nonlinearities as in \eqref{KGNon1}--\eqref{CondForEE}. For $N\geq 4$ and $u\in C([0,T]:H^N_r)\cap C^1([0,T]:H^{N-1}_r)$ we define the higher order energies
\begin{equation}\label{abc1}
\begin{split}
\mathcal{E}_N^{\rm{KG}}(t):=\sum_{|\rho|\leq N-1}\Big\{&\int_{\mathbb{R}^3}\sum_{\sigma=1}^d\Big[(\partial_tD^\rho_xu_{\sigma})^2+b_{\sigma}^2(D^\rho_xu_{\sigma})^2+\sum_{j=1}^3c_{\sigma}^2(\partial_jD^\rho_xu_{\sigma})^2\Big]\,dx\\
+&\int_{\mathbb{R}^3}\sum_{\mu,\nu=1}^d\sum_{j,k=1}^3G_{\mu\nu}^{jk}(u,\nabla_{x,t}u)\partial_jD^\rho_xu_\mu\partial_kD^\rho_xu_\nu\,dx\Big\}.
\end{split}
\end{equation}
The following proposition is our first local regularity result:

\begin{proposition}\label{Localexistence0}
(i) There is $\delta_0>0$ such that if 
\begin{equation}\label{abc2}
\|v_0\|_{H^4_r}+\|v_1\|_{H^3_r}\leq\delta_0
\end{equation}
then there is a unique solution $u=(u_1,\ldots,u_d)\in C([0,1]:H^4_r)\cap C^1([0,1]:H^3_r)$ of the system
\begin{equation}\label{abc3}
(\partial_t^2-c^2_\mu\Delta+b_\mu^2)u_\mu=F_\mu,\qquad\qquad\mu=1,\ldots,d,
\end{equation}
with $(u(0),\dot{u}(0))=(v_0,v_1)$. Moreover,
\begin{equation*}
\sup_{t\in[0,1]}\|u(t)\|_{H_r^4}+\sup_{t\in[0,1]}\|\dot{u}(t)\|_{H_r^3}\lesssim \|v_0\|_{H^4_r}+\|v_1\|_{H^3_r}.
\end{equation*}

(ii) If $N\geq 4$ and $(v_0,v_1)\in H^{N}_r\times H^{N-1}_r$ satisfies \eqref{abc2}, then $u\in C([0,1]:H^N_r)\cap C^1([0,1]:H^{N-1}_r)$, and
\begin{equation}\label{abc4}
\mathcal{E}^{\rm{KG}}_N(t')-\mathcal{E}^{\rm{KG}}_N(t)\lesssim \int_{t}^{t'}\mathcal{E}^{\rm{KG}}_N(s)\cdot \big[\sum_{|\rho|\leq 2}\|D^\rho_xu(s)\|_{L^\infty}+\sum_{|\rho|\leq 1}\|D^\rho_x\dot{u}(s)\|_{L^\infty}\big]\,ds,
\end{equation}
for any $t\leq t'\in[0,1]$.
\end{proposition}

We remark that the non-resonance condition \eqref{bc} is not needed in this local regularity result. On the other hand, the symmetry conditions \eqref{CondForEE} on the quasilinear components of the nonlinearities are important. 

\begin{proof}[Proof of Proposition \ref{Localexistence0}] The local existence claim in part (i) and the propagation of regularity claim in part (ii) are standard consequences of the general local existence theory of quasilinear symmetric hyperbolic systems, see Theorem II and Theorem III in \cite{Ka}. To prove the estimate \eqref{abc4}, we use the equations \eqref{abc3} and the definitions to estimate
\begin{equation}\label{abc7}
\begin{split}
\Big|\frac{d}{dt}\mathcal{E}_N^{\rm{KG}}(t)\Big|&\leq \sum_{|\rho|\leq N-1}\Big|\int_{\mathbb{R}^3}\sum_{\sigma=1}^d2(\partial_tD^\rho_xu_{\sigma})\cdot D^\rho_xF_\sigma\,dx+\int_{\mathbb{R}^3}\sum_{\sigma,\nu=1}^d\sum_{j,k=1}^32G_{\sigma\nu}^{jk}\cdot \partial_t\partial_jD^\rho_xu_\sigma\cdot \partial_kD^\rho_xu_\nu\,dx\Big|\\
&+\sum_{|\rho|\leq N-1}\Big|\int_{\mathbb{R}^3}\sum_{\mu,\nu=1}^d\sum_{j,k=1}^3\partial_tG_{\mu\nu}^{jk}\cdot \partial_jD^\rho_xu_\mu\cdot \partial_kD^\rho_xu_\nu\,dx\Big|.
\end{split}
\end{equation}

We will use the standard bound
\begin{equation}\label{sobo}
\|D^\rho_xf\cdot D^{\rho'}_xg\|_{L^2}\lesssim \|\nabla_x f\|_{L^\infty}\|g\|_{H^M}+\|\nabla_x g\|_{L^\infty}\|f\|_{H^M},
\end{equation}
provided that $|\rho|+|\rho'|\leq M+1$, $M\geq 1$, and $|\rho|,|\rho'|\geq 1$. For any multi-index $\rho$ with $|\rho|\leq N-1$ we estimate, as long as $\|u\|_{H^4}+\|\dot{u}\|_{H^3}\leq 1$,
\begin{equation*}
\Big|\int_{\mathbb{R}^3}\sum_{\mu,\nu=1}^d\sum_{j,k=1}^3\partial_tG_{\mu\nu}^{jk}\cdot \partial_jD^\rho_xu_\mu\cdot \partial_kD^\rho_xu_\nu\,dx\Big|\lesssim \|u\|_{H^N}^2\cdot \big[\sum_{|\alpha|\leq 2}\|D^\alpha_xu\|_{L^\infty}+\sum_{|\alpha|\leq 1}\|D^\alpha_x\dot{u}\|_{L^\infty}\big],
\end{equation*}
and, using also \eqref{sobo},
\begin{equation*}
\Big|\int_{\mathbb{R}^3}\sum_{\sigma=1}^d2(\partial_tD^\rho_xu_{\sigma})\cdot D^\rho_xQ_\sigma\,dx\Big|\lesssim \big[\|u\|_{H^N}^2+\|\dot{u}\|_{H^{N-1}}^2\big]\cdot \big[\sum_{|\alpha|\leq 2}\|D^\alpha_xu\|_{L^\infty}+\sum_{|\alpha|\leq 1}\|D^\alpha_x\dot{u}\|_{L^\infty}\big].
\end{equation*}
Moreover, for any $j,k\in\{1,2,3\}$ and $\sigma,\nu\in\{1,\ldots,d\}$ we estimate, using \eqref{sobo},
\begin{equation*}
\begin{split}
\Big|\int_{\mathbb{R}^3}2\partial_tD^\rho_xu_{\sigma}&\cdot [D^\rho_x(G^{jk}_{\sigma\nu}\cdot\partial_j\partial_ku_\nu)-G^{jk}_{\sigma\nu}\cdot D^\rho_x\partial_j\partial_ku_\nu]\,dx\Big|\\
&\lesssim \big[\|u\|_{H^N}^2+\|\dot{u}\|_{H^{N-1}}^2\big]\cdot \big[\sum_{|\alpha|\leq 2}\|D^\alpha_xu\|_{L^\infty}+\sum_{|\alpha|\leq 1}\|D^\alpha_x\dot{u}\|_{L^\infty}\big],
\end{split}
\end{equation*}
and
\begin{equation*}
\begin{split}
\Big|\int_{\mathbb{R}^3}2\partial_tD^\rho_xu_{\sigma}&\cdot G^{jk}_{\sigma\nu}\cdot D^\rho_x\partial_j\partial_ku_\nu\,dx+\int_{\mathbb{R}^3}2G_{\sigma\nu}^{jk}\cdot \partial_t\partial_jD^\rho_xu_\sigma\cdot \partial_kD^\rho_xu_\nu\,dx\Big|\\
&\lesssim \big[\|u\|_{H^N}^2+\|\dot{u}\|_{H^{N-1}}^2\big]\cdot \big[\sum_{|\alpha|\leq 2}\|D^\alpha_xu\|_{L^\infty}+\sum_{|\alpha|\leq 1}\|D^\alpha_x\dot{u}\|_{L^\infty}\big].
\end{split}
\end{equation*}
Therefore, using \eqref{abc7}, 
\begin{equation*}
\begin{split}
\Big|\frac{d}{dt}\mathcal{E}_N^{\rm{KG}}(t)\Big|\lesssim \big[\|u(t)\|_{H^N}^2+\|\dot{u}(t)\|_{H^{N-1}}^2\big]\cdot \big[\sum_{|\alpha|\leq 2}\|D^\alpha_xu(u)\|_{L^\infty}+\sum_{|\alpha|\leq 1}\|D^\alpha_x\dot{u}(t)\|_{L^\infty}\big],
\end{split}
\end{equation*}
for any $t\in[0,1]$. We notice that $\|u(t)\|_{H^N}^2+\|\dot{u}(t)\|_{H^{N-1}}^2\approx \mathcal{E}_N^{\rm{KG}}(t)$, provided that $\|u\|_{H^4}+\|\dot{u}\|_{H^3}\ll 1$. The desired estimate \eqref{abc4} follows.
\end{proof}

We consider now the Euler--Maxwell system. Recalling the definition \eqref{basicspace}, for any $(n,v,E,B)\in\widetilde{H}^N$ we define
\begin{equation}\label{pla0}
\mathcal{E}_N:=\sum_{|\rho|\leq N}\int_{\mathbb{R}^3}\big[T|D_x^\rho n|^2+(1+n)|D_x^\rho v|^2+|D_x^\rho E|^2+c^2|D_x^\rho B|^2\big]\,dx,
\end{equation}
and
\begin{equation}\label{pla0.5}
\|(n,v,E,B)\|_{Z'}:=\|\nabla n\|_{L^\infty}+\|v\|_{L^\infty}+\|\nabla v\|_{L^\infty}+\|\nabla E\|_{L^\infty}+\Vert B\Vert_{L^\infty}+\|\nabla B\|_{L^\infty}.
\end{equation}
The following proposition is our second local regularity result:

\begin{proposition}\label{Localexistence}
(i) There is $\delta_0>0$ such that if 
\begin{equation}\label{pla1}
\|(n_0,v_0,E_0,B_0)\|_{\widetilde{H}^4}\leq\delta_0
\end{equation}
then there is a unique solution $(n,v,E,B)\in C([0,1]:\widetilde{H}^4)$ of the system
\begin{equation}\label{EM4}
\begin{split}
&\partial_tn+\hbox{div}((1+n)v)=0,\\
&\partial_tv+v\cdot\nabla v+T\nabla n+E+v\times B=0,\\
&\partial_tB+\nabla\times E=0,\\
&\partial_tE-c^2\nabla\times B-(1+n)v=0,\\
\end{split}
\end{equation}
with $(n(0),v(0),E(0),B(0))=(n_0,v_0,E_0,B_0)$. Moreover,
\begin{equation*}
\sup_{t\in[0,1]}\|(n(t),v(t),E(t),B(t))\|_{\widetilde{H}^4}\lesssim \|(n_0,v_0,E_0,B_0)\|_{\widetilde{H}^4}.
\end{equation*}

(ii) If $N\geq 4$ and $(n_0,v_0,E_0,B_0)\in\widetilde{H}^N$ satisfies \eqref{pla1}, then $(n,v,E,B)\in C([0,1]:\widetilde{H}^N)$, and
\begin{equation}\label{pla2}
\mathcal{E}_N(t')-\mathcal{E}_N(t)\lesssim \int_{t}^{t'}\mathcal{E}_N(s)\cdot \|(n,v,E,B)(s)\|_{Z'}\,ds.
\end{equation}
for any $t\leq t'\in[0,1]$.

(iii) If $(n_0,v_0,E_0,B_0)\in\widetilde{H}^4$ satisfies \eqref{pla1}, and, in addition,
\begin{equation*}
\hbox{div}(E_0)+n_0=0,\qquad B_0-\nabla\times v_0=0,
\end{equation*}
then, for any $t\in[0,1]$,
\begin{equation}\label{pla5}
\hbox{div}(E)(t)+n(t)=0,\qquad B(t)-(\nabla\times v)(t)=0.
\end{equation}
\end{proposition}

\begin{proof}[Proof of Proposition \ref{Localexistence}] We multiply each equation by a suitable factor and rewrite the system \eqref{EM4} as a symmetric hyperbolic system,
\begin{equation*}
\begin{split}
&T\partial_tn+T\sum_{k=1}^3v_k\partial_kn+T(1+n)\sum_{k=1}^3\partial_kv_k=0,\\
&(1+n)\partial_tv_j+T(1+n)\partial_jn+(1+n)\sum_{k=1}^3v_k\partial_kv_j=-(1+n)E_j-(1+n)\sum_{k,m=1}^3\in_{jmk}v_mB_k,\\
&c^2\partial_tB_j+c^2\sum_{k,m=1}^3\in_{jmk}\partial_mE_k=0,\\
&\partial_tE_j-c^2\sum_{k,m=1}^3\in_{jmk}\partial_mB_k=(1+n)v_j.
\end{split}
\end{equation*}
Then we apply Theorem II and Theorem III in \cite{Ka} to prove the local existence claim in part (i) and the propagation of regularity claim in part (ii). 

To verify the energy inequality \eqref{pla2} we let, for $P=D^\rho_x$, $|\rho|\leq N$,
\begin{equation*}
\mathcal{E}'_P:=\int_{\mathbb{R}^3}\big[T|P n|^2+(1+n)|P v|^2+|P E|^2+c^2|P B|^2\big]\,dx,
\end{equation*}
Then we calculate
\begin{equation*}
\begin{split}
&\frac{d}{dt}\mathcal{E}'_P=I_P+II_P+III_P+IV_P,\\
&I_P:=\int_{\mathbb{R}^3}2TP n\cdot P\partial_tn\,dx,\\
&II_P:=\sum_{j=1}^3\int_{\mathbb{R}^3}\partial_tn\cdot Pv_j\cdot Pv_j\,dx,\\
&III_P:=\sum_{j=1}^3\int_{\mathbb{R}^3}2(1+n)\cdot Pv_j\cdot P\partial_tv_j\,dx,\\
&IV_P:=\sum_{j=1}^3\int_{\mathbb{R}^3}2P E_j\cdot P\partial_tE_j\,dx+\sum_{j=1}^3\int_{\mathbb{R}^3}2c^2PB_j\cdot P\partial_tB_j\,dx.
\end{split}
\end{equation*}
Then we estimate, using the equations and the general bound \eqref{sobo},
\begin{equation*}
\Big|I_P+2T\sum_{k=1}^3\int_{\mathbb{R}^3}Pn\cdot (1+n)\cdot P\partial_kv_k\,dx\Big|\lesssim \|(n,v,E,B)\|_{\widetilde{H}^N}^2\cdot \|(n,v,E,B)\|_{Z'},
\end{equation*}
\begin{equation*}
\Big|II_P\Big|\lesssim \|(n,v,E,B)\|_{\widetilde{H}^N}^2\cdot \|(n,v,E,B)\|_{Z'},
\end{equation*}
\begin{equation*}
\Big|III_P+2T\sum_{j=1}^3\int_{\mathbb{R}^3}P\partial_jn\cdot (1+n)\cdot Pv_j\,dx+2\sum_{j=1}^3\int_{\mathbb{R}^3}PE_j\cdot Pv_j\cdot (1+n)\,dx\Big|\lesssim \|(n,v,E,B)\|_{\widetilde{H}^N}^2\cdot \|(n,v,E,B)\|_{Z'},
\end{equation*}
\begin{equation*}
\Big|IV_P-2\sum_{j=1}^3\int_{\mathbb{R}^3}PE_j\cdot Pv_j\cdot (1+n)\,dx\Big|\lesssim \|(n,v,E,B)\|_{\widetilde{H}^N}^2\cdot \|(n,v,E,B)\|_{Z'}.
\end{equation*}
Therefore
\begin{equation*}
\Big|\frac{d}{dt}\mathcal{E}'_P\Big|\lesssim \|(n,v,E,B)\|_{\widetilde{H}^N}^2\cdot \|(n,v,E,B)\|_{Z'},
\end{equation*}
and the bound \eqref{pla2} follows since $\mathcal{E}_N=\sum_{P=D^\rho_x,\,|\rho|\leq N}\mathcal{E}'_P\approx \|(n,v,E,B)\|_{\widetilde{H}^N}^2$.

Finally, to verify that the identities \eqref{pla5} are propagated by the flow, we let
\begin{equation*}
X:=n+\hbox{div}(E),\qquad Y:=B-\nabla\times v.
\end{equation*}
Using the equations in \eqref{EM4} we calculate
\begin{equation*}
\partial_tX=\partial_tn+\sum_{j=1}^3\partial_j\partial_tE_j=-\sum_{j=1}^3\partial_j[(1+n)v_j]+\sum_{j=1}^3\partial_j[(1+n)v_j]=0,
\end{equation*}
therefore $X\equiv 0$. Moreover
\begin{equation*}
\partial_t\big(\sum_{k=1}^3\partial_kB_k\big)=0,
\end{equation*}
therefore
\begin{equation*}
\sum_{k=1}^3\partial_kB_k\equiv 0,\qquad \sum_{k=1}^3\partial_kY_k\equiv 0.
\end{equation*}
In addition, for any $m,n\in\{1,2,3\}$,
\begin{equation*}
\partial_mv_n-\partial_nv_m=\sum_{j=1}^3\in_{jmn}(B_j-Y_j).
\end{equation*}
Finally we calculate, for $i\in\{1,2,3\}$,
\begin{equation*}
\begin{split}
\partial_tY_i&=\partial_tB_i-\sum_{j,k=1}^3\in_{ijk}\partial_j\partial_tv_k\\
&=-\sum_{j,k=1}^3\in_{ijk}\partial_jE_k+\sum_{j,k=1}^3\in_{ijk}\partial_j\big[T\partial_kn+E_k+\sum_{l=1}^3v_l\partial_lv_k+\sum_{l,m=1}^3\in_{klm}v_lB_m\big]\\
&=\sum_{j,k,l=1}^3\in_{ijk}(\partial_jv_l\partial_lv_k+v_l\partial_j\partial_lv_k)+\sum_{j,k,l,m=1}^3\in_{ijk}\in_{klm}\partial_j(v_lB_m)\\
&=\sum_{j,k,l=1}^3\in_{ijk}\partial_jv_l(\partial_lv_k-\partial_kv_l)+\sum_{j,k,l=1}^3\in_{ijk}v_l\partial_l\partial_jv_k+\sum_{j,l,m=1}^3(\delta_{il}\delta_{jm}-\delta_{jl}\delta_{im})\partial_j(v_lB_m)\\
&=\sum_{l=1}^3\big[(B_i-Y_i)\partial_lv_l-\partial_lv_i(B_l-Y_l)+v_l\partial_l(B_i-Y_i)\big]+\sum_{j=1}^3\big[B_j\partial_jv_i +v_i\partial_jB_j-B_i\partial_jv_j-v_j\partial_jB_i\big]\\
&=\sum_{l=1}^3\big[-Y_i\partial_lv_l+Y_l\partial_lv_i-v_l\partial_lY_i\big].
\end{split}
\end{equation*}
Therefore, using energy estimates, $Y\equiv 0$ as desired.
\end{proof}

\subsection{Definitions, function spaces, and the main propositions}\label{MainSec2}

We fix $\varphi:\mathbb{R}\to[0,1]$ an even smooth function supported in $[-8/5,8/5]$ and equal to $1$ in $[-5/4,5/4]$. Let
\begin{equation*}
\varphi_k(x):=\varphi(|x|/2^k)-\varphi(|x|/2^{k-1})\qquad\text{ for any }k\in\mathbb{Z},\,x\in\mathbb{R}^3,\qquad \varphi_I:=\sum_{m\in I\cap\mathbb{Z}}\varphi_m\text{ for any }I\subseteq\mathbb{R}.
\end{equation*}
Let
\begin{equation*}
\mathcal{J}:=\{(k,j)\in\mathbb{Z}\times\mathbb{Z}_+:\,k+j\geq 0\}.
\end{equation*}
For any $(k,j)\in\mathcal{J}$ let
\begin{equation*}
\phii^{(k)}_j(x):=
\begin{cases}
\varphi_{(-\infty,-k]}(x)\quad&\text{ if }k+j=0\text{ and }k\leq 0,\\
\varphi_{(-\infty,0]}(x)\quad&\text{ if }j=0\text{ and }k\geq 0,\\
\varphi_j(x)\quad&\text{ if }k+j\geq 1\text{ and }j\geq 1.
\end{cases}
\end{equation*}
and notice that, for any $k\in\mathbb{Z}$ fixed,
\begin{equation*}
\sum_{j\geq-\min(k,0)}\phii^{(k)}_j=1.
\end{equation*}
For any interval $I\subseteq\mathbb{R}$ let
\begin{equation*}
\phii^{(k)}_I(x):=\sum_{j\in I,\,(k,j)\in\mathcal{J}}\phii^{(k)}_j(x).
\end{equation*}

Let $P_k$, $k\in\mathbb{Z}$, denote the operator on $\mathbb{R}^3$ defined by the Fourier multiplier $\xi\to \varphi_k(\xi)$. Similarly, for any $I\subseteq \mathbb{R}$ let $P_I$ denote the operator on $\mathbb{R}^3$
 defined by the Fourier multiplier $\xi\to \varphi_I(\xi)$. For any $k\in\mathbb{Z}$ let
\begin{equation}\label{sets}
\begin{split}
&\mathcal{X}_k^1:=\{(k_1,k_2)\in\mathbb{Z}\times\mathbb{Z}:|\max(k_1,k_2)-k|\leq 8\},\\
&\mathcal{X}_k^2:=\{(k_1,k_2)\in\mathbb{Z}\times\mathbb{Z}:\max(k_1,k_2)-k\geq 8\text{ and }|k_1-k_2|\leq 8\},\\
&\mathcal{X}_k:=\mathcal{X}_k^1\cup \mathcal{X}_k^2.
\end{split}
\end{equation}

For integers $n\geq 1$ let
\begin{equation}\label{symb}
\mathcal{S}^{n}:=\{q:\mathbb{R}^3\to\mathbb{C}:\|q\|_{\mathcal{S}^{n}}:=\sup_{\xi\in\mathbb{R}^3\setminus\{0\}}\sup_{|\rho|\leq n}|\xi|^{|\rho|}|D^\rho_\xi q(\xi)|<\infty\},
\end{equation}
denote classes of symbols satisfying differential inequalities of the H\"{o}rmander--Michlin type. An operator $Q$ will be called a {\it{normalized Calderon--Zygmund operator}} if
\begin{equation}\label{CZ}
\widehat{Q f}(\xi)=q(\xi)\cdot \widehat{f}(\xi),\qquad \text{ for some }q\in\mathcal{S}^{100},\,\|q\|_{\mathcal{S}^{100}}\leq 1.
\end{equation}
For any integer $d'\geq 1$ let
\begin{equation}\label{symb2}
\begin{split}
&\mathcal{M}_{d'}:=\big\{m:\mathbb{R}^3\times\mathbb{R}^3\to\mathbb{C}:\,\,m(\xi,\eta)=\sum_{l=1}^{d'}m^l(\xi,\eta)\cdot q_1^l(\xi)\cdot q_2^l(\xi-\eta)\cdot q_3^l(\eta),\\
&\sup_{n\in\{1,2,3\}}\|q_n^l\|_{\mathcal{S}^{100}}\leq 1,\,\,m^l\in\{(1+|\xi|^2)^{1/2},\,(1+|\eta|^2)^{1/2},\,(1+|\xi-\eta|^2)^{1/2}\}\text{ for any }l=1,\ldots,d'\big\}.
\end{split}
\end{equation}

\begin{definition}\label{MainDef}
Let
\begin{equation}\label{sec5.6}
\beta:=1/100,\qquad \al:=\beta/2, \qquad \gamma:=11/8.
\end{equation}
We define
\begin{equation}\label{sec5}
Z:=\{f\in L^2(\mathbb{R}^3):\,\|f\|_{Z}:=\sup_{(k,j)\in\mathcal{J}}\|\phii^{(k)}_j(x)\cdot P_kf(x)\|_{B_{k,j}}<\infty\},
\end{equation}
where, with $\widetilde{k}:=\min(k,0)$ and $k_+:=\max(k,0)$,
\begin{equation}\label{sec5.2}
\|g\|_{B_{k,j}}:=\inf_{g=g_1+g_2}\big[\|g_1\|_{B^1_{k,j}}+\|g_2\|_{B^2_{k,j}}\big],
\end{equation}
\begin{equation}\label{sec5.3}
\|h\|_{B^1_{k,j}}:=(2^{\alpha k}+2^{10k})\big[2^{(1+\beta)j}\|h\|_{L^2}+2^{(1/2-\beta)\widetilde{k}}\|\widehat{h}\|_{L^\infty}\big],
\end{equation}
and
\begin{equation}\label{sec5.4}
\begin{split}
\|h\|_{B^2_{k,j}}:=(2^{\alpha k}+2^{10k})\big[&2^{-2\beta\widetilde{k}}2^{(1-\beta)j}\|h\|_{L^2}+2^{(1/2-\beta)\widetilde{k}}\|\widehat{h}\|_{L^\infty}\\
&+2^{(\gamma-\beta-1/2)\widetilde{k}}2^{2k_+}2^{\gamma j}\sup_{R\in[2^{-j},2^k],\,\xi_0\in\mathbb{R}^3}R^{-2}\|\widehat{h}\|_{L^1(B(\xi_0,R))}\big].
\end{split}
\end{equation}
\end{definition}

In order to properly understand the $Z$ norm, one should keep in mind that the $B^1_{k,j}$ is the easiest norm that one would want to use and in particular its $x$-integrability of the $L^2$-norm is sufficient to obtain the needed $1/t$ decay after we apply the linear flow. However, the $B^2_{k,j}$ is forced upon us by the presence of space-time resonances. It has slightly too weak decay, but this is compensated for by the last term that captures the two-dimensional property of the support.

The weak component $B^2_{k,j}$ is important only at middle frequencies $|k|\lesssim 1$, where one has the more friendly expression
\begin{equation}\label{sec5.40001}
\begin{split}
\|h\|_{B^1_{k,j}}&\approx 2^{(1+\beta)j}\|h\|_{L^2}+\|\widehat{h}\|_{L^\infty},\\
\|h\|_{B^2_{k,j}}&\approx 2^{(1-\beta)j}\|h\|_{L^2}+\|\widehat{h}\|_{L^\infty}+2^{\gamma j}\sup_{R\in[2^{-j},1],\,\xi_0\in\mathbb{R}^3}R^{-2}\|\widehat{h}\|_{L^1(B(\xi_0,R))}.
\end{split}
\end{equation}
One should think of $j$ as very large; the $B^2_{k,j}$ norm is relevant to measure functions that have thin, essentially $2$-dimensional Fourier support. 

Finally, the weights in $k$ in \eqref{sec5.3}-\eqref{sec5.4} are chosen so as to give \eqref{sec5.40001} when $k=0$ and so that, at the uncertainty principle $k+j=0$, all norms should be comparable for a bump function.

\medskip

The definition above shows that if $\|f\|_Z\leq 1$ then, for any $(k,j)\in\mathcal{J}$ one can decompose
\begin{equation}\label{sec5.8}
\widetilde{\varphi}^{(k)}_{j}\cdot P_kf=(2^{\alpha k}+2^{10k})^{-1}(g+h),
\end{equation}
where{\footnote{The support condition \eqref{sec5.81} can easily be achieved by starting with a decomposition 
$\widetilde{\varphi}^{(k)}_{j}\cdot P_kf=(2^{\alpha k}+2^{10k})^{-1}(g'+h')$ that minimizes the $B_{k,j}$ norm up to a constant, and then 
redefining $g:=g'\cdot \widetilde{\varphi}^{(k)}_{[j-1,j+1]}$ and $h:=h'\cdot \widetilde{\varphi}^{(k)}_{[j-1,j+1]}$, see the proof of Lemma \ref{tech3}.}}
\begin{equation}\label{sec5.81}
g=g\cdot \widetilde{\varphi}^{(k)}_{[j-2,j+2]},\qquad h=h\cdot \widetilde{\varphi}^{(k)}_{[j-2,j+2]},
\end{equation}
and
\begin{equation}\label{sec5.815}
\begin{split}
&2^{(1+\be)j}\|g\|_{L^2}+2^{(1/2-\beta)\widetilde{k}}\|\widehat{g}\|_{L^\infty}\lesssim 1,\\
&2^{-2\beta\widetilde{k}}2^{(1-\be)j}\|h\|_{L^2}+2^{(1/2-\beta)\widetilde{k}}\|\widehat{h}\|_{L^\infty}+
2^{(\gamma-\beta-1/2)\widetilde{k}}2^{2k_+}2^{\gamma j}\sup_{R\in[2^{-j},2^k],\,\xi_0\in\mathbb{R}^3}R^{-2}\|\widehat{h}\|_{L^1(B(\xi_0,R))}
\lesssim 1.
\end{split}
\end{equation}
In some of the easier estimates we will often use the weaker bound, obtained by setting $R=2^k$, 
\begin{equation}\label{sec5.82}
\begin{split}
&2^{(1+\be)j}\|g\|_{L^2}+2^{(1/2-\beta)\widetilde{k}}\|\widehat{g}\|_{L^\infty}\lesssim 1,\\
&2^{-2\beta\widetilde{k}}2^{(1-\be)j}\|h\|_{L^2}+2^{(1/2-\beta)\widetilde{k}}\|\widehat{h}\|_{L^\infty}
+2^{(\gamma-\beta-5/2)\widetilde{k}}2^{\gamma j}\|\widehat{h}\|_{L^1}\lesssim 1.
\end{split}
\end{equation}

As before, assume $A\geq 1$ is a (large number), $d\geq 1$ is a fixed integer, and $b_1,\ldots,b_d,c_1,\ldots,c_d\in(0,\infty)$ are positive real numbers with the properties
\begin{equation}\label{abcond}
b_1,\ldots,b_d,c_1,\ldots,c_d\in[1/A,A]
\end{equation}
and, see \eqref{bc},
\begin{equation}\label{abcond2}
\begin{split}
&|b_{\sigma_1}+b_{\sigma_2}-b_{\sigma_3}|\geq 1/A\qquad\qquad\qquad\qquad\qquad\quad\,\text{ for any }\sigma_1,\sigma_2,\sigma_3\in\{1,\ldots,d\},\\
&|c_{\sigma_1}-c_{\sigma_2}|,\,|b_{\sigma_1}-b_{\sigma_2}|\in\{0\}\cup[1/A,\infty)\qquad\qquad\text{ for any }\sigma_1,\sigma_2\in\{1,\ldots,d\},\\
&(c_{\sigma_1}-c_{\sigma_2})(c_{\sigma_1}^2b_{\sigma_2}-c_{\sigma_2}^2b_{\sigma_1})\geq 0\qquad\qquad\qquad\quad\,\,\text{ for any }\sigma_1,\sigma_2\in\{1,\ldots,d\}.
\end{split}
\end{equation}
Let $\Lambda_\sigma:\mathbb{R}^3\to[0,\infty)$, $\sigma=1,\ldots,d$,
\begin{equation}\label{lambdas}
\Lambda_\sigma(\xi):=(b_\sigma^2+c_\sigma^2|\xi|^2)^{1/2}.
\end{equation}
Let
\begin{equation}\label{Is}
\mathcal{I}_d:=\{(1+),\ldots,(d+),(1-),\ldots,(d-)\}.
\end{equation}
Assume $D=D(d,A,d')$ is a sufficiently large fixed constant.

Given $U=(U_1,\ldots,U_d)\in C([0,T]:H^N)$, for some $T\geq 1$ and $N\geq 4$, we are considering quadratic nonlinearities of the form
\begin{equation}\label{nonlinearity}
\widehat{\mathcal{N}_\sigma}(\xi,t)=\sum_{\mu,\nu\in\mathcal{I}_d}\int_{\mathbb{R}^3}m_{\sigma;\mu,\nu}(\xi,\eta)\widehat{U_\mu}(\xi-\eta,t)
\widehat{U_\nu}(\eta,t)\,d\eta,\qquad\sigma=1,\ldots,d,
\end{equation}
for symbols $m_{\sigma;\mu,\nu}\in\mathcal{M}_{d'}$, where $U_{\sigma+}:=U_{\sigma},U_{\sigma-}:=\overline{U_\sigma}$, $\sigma\in\{1,\ldots,d\}$.

We claim first that smooth solutions of suitable systems that start with data in the space $Z$ remain in the space $Z$, in a continuous way. More precisely:

\begin{proposition}\label{Norm0}
Assume $N_0=10^4$, $T_0\geq 1$, and $U=(U_1,\ldots,U_d)\in C([0,T_0]:H^{N_0})$ is a solution of the system of equations
\begin{equation}\label{maineq2}
(\partial_t+i\Lambda_\sigma)U_\sigma=\mathcal{N}_\sigma,\qquad\sigma=1,\ldots,d,
\end{equation}
where $\mathcal{N}_\sigma$ are defined as in \eqref{nonlinearity}. Assume that, for some $t_0\in[0,T_0]$,
\begin{equation}\label{data1}
e^{it_0\Lambda_\sigma}U_\sigma(t_0)\in Z,\qquad \qquad\sigma=1,\ldots,d.
\end{equation}
Then there is 
\begin{equation*}
\tau=\tau\Big(T_0,\sup_{\sigma\in\{1,\ldots,d\}}\|e^{it_0\Lambda_\sigma}U_\sigma(t_0)\|_Z,\sup_{\sigma\in\{1,\ldots,d\}}\sup_{t\in[0,T_0]}\|U_\sigma(t)\|_{H^{N_0}}\Big)>0
\end{equation*}
such that
\begin{equation}\label{data2}
\sup_{t\in [0,T_0]\cap[t_0,t_0+\tau]}\sup_{\sigma=1,\ldots,d}\|e^{it\Lambda_\sigma}U_\sigma(t)\|_{Z}\leq 2\sup_{\sigma\in\{1,\ldots,d\}}\|e^{it_0\Lambda_\sigma}U_\sigma(t_0)\|_Z,
\end{equation}
and the mapping $t\to e^{it\Lambda_\sigma}U_\sigma(t)$ is continuous from $[0,T_0]\cap[t_0,t_0+\tau]$ to $Z$, for any $\sigma\in\{1,\ldots,d\}$.
\end{proposition}

The key proposition is the following bootstrap estimate:

\begin{proposition}\label{Norm}
Assume $N_0=10^4$, $T_0\geq 0$, and $U=(U_1,\ldots,U_d)\in C([0,T_0]:H^{N_0})$ is a solution of the system of equations
\begin{equation}\label{maineq}
(\partial_t+i\Lambda_\sigma)U_\sigma=\mathcal{N}_\sigma,\qquad\sigma=1,\ldots,d,
\end{equation}
where $\mathcal{N}_\sigma$ are defined as in \eqref{nonlinearity} and the coefficients $b_\sigma, c_\sigma$ verify \eqref{abcond}--\eqref{abcond2}. Assume that
\begin{equation}\label{sec2}
\sup_{t\in [0,T_0]}\sup_{\sigma=1,\ldots,d}\|e^{it\Lambda_\sigma}U_\sigma(t)\|_{H^{N_0}\cap Z}\leq\delta_1\leq 1.
\end{equation}
Then
\begin{equation}\label{sec3}
\sup_{t\in [0,T_0]}\sup_{\sigma=1,\ldots,d}\|e^{it\Lambda_\sigma}U_\sigma(t)-U_\sigma(0)\|_{Z}\lesssim \delta_1^2,
\end{equation}
where the implicit constant in \eqref{sec3} may depend only on the constants $A$, $d$, and $d'$.
\end{proposition}

We prove the easier Proposition \ref{Norm0} in section \ref{ContProof} and we prove the harder Proposition \ref{Norm} in sections \ref{normproof} and \ref{technical}. In the rest of this section we show how to use these propositions and the local theory to complete the proofs of Theorem \ref{MainThm1} and Theorem \ref{MainThm2}.

\subsection{Proof of Theorem \ref{MainThm1}}\label{MainSec4} We prove now Theorem \ref{MainThm1}, as a consequence of Proposition \ref{Localexistence0}, Proposition \ref{Norm0}, and Proposition \ref{Norm}. Indeed, assume that we start with data $(v_0,v_1)$ as in \eqref{maincond1}, where $\overline{\varepsilon}$ is taken sufficiently small. Using Proposition \ref{Localexistence0} there is $T_1\geq 1$ and a unique solution $u\in C([0,T_1]:H^{N_0+1}_r)\cap C^1([0,T_1]:H^{N_0}_r)$ of the system \eqref{abc3}, with
\begin{equation}\label{kim1}
\sup_{t\in[0,T_1]}\|u(t)\|_{H^{N_0+1}_r}+\sup_{t\in[0,T_1]}\|\dot{u}(t)\|_{H^{N_0}_r}\leq\varepsilon_0^{3/4}.
\end{equation}
For $\sigma\in \{1,\ldots,d\}$ let
\begin{equation}\label{kim1.5}
U_\sigma(t):=\dot{u}_\sigma(t)-i\Lambda_\sigma u_\sigma,
\end{equation}
where, as in \eqref{lambdas}, $\Lambda_\sigma=(b_\sigma^2-c_\sigma^2\Delta)^{1/2}$. Then $U_\sigma\in C([0,T_1]:H^{N_0})$ for any $\sigma\in\{1,\ldots,d\}$, and 
\begin{equation}\label{kim2}
u_\sigma=-\Lambda_\sigma^{-1}\Im U_\sigma,\qquad \dot{u}_\sigma=\Re U_\sigma.
\end{equation}
Using these definitions we calculate
\begin{equation*}
(\partial_t+i\Lambda_\sigma)U_\sigma=(\partial_t^2+b_\sigma^2-c_\sigma^2\Delta)u_\sigma=\sum_{j,k=1}^3\sum_{\nu=1}^dG_{\sigma\nu}^{jk}(u,\nabla_{x,t} u)\partial_j\partial_ku_\nu+Q_\sigma(u,\nabla_{x,t} u),
\end{equation*}
see \eqref{KGNon1}. Using the formulas in \eqref{kim2}, it is easy to see that this is a system of the form
\begin{equation*}
(\partial_t+i\Lambda_\sigma)U_\sigma=\mathcal{N}_\sigma,\qquad \sigma\in\{1,\ldots,d\},
\end{equation*}
where the nonlinearities $\mathcal{N}_\sigma$ can be expressed in terms of the functions $U_\sigma$ as in \eqref{nonlinearity}. Therefore we can apply the results in Proposition \ref{Norm0} and Proposition \eqref{Norm}. 

Using the definition \eqref{kim1.5} and Lemma \ref{tech3}, it follows that $U\in C([0,T_1]:H^{N_0})$ and
\begin{equation}\label{kim6}
\sup_{t\in[0,T_1]}\|U(t)\|_{H^{N_0}}\lesssim \varepsilon_0^{3/4},\qquad\sup_{\sigma\in\{1,\ldots,d\}}\|U_\sigma(0)\|_{Z}\lesssim \varepsilon_0.
\end{equation}
Let $T_2$ denote the largest number in $(0,T_1]$ with the property that
\begin{equation*}
\sup_{t\in[0,T_2)}\sup_{\sigma\in\{1,\ldots,d\}}\|e^{it\Lambda_\sigma}U_\sigma(t)\|_{Z}\leq \varepsilon_0^{3/4}.
\end{equation*}
Such a $T_2\in(0,T_1]$ exists, in view of \eqref{kim6} and Proposition \ref{Norm0}. We apply now Proposition \ref{Norm} on the intervals $[0,T_2(1-1/n)]$, $n=2,3,\ldots$, with $\delta_1\approx \varepsilon_0^{3/4}$. It follows that 
\begin{equation*}
\sup_{t\in[0,T_2)}\sup_{\sigma\in\{1,\ldots,d\}}\|e^{it\Lambda_\sigma}U_\sigma(t)\|_{Z}\lesssim \varepsilon_0.
\end{equation*}
Using again Proposition \ref{Norm0} it follows that $T_2=T_1$ and 
\begin{equation}\label{kim7}
\sup_{t\in[0,T_1]}\sup_{\sigma\in\{1,\ldots,d\}}\|e^{it\Lambda_\sigma}U_\sigma(t)\|_{Z}\lesssim \varepsilon_0.
\end{equation}

Using the formulas in \eqref{kim2}, and the bounds \eqref{kim7} and \eqref{ok10} it follows that
\begin{equation}\label{kim8}
\sup_{t\in[0,T_1]}\big[(1+t)^{1+\beta}\big(\sup_{|\rho|\leq 4}\|D^\rho_x u(t)\|_{L^\infty}+\sup_{|\rho|\leq 3}\|D^\rho_x \dot{u}(t)\|_{L^\infty}\big)\big]\lesssim \varepsilon_0.
\end{equation}
Therefore, using the energy estimate \eqref{abc4}, it follows that
\begin{equation*}
\sup_{t\in[0,T_1]}\mathcal{E}^{\mathrm{KG}}_{N_0+1}(t)\lesssim \varepsilon_0.
\end{equation*}
As a consequence, if the solution $u$ satisfies the bound \eqref{kim1} on some interval $[0,T_1]$, then it has to satisfy the stronger bound
\begin{equation*}
\sup_{t\in[0,T_1]}\|u(t)\|_{H^{N_0+1}_r}+\sup_{t\in[0,T_1]}\|\dot{u}(t)\|_{H^{N_0}_r}\lesssim \varepsilon_0.
\end{equation*}
Therefore the solution can be extended globally, and the desired bound \eqref{mainconcl1.1} follows using also \eqref{kim8}. This completes the proof of Theorem \ref{MainThm1}.

\subsection{Proof of Theorem \ref{MainThm2}}\label{MainSec5} As before, Theorem \ref{MainThm2} is a consequence of Proposition \ref{Localexistence}, Proposition \ref{Norm0}, and Proposition \ref{Norm}. Indeed, assume that we start with data $(n_0,v_0,E_0,B_0)$ as in \eqref{maincond2}, where $\overline{\varepsilon}$ is taken sufficiently small. Using Proposition \ref{Localexistence} there is $T_1\geq 1$ and a unique solution $(n,v,E,B)\in C([0,T_1]:\widetilde{H}^{N_0+1})$ of the system \eqref{EM4}, with $(n(0),v(0),E(0),B(0))=(n_0,v_0,E_0,B_0)$,
\begin{equation}\label{ki0}
n(t)=-\hbox{div}(E)(t),\qquad B(t)=(\nabla\times v)(t),\qquad t\in[0,T_1],
\end{equation}
and
\begin{equation}\label{ki1}
\sup_{t\in[0,T_1]}\|(n(t),v(t),E(t),B(t))\|_{\widetilde{H}^{N_0+1}}\leq\varepsilon_0^{3/4}.
\end{equation}
Given the restriction \eqref{ki0}, the system \eqref{EM4} can be written in an equivalent way, in terms only of the vectors $v$ and $E$,
\begin{equation}\label{EM5}
\begin{split}
&\partial_t v_j=-E_j+\sum_{k=1}^3T\partial_j\partial_kE_k-\sum_{k=1}^3v_k\partial_jv_k,\\
&\partial_t E_j=v_j-c^2\Delta v_j+\sum_{k=1}^3c^2\partial_k\partial_jv_k-\sum_{k=1}^3v_j\partial_kE_k,\\
&n=-\sum_{k=1}^3\partial_kE_k,\qquad B_j=\sum_{k,l=1}^3\in_{jkl}\partial_kv_l.
\end{split}
\end{equation}
Let
\begin{equation}\label{ki1.5}
\begin{split}
&U_1:=\Lambda_1|\nabla|^{-1}\hbox{div}(E)+i|\nabla|^{-1}\hbox{div}(v),\\
&U_2:=\Lambda_2^{-1}|\nabla|^{-1}\hbox{curl}(E)+i|\nabla|^{-1}\hbox{curl}(v),
\end{split}
\end{equation}
where
\begin{equation*}
\Lambda_1:=\sqrt{1-T\Delta},\qquad \Lambda_2:=\sqrt{1-c^2\Delta}.
\end{equation*}
Then $U_1,U_2\in C([0,T_1]:H^{N_0})$ and 
\begin{equation}\label{ki2}
\begin{split}
&\hbox{div}(E)=\Lambda_1^{-1}|\nabla|(\Re U_1),\qquad \hbox{curl}(E)=\Lambda_2|\nabla|(\Re U_2),\qquad \hbox{div}(v)=|\nabla|(\Im U_1),\qquad \hbox{curl}(v)=|\nabla|(\Im U_2),\\
&v_j=-R_j(\Im U_1)+\sum_{m,n=1}^3\in_{jmn}(R_m(\Im U_{2,n})),\qquad E_j=-R_j\Lambda_1^{-1}(\Re U_1)+\sum_{m,n=1}^3\in_{jmn}(\Lambda_2 R_m(\Re U_{2,n})).
\end{split}
\end{equation}
Using these definitions we calculate
\begin{equation*}
\begin{split}
(\partial_t+i\Lambda_1)U_1&=i\Lambda_1^2|\nabla|^{-1}(\hbox{div}(E))-\Lambda_1|\nabla|^{-1}(\hbox{div}(v))\\
&+\Lambda_1|\nabla|^{-1}\big[\hbox{div}(v)-\sum_{j,k=1}^3\partial_j(v_j\partial_kE_k)\big]+i|\nabla|^{-1}\big[(-1+T\Delta)(\hbox{div}(E))-\frac{1}{2}\Delta(|v|^2)\big]\\
&=-\sum_{j=1}^3\Lambda_1R_j(v_j\hbox{div}(E))+\frac{i}{2}\sum_{j=1}^3|\nabla|(v_j^2),
\end{split}
\end{equation*}
and
\begin{equation*}
\begin{split}
(\partial_t+i\Lambda_2)U_{2,j}&=i|\nabla|^{-1}\Big[\sum_{m,n=1}^3\in_{jmn}\partial_mE_n\Big]-\Lambda_2|\nabla|^{-1}\Big[\sum_{m,n=1}^3\in_{jmn}\partial_mv_n\Big]\\
&+\Lambda_2^{-1}|\nabla|^{-1}\Big[\sum_{m,n=1}^3\in_{jmn}\partial_m\big[(1-c^2\Delta) v_n-v_n\hbox{div}(E)\big]\Big]-i|\nabla|^{-1}\Big[\sum_{m,n=1}^3\in_{jmn}\partial_mE_n\Big]\\
&=-\sum_{m,n=1}^3\in_{jmn}\Lambda_2^{-1}R_m\big[v_n\hbox{div}(E)\big].
\end{split}
\end{equation*}
Using the formulas in \eqref{ki2}, it is easy to see that the functions $U_1,U_{2,j}$, $j\in\{1,2,3\}$ satisfy the system of equations
\begin{equation*}
(\partial_t+i\Lambda_1)U_1=\mathcal{N}_1,\qquad (\partial_t+i\Lambda_2)U_{2,j}=\mathcal{N}_{2,j},\qquad j\in\{1,2,3\},
\end{equation*}
where the nonlinearities $\mathcal{N}_1, \mathcal{N}_{2,j}$ can be expressed in terms of the functions $U_1,U_{2,j}$ as in \eqref{nonlinearity}. Therefore we can apply the results in Proposition \ref{Norm0} and Proposition \eqref{Norm}.

We can now proceed as in the previous subsection. Using the definition \eqref{ki1.5} and Lemma \ref{tech3}, it follows that $U_1,U_2\in C([0,T_1]:H^{N_0})$ and
\begin{equation}\label{ki6}
\sup_{t\in[0,T_1]}\big(\|U_1(t)\|_{H^{N_0}}+\|U_2(t)\|_{H^{N_0}}\big)\lesssim \varepsilon_0^{3/4},\qquad \|U_1(0)\|_{Z}+\|U_2(0)\|_{Z}\lesssim \varepsilon_0.
\end{equation}
Let $T_2$ denote the largest number in $(0,T_1]$ with the property that
\begin{equation*}
\sup_{t\in[0,T_2)}\big[\|e^{it\Lambda_1}U_1(t)\|_{Z}+\|e^{it\Lambda_2}U_2(t)\|_{Z}\big]\leq \varepsilon_0^{3/4}.
\end{equation*}
Such a $T_2\in(0,T_1]$ exists, in view of \eqref{ki6} and Proposition \ref{Norm0}. We apply now Proposition \ref{Norm} on the intervals $[0,T_2(1-1/n)]$, $n=2,3,\ldots$, with $\delta_1\approx \varepsilon_0^{3/4}$. It follows that 
\begin{equation*}
\sup_{t\in[0,T_2)}\big[\|e^{it\Lambda_1}U_1(t)\|_{Z}+\|e^{it\Lambda_2}U_2(t)\|_{Z}\big]\lesssim \varepsilon_0.
\end{equation*}
Using again Proposition \ref{Norm0} it follows that $T_2=T_1$ and 
\begin{equation}\label{ki7}
\sup_{t\in[0,T_1]}\big[\|e^{it\Lambda_1}U_1(t)\|_{Z}+\|e^{it\Lambda_2}U_2(t)\|_{Z}\big]\lesssim \varepsilon_0.
\end{equation}

Using the formulas in the second line of \eqref{ki2}, and the bounds \eqref{ki7} and \eqref{ok10} it follows that
\begin{equation}\label{ki8}
\sup_{t\in[0,T_1]}\sup_{|\rho|\leq 4}\big[(1+t)^{1+\beta}\big(\|D^\rho v(t)\|_{L^\infty}+\|D^\rho E(t)\|_{L^\infty}\big)\big]\lesssim \varepsilon_0.
\end{equation}
Recalling the definition \eqref{pla0.5} and the restriction \eqref{ki0}, it follows that
\begin{equation*}
\sup_{t\in[0,T_1]}\big[(1+t)^{1+\beta}\|(n,v,E,B)(t)\|_{Z'}\big]\lesssim \varepsilon_0.
\end{equation*}
Therefore, using the energy estimate \eqref{pla2}, it follows that
\begin{equation*}
\sup_{t\in[0,T_1]}\mathcal{E}_{N_0+1}(t)\lesssim \varepsilon_0.
\end{equation*}
As a consequence, if the solution $(n,v,E,B)$ satisfies the bound \eqref{ki1} on some interval $[0,T_1]$, then it has to satisfy the stronger bound
\begin{equation*}
\sup_{t\in[0,T_1]}\|(n(t),v(t),E(t),B(t))\|_{\widetilde{H}^{N_0+1}}\lesssim \varepsilon_0.
\end{equation*}
Therefore the solution can be extended globally, and the desired bound \eqref{mainconcl2.1} follows using also \eqref{ki8}. This completes the proof of Theorem \ref{MainThm2}.

\section{Proof of Proposition \ref{Norm0}}\label{ContProof} 

In this section we prove Proposition \ref{Norm0}. For simplicity of notation, in this section we let $\widetilde{C}$ denote constants that may depend only on $T_0$, $\sup_{\sigma\in\{1,\ldots,d\}}\|e^{it_0\Lambda_\sigma}U_\sigma(t_0)\|_Z$, $\sup_{\sigma\in\{1,\ldots,d\}}\sup_{t\in[0,T_0]}\|U_\sigma(t)\|_{H^{N_0}}$, and the basic constant $A,d,d'$.

For any integer $J\geq 0$ and $f\in H^{N_0}$ we define
\begin{equation}\label{cu0}
\|f\|_{Z_J}:=\sup_{(k,j)\in\mathcal{J}}2^{\min(0,2J-2j)}\|\phii^{(k)}_j(x)\cdot P_kf(x)\|_{B_{k,j}},
\end{equation}
compare with Definition \ref{MainDef}, and notice that
\begin{equation*}
 \|f\|_{Z_J}\leq\|f\|_{Z},\qquad \|f\|_{Z_J}\lesssim_J\|f\|_{H^{N_0}}.
\end{equation*}
We will show that if $t\leq t'\in [0,T_0]\cap [t_0,t_0+1]$ and $J\in\mathbb{Z}_+$ then
\begin{equation}\label{cu1}
 \sup_{\sigma\in\{1,\ldots,d\}}\|e^{it'\Lambda_\sigma}U_\sigma(t')-e^{it\Lambda_\sigma}U_\sigma(t)\|_{Z_J}\leq \widetilde{C} |t'-t|(1+\sup_{s\in[t,t']}\sup_{\sigma\in\{1,\ldots,d\}}\|e^{is\Lambda_\sigma}U_\sigma(s)\|_{Z_J})^2.
\end{equation}
Assuming \eqref{cu1}, it follows easily that 
\begin{equation*}
\begin{split}
&\sup_{\sigma\in\{1,\ldots,d\}}\sup_{t\in[0,T]\cap[t_0,t_0+\tau]}\|e^{it\Lambda_\sigma}U_\sigma(t)\|_{Z_J}\leq \widetilde{C},\\
&\|e^{it'\Lambda_\sigma}U_\sigma(t')-e^{it\Lambda_\sigma}U_\sigma(t)\|_{Z_J}\leq \widetilde{C}|t'-t|, \qquad \text{ for any }t,t'\in[0,T]\cap[t_0,t_0+\tau],\,\sigma\in\{1,\ldots,d\},
\end{split}
\end{equation*}
uniformly in $J$, provided that $\tau\leq\widetilde{C}^{-1}$ is sufficiently small. The desired conclusions follow by letting $J\to\infty$.

It remains to prove \eqref{cu1}. The equations \eqref{maineq2} and \eqref{nonlinearity} give
\begin{equation}\label{norm2.0}
[\partial_t+i\Lambda_\sigma(\xi)]\widehat{U_{\sigma+}}(\xi,t)=\sum_{\mu,\nu\in\mathcal{I}_d}\int_{\mathbb{R}^3}m_{\sigma;\mu,\nu}(\xi,\eta)\widehat{U_\mu}(\xi-\eta,t)\widehat{U_\nu}(\eta,t)\,d\eta,
\end{equation}
for $\sigma=1,\ldots,d$. Letting
\begin{equation*}
V_{\sigma+}(t):=e^{it\Lambda_\sigma}U_{\sigma+}(t),\qquad V_{\sigma-}(t):=e^{-it\Lambda_\sigma}U_{\sigma-}(t)=\overline{V_{\sigma+}(t)},\qquad\sigma=1,\ldots,d,
\end{equation*}
and
\begin{equation*}
\widetilde{\Lambda}_{\sigma+}:=+\Lambda_\sigma,\qquad \widetilde{\Lambda}_{\sigma-}:=-\Lambda_\sigma,\qquad\sigma=1,\ldots,d,
\end{equation*}
the equations \eqref{norm2.0} are equivalent to
\begin{equation*}
\frac{d}{dt}[\widehat{V_{\sigma+}}(\xi,t)]=\sum_{\mu,\nu\in\mathcal{I}_d}\int_{\mathbb{R}^3}e^{it[\Lambda_\sigma(\xi)-\widetilde{\Lambda}_{\mu}(\xi-\eta)-\widetilde{\Lambda}_{\nu}(\eta)]}m_{\sigma;\mu,\nu}(\xi,\eta)\widehat{V_\mu}(\xi-\eta,t)\widehat{V_\nu}(\eta,t)\,d\eta.
\end{equation*}
Therefore, for any $t\leq t'\in[0,T_0]$ and $\sigma=1,\ldots,d$,
\begin{equation}\label{norm4.1}
\begin{split}
\widehat{V_{\sigma+}}(\xi,t')-\widehat{V_{\sigma+}}(\xi,t)&=\sum_{\mu,\nu\in\mathcal{I}_d}\int_t^{t'}\int_{\mathbb{R}^3}
e^{is[\Lambda_\sigma(\xi)-\widetilde{\Lambda}_{\mu}(\xi-\eta)-\widetilde{\Lambda}_{\nu}(\eta)]}m_{\sigma;\mu,\nu}(\xi,\eta)
\widehat{V_\mu}(\xi-\eta,s)\widehat{V_\nu}(\eta,s)\,d\eta ds\\
&=\sum_{\mu,\nu\in\mathcal{I}_d}\int_t^{t'}Q_s^{\sigma;\mu,\nu}(V_\mu(s),V_\nu(s))ds,
\end{split}
\end{equation}
where
\begin{equation}\label{norm4.0}
\mathcal{F}[Q_s^{\sigma;\mu,\nu}(f,g)](\xi):=\int_{\mathbb{R}^3} e^{is[\Lambda_\sigma(\xi)-\widetilde{\Lambda}_{\mu}(\xi-\eta)-\widetilde{\Lambda}_{\nu}(\eta)]}m_{\sigma;\mu,\nu}(\xi,\eta)
\widehat{f}(\xi-\eta)\widehat{g}(\eta)\,d\eta.
\end{equation}

The desired bound \eqref{cu1} is equivalent to proving that
\begin{equation*}
\sup_{\sigma\in\{1,\ldots,d\}}\|V_{\sigma+}(t')-V_{\sigma+}(t)\|_{Z_J}\leq \widetilde{C} |t'-t|(1+\sup_{s\in[t,t']}\sup_{\sigma\in\{1,\ldots,d\}}\|V_{\sigma+}(s)\|_{Z_J})^2.
\end{equation*}
Using the formulas \eqref{norm4.1}--\eqref{norm4.0} and Definition \ref{MainDef}, it suffices to prove the uniform bound
\begin{equation}\label{cu2}
2^{\min(0,2J-2j)}\|\phii^{(k)}_j\cdot P_kQ_s^{\sigma;\mu,\nu}(V_\mu(s),V_\nu(s))\|_{B^1_{k,j}}\leq \widetilde{C} (1+\sup_{\sigma\in\{1,\ldots,d\}}\|V_{\sigma+}(s)\|_{Z_J})^2,
\end{equation}
for any fixed $(k,j)\in \mathcal{J}$, $s\in[0,T_0]$, $\sigma\in\{1,\ldots,d\}$, and $\mu,\nu\in\mathcal{I}_d$.

Using just the definition \eqref{norm4.0} we estimate easily the $L^\infty$ part of the $B^1_{k,j}$ norm: if $k\leq 0$ then
\begin{equation*}
\big\|\mathcal{F}\big[P_kQ_s^{\sigma;\mu,\nu}(V_\mu(s),V_\nu(s))\big]\big\|_{L^\infty}\lesssim\|(1+|\eta|)\widehat{V_\mu(s)}(\eta)\|_{L^2}\|(1+|\eta|)\widehat{V_\nu(s)}(\eta)\|_{L^2}\leq \widetilde{C}.
\end{equation*}
Similarly, if $k\geq 0$ then
\begin{equation*}
\begin{split}
2^{50 k}&\big\|\mathcal{F}\big[P_kQ_s^{\sigma;\mu,\nu}(V_\mu(s),V_\nu(s))\big]\big\|_{L^\infty}\\
&\lesssim 2^{15k}\Big[\|\mathcal{F}[P_{\leq k}V_\mu(s)]\|_{L^2} \|\mathcal{F}[P_{[k-4,k+4]}V_\nu(s)]\|_{L^2}+ \|\mathcal{F}[P_{[k-4,k+4]}V_\mu(s)]\|_{L^2} \|\mathcal{F}[P_{\leq k}V_\nu(s)]\|_{L^2}\\
&+\sum_{|k_1-k_2|\leq 4,\,k_1\geq k-6}(1+2^{k_1})\|\widehat{P_{k_1}V_\mu(s)}\|_{L^2}\cdot(1+2^{k_2})\|\widehat{P_{k_2}V_\nu(s)}\|_{L^2}\Big]\\
&\leq \widetilde{C}.
\end{split}
\end{equation*}
Therefore, letting $B:=1+\sup_{\sigma\in\{1,\ldots,d\}}\|V_{\sigma+}(s)\|_{Z_J}$, for \eqref{cu2} it remains to prove the uniform bound
\begin{equation}\label{cu3}
2^{\min(0,2J-2j)}(2^{\alpha k}+2^{10k})2^{(1+\beta)j}\|\phii^{(k)}_j\cdot P_kQ_s^{\sigma;\mu,\nu}(V_\mu(s),V_\nu(s))\|_{L^2}\leq \widetilde{C} B^2,
\end{equation}
for any fixed $(k,j)\in \mathcal{J}$, $s\in[0,T_0]$, $\sigma\in\{1,\ldots,d\}$, and $\mu,\nu\in\mathcal{I}_d$.

The desired $L^2$ bound \eqref{cu3} follows easily from the $L^\infty$ bounds proved earlier unless
\begin{equation}\label{cu4}
j\geq \widetilde{C}+\max(20k,-5k/4).
\end{equation}
Decomposing
\begin{equation*}
V_\mu(s)=\sum_{k_1\in\mathbb{Z}}P_{k_1}(V_\mu(s)),\qquad V_\nu(s)=\sum_{k_2\in\mathbb{Z}}P_{k_2}(V_\nu(s))
\end{equation*}
for \eqref{cu3} it suffices to prove that
\begin{equation}\label{cu5}
2^{\min(0,2J-2j)}(2^{\alpha k}+2^{10k})2^{(1+\beta)j}\sum_{(k_1,k_2)\in\mathcal{X}_k}\|\phii^{(k)}_j\cdot P_kQ_s^{\sigma;\mu,\nu}(P_{k_1}V_\mu(s),P_{k_2}V_\nu(s))\|_{L^2}\leq \widetilde{C} B^2,
\end{equation}
for any fixed $(k,j)\in \mathcal{J}$ satisfying \eqref{cu4}, $s\in[0,T_0]$, $\sigma\in\{1,\ldots,d\}$, and $\mu,\nu\in\mathcal{I}_d$.

Using first the simple bound
\begin{equation*}
\begin{split}
\|\mathcal{F}[P_k&Q_s^{\sigma;\mu,\nu}(P_{k_1}V_\mu(s),P_{k_2}V_\nu(s))]\|_{L^2}\\
&\lesssim (1+2^{\max(k_1,k_2)})\min\big[\|\widehat{P_{k_1}V_\mu(s)}\|_{L^2}\|\widehat{P_{k_2}V_\nu(s)}\|_{L^1},\|\widehat{P_{k_1}V_\mu(s)}\|_{L^1}\|\widehat{P_{k_2}V_\nu(s)}\|_{L^2}\big]\\
&\lesssim (1+2^{\max(k_1,k_2)})2^{3\min(k_1,k_2)/2}\|\widehat{P_{k_1}V_\mu(s)}\|_{L^2}\|\widehat{P_{k_2}V_\nu(s)}\|_{L^2},
\end{split}
\end{equation*}
we estimate 
\begin{equation*}
(2^{\alpha k}+2^{10k})2^{(1+\beta)j}\sum_{(k_1,k_2)\in\mathcal{X}_k,\,\min(k_1,k_2)\leq-4j/5}\|P_kQ_s^{\sigma;\mu,\nu}(P_{k_1}V_\mu(s),P_{k_2}V_\nu(s))\|_{L^2}\leq \widetilde{C},
\end{equation*}
and
\begin{equation*}
(2^{\alpha k}+2^{10k})2^{(1+\beta)j}\sum_{(k_1,k_2)\in\mathcal{X}_k,\,\max(k_1,k_2)\geq j/20}\|P_kQ_s^{\sigma;\mu,\nu}(P_{k_1}V_\mu(s),P_{k_2}V_\nu(s))\|_{L^2}\leq \widetilde{C}.
\end{equation*}
Therefore, for \eqref{cu5} it suffices to prove the uniform bound
\begin{equation}\label{cu6}
\begin{split}
2^{\min(0,2J-2j)}(2^{\alpha k}+2^{10k})&2^{(1+\beta)j}\sum_{(k_1,k_2)\in\mathcal{X}_k,\,-4j/5\leq k_1\leq k_2\leq j/20}\\
&\|\phii^{(k)}_j\cdot P_kQ_s^{\sigma;\mu,\nu}(P_{k_1}V_\mu(s),P_{k_2}V_\nu(s))\|_{L^2}\leq \widetilde{C} B^2,
\end{split}
\end{equation}
for any fixed $(k,j)\in \mathcal{J}$ satisfying \eqref{cu4}, $s\in[0,T_0]$, $\sigma\in\{1,\ldots,d\}$, and $\mu,\nu\in\mathcal{I}_d$.

To prove \eqref{cu6} we further decompose
\begin{equation*}
\begin{split}
&P_{k_1}V_\mu(s)=\sum_{j_1\geq\max(-k_1,0)}P_{[k_1-2,k_1+2]}[\phii_{j_1}^{(k_1)}\cdot P_{k_1}(V_\mu(s))]=\sum_{j_1\geq\max(-k_1,0)}P_{[k_1-2,k_1+2]}(g_{k_1,j_1}),\\
&P_{k_2}V_\nu(s)=\sum_{j_2\geq\max(-k_2,0)}P_{[k_2-2,k_2+2]}[\phii_{j_2}^{(k_2)}\cdot P_{k_2}(V_\nu(s))]=\sum_{j_2\geq\max(-k_2,0)}P_{[k_2-2,k_2+2]}(g_{k_2,j_2}).
\end{split}
\end{equation*}
Then we rewrite, using the definitions,
\begin{equation*} P_kQ_s^{\sigma;\mu,\nu}(P_{[k_1-2,k_1+2]}(g_{k_1,j_1}),P_{[k_2-2,k_2+2]}(g_{k_2,j_2}))(x)=\int_{\mathbb{R}^3\times\mathbb{R}^3}K(x,y_1,y_2)g_{k_1,j_1}(y_1)g_{k_2,j_2}(y_2)\,dy_1dy_2,
\end{equation*}
where 
\begin{equation*}
\begin{split}
K(x,y_1,y_2):=C\int_{\mathbb{R}^3\times\mathbb{R}^3}&e^{i[(x-y_1)\cdot\xi+(y_1-y_2)\cdot\eta]}e^{is[\Lambda_\sigma(\xi)-\widetilde{\Lambda}_{\mu}(\xi-\eta)-\widetilde{\Lambda}_{\nu}(\eta)]}\\
&\times m_{\sigma;\mu,\nu}(\xi,\eta)\varphi_{[k_1-2,k_1+2]}(\xi-\eta)\varphi_{[k_2-2,k_2+2]}(\eta)\varphi_k(\xi)\,d\xi d\eta.
\end{split}
\end{equation*}
Recall that $k,k_1,k_2\in[-4j/5,j/20]$ and $j\geq\widetilde{C}$. Therefore we can integrate by parts in $\xi$ or $\eta$ to conclude that
\begin{equation*}
\text{ if }|x-y_1|+|y_1-y_2|\geq 2^{j-10} \text{ then }|K(x,y_1,y_2)|\leq\widetilde{C}(|x-y_1|+|y_1-y_2|)^{-10}.
\end{equation*}
Therefore, the contributions of the functions $g_{k_1,j_1}$ and $g_{k_2,j_2}$ corresponding to $|j_1-j|+|j_2-j|\geq 10$ are easily bounded,
\begin{equation*}
\begin{split}
(2^{\alpha k}+2^{10k})&2^{(1+\beta)j}\sum_{(k_1,k_2)\in\mathcal{X}_k,\,-4j/5\leq k_1,k_2\leq j/20}\\
&\sum_{|j_1-j|+|j_2-j|\geq 10}\|\phii^{(k)}_j\cdot P_kQ_s^{\sigma;\mu,\nu}(P_{[k_1-2,k_1+2]}(g_{k_1,j_1}),P_{[k_2-2,k_2+2]}(g_{k_2,j_2}))\|_{L^2}\leq \widetilde{C}.
\end{split}
\end{equation*}
Finally, for \eqref{cu6} it remains to prove the uniform bound
\begin{equation}\label{cu10}
\begin{split}
2^{\min(0,2J-2j)}&(2^{\alpha k}+2^{10k})2^{(1+\beta)j}\sum_{(k_1,k_2)\in\mathcal{X}_k,\,-4j/5\leq k_1\leq k_2\leq j/20}\\
&\|P_kQ_s^{\sigma;\mu,\nu}(P_{[k_1-2,k_1+2]}(g_{k_1,j_1}),P_{[k_2-2,k_2+2]}(g_{k_2,j_2}))\|_{L^2}\leq \widetilde{C} B^2,
\end{split}
\end{equation}
for any fixed $(k,j)\in \mathcal{J}$ satisfying \eqref{cu4}, $j_1,j_2\in[j-10,j+10]$, $s\in[0,T_0]$, $\sigma\in\{1,\ldots,d\}$, and $\mu,\nu\in\mathcal{I}_d$.

Using the definition \eqref{cu0}, 
\begin{equation*}
\|g_{k_1,j_1}\|_{B_{k_1,j_1}}+\|g_{k_2,j_2}\|_{B_{k_2,j_2}}\lesssim B2^{-\min(0,2J-2j)}
\end{equation*}
for any $k_1,k_2\in [-4j/5,j/20]$ and $j_1,j_2\in[j-10,j+10]$. Therefore, using \eqref{sec5.82},
\begin{equation*}
\|\mathcal{F}(P_{[k_1-2,k_1+2]}(g_{k_1,j_1}))\|_{L^1}\lesssim B2^{-\min(0,2J-2j)}\cdot (2^{\alpha k_1}+2^{10k_1})^{-1}2^{3k_1/2}2^{-(1+\beta)j_1}.
\end{equation*}
Since
\begin{equation*}
\|\widehat{g_{k_2,j_2}}\|_{L^2}\leq\widetilde{C}(1+2^{k_2})^{-N_0},
\end{equation*}
we can estimate, for $k_1\leq k_2\in[-4j/5,j/20]$ and $j_1,j_2\in[j-10,j+10]$,
\begin{equation*}
\begin{split}
\|P_kQ_s^{\sigma;\mu,\nu}&(P_{[k_1-2,k_1+2]}(g_{k_1,j_1}),P_{[k_2-2,k_2+2]}(g_{k_2,j_2}))\|_{L^2}\\
&\lesssim (2^{k_2}+1)\|\mathcal{F}(P_{[k_1-2,k_1+2]}(g_{k_1,j_1}))\|_{L^1}\|\mathcal{F}(P_{[k_2-2,k_2+2]}(g_{k_2,j_2}))\|_{L^2}\\
&\leq\widetilde{C}B2^{-\min(0,2J-2j)}\cdot (2^{\alpha k_1}+2^{10k_1})^{-1}2^{3k_1/2}2^{-(1+\beta)j}\cdot(1+2^{k_2})^{-(N_0-1)}.
\end{split}
\end{equation*}
Therefore the left-hand side of \eqref{cu10} is dominated by
\begin{equation*}
(2^{\alpha k}+2^{10k})\sum_{(k_1,k_2)\in\mathcal{X}_k,\,k_1\leq k_2}\widetilde{C}B(2^{\alpha k_1}+2^{10k_1})^{-1}2^{3k_1/2}(1+2^{k_2})^{-(N_0-1)}\lesssim \widetilde{C}B,
\end{equation*}
as desired. This completes the proof of the proposition.

\section{Proof of Proposition \ref{Norm}}\label{normproof}

In this section we prove Proposition \ref{Norm}, in several stages. We derive first several new formulas describing the solutions $U_\sigma$. 

\subsection{Renormalizations} We will use the definition and the notation introduced in subsection \ref{MainSec2}. The equations \eqref{maineq} and \eqref{nonlinearity} give
\begin{equation}\label{norm2}
[\partial_t+i\Lambda_\sigma(\xi)]\widehat{U_{\sigma+}}(\xi,t)=\sum_{\mu,\nu\in\mathcal{I}_d}\int_{\mathbb{R}^3}m_{\sigma;\mu,\nu}(\xi,\eta)\widehat{U_\mu}(\xi-\eta,t)\widehat{U_\nu}(\eta,t)\,d\eta,
\end{equation}
for $\sigma=1,\ldots,d$. Letting
\begin{equation*}
V_{\sigma+}(t):=e^{it\Lambda_\sigma}U_{\sigma+}(t),\qquad V_{\sigma-}(t):=e^{-it\Lambda_\sigma}U_{\sigma-}(t)=\overline{V_{\sigma+}(t)},\qquad\sigma=1,\ldots,d,
\end{equation*}
and
\begin{equation*}
\widetilde{\Lambda}_{\sigma+}:=+\Lambda_\sigma,\qquad \widetilde{\Lambda}_{\sigma-}:=-\Lambda_\sigma,\qquad\sigma=1,\ldots,d,
\end{equation*}
the equations \eqref{norm2} are equivalent to
\begin{equation}\label{norm3}
\frac{d}{dt}[\widehat{V_{\sigma+}}(\xi,t)]=\sum_{\mu,\nu\in\mathcal{I}_d}\int_{\mathbb{R}^3}e^{it[\Lambda_\sigma(\xi)-\widetilde{\Lambda}_{\mu}(\xi-\eta)-\widetilde{\Lambda}_{\nu}(\eta)]}m_{\sigma;\mu,\nu}(\xi,\eta)\widehat{V_\mu}(\xi-\eta,t)\widehat{V_\nu}(\eta,t)\,d\eta.
\end{equation}
Therefore, for any $t\in[0,T_0]$ and $\sigma=1,\ldots,d$,
\begin{equation}\label{norm4}
\widehat{V_{\sigma+}}(\xi,t)-\widehat{V_{\sigma+}}(\xi,0)=\sum_{\mu,\nu\in\mathcal{I}_d}\int_0^t\int_{\mathbb{R}^3}
e^{is[\Lambda_\sigma(\xi)-\widetilde{\Lambda}_{\mu}(\xi-\eta)-\widetilde{\Lambda}_{\nu}(\eta)]}m_{\sigma;\mu,\nu}(\xi,\eta)
\widehat{V_\mu}(\xi-\eta,s)\widehat{V_\nu}(\eta,s)\,d\eta ds.
\end{equation}

The desired bound \eqref{sec3} is equivalent to proving that
\begin{equation}\label{nh1}
\|V_{\sigma+}(t)-V_{\sigma+}(0)\|_Z\lesssim \delta_1^2,
\end{equation}
for any $t\in[0,T_0]$ and any $\sigma\in\{1,\ldots,d\}$. Given $t\in[0,T_0]$, we fix a suitable decomposition of the function $\mathbf{1}_{[0,t]}$, i.e. we fix functions $q_0,\ldots,q_{L+1}:\mathbb{R}\to[0,1]$, $|L-\log_2(2+t)|\leq 2$, with the properties
\begin{equation}\label{nh2}
\begin{split}
&\sum_{m=0}^{L+1}q_l(s)=\mathbf{1}_{[0,t]}(s),\qquad \mathrm{supp}\,q_0\subseteq [0,2], \qquad \mathrm{supp}\,q_{L+1}\subseteq [t-2,t],\qquad\mathrm{supp}\,q_m\subseteq [2^{m-1},2^{m+1}],\\
&q_m\in C^1(\mathbb{R})\qquad\text{ and }\qquad \int_0^t|q'_m(s)|\,ds\lesssim 1\qquad\text{ for }m=1,\ldots,L.
\end{split}
\end{equation}

Recall the assumption $m_{\sigma;\mu,\nu}\in\mathcal{M}_{d'}$ and the definition \eqref{symb2}. Using also Lemma \ref{tech3} and the formula \eqref{norm4}, for \eqref{nh1} it suffices to prove the following proposition.

\begin{proposition}\label{reduced1} Assume $t\in[0,T_0]$ is fixed and define the functions $q_m$ as in \eqref{nh2}. For any $\sigma\in\{1,\ldots,d\}$, $\mu,\nu\in\mathcal{I}_d$ we define the bilinear operators $T_{m}^{\sigma;\mu,\nu}$ by 
\begin{equation}\label{nh5}
\mathcal{F}\big[T_{m}^{\sigma;\mu,\nu}(f,g)\big](\xi):=\int_{\mathbb{R}}\int_{\mathbb{R}^3}e^{is[\Lambda_\sigma(\xi)-\widetilde{\Lambda}_{\mu}(\xi-\eta)-\widetilde{\Lambda}_{\nu}(\eta)]}q_m(s)\cdot \widehat{f}(\xi-\eta,s)\widehat{g}(\eta,s)\,d\eta ds.
\end{equation}
Assume that
\begin{equation}\label{nh4.5}
f_\mu:=\delta_1^{-1}Q_\mu V_\mu,\qquad \text{ for some normalized Calderon--Zygmund operator }Q_\mu
\end{equation}
for any $\mu\in\mathcal{I}_d$, and decompose
\begin{equation}\label{ok20}
f_\mu=\sum_{k'\in\mathbb{Z}}\sum_{j'\geq \max(-k',0)}P_{[k'-2,k'+2]}(\phii_{j'}^{(k')}\cdot P_{k'}f_\mu)=\sum_{(k',j')\in\mathcal{J}}f_{k',j'}^\mu.
\end{equation}
Then
\begin{equation}\label{nh3}
\sum_{(k_1,j_1),(k_2,j_2)\in\mathcal{J}}(1+2^{k_1}+2^{k_2})\big\|\widetilde{\varphi}^{(k)}_j\cdot P_kT_{m}^{\sigma;\mu,\nu}(f^\mu_{k_1,j_1},f^\nu_{k_2,j_2})\big\|_{B_{k,j}}\lesssim 2^{-\beta^4 m}
\end{equation}
for any fixed
\begin{equation}\label{nh4}
\sigma\in\{1,\ldots,d\},\quad\mu,\nu\in\mathcal{I}_d,\quad (k,j)\in \mathcal{J},\quad m\in\{0,\ldots,L+1\},
\end{equation}
\end{proposition}

It follows from the definition that
\begin{equation}\label{nh9.2}
\begin{split}
&T^{\sigma;\mu,\nu}_m(f,g)=\int_\mathbb{R}q_m(s)\widetilde{T}^{\sigma;\mu,\nu}_s(f(s),g(s))\,ds,\\
&\mathcal{F}\big[\widetilde{T}^{\sigma;\mu,\nu}_s(f',g')\big](\xi):=\int_{\mathbb{R}^3}e^{is[\Lambda_\sigma(\xi)-\widetilde{\Lambda}_{\mu}(\xi-\eta)-\widetilde{\Lambda}_{\nu}(\eta)]}\cdot \widehat{f'}(\xi-\eta)\widehat{g'}(\eta)\,d\eta.
\end{split}
\end{equation}

For $\sigma\in\{1,\ldots,d\}$ and $\mu,\nu\in\mathcal{I}_d$, we define the smooth functions $\Phi^{\sigma;\mu,\nu}:\mathbb{R}^3\times\mathbb{R}^3\to\mathbb{R}$ and $\Xi^{\mu,\nu}:\mathbb{R}^3\times\mathbb{R}^3\to\mathbb{R}^3$,
\begin{equation}\label{ra2}
\begin{split}
\Phi^{\sigma;\mu,\nu}(\xi,\eta):=\Lambda_\sigma(\xi)-\widetilde{\Lambda}_{\mu}(\xi-\eta)-\widetilde{\Lambda}_{\nu}(\eta),\qquad \Xi^{\mu,\nu}(\xi,\eta):=(\nabla_\eta\Phi^{\sigma;\mu,\nu})(\xi,\eta).
\end{split}
\end{equation}
Many of the bounds needed in the proof of of Proposition \ref{reduced1} rely on having a good understanding of the functions $\Phi^{\sigma;\mu,\nu}$ and $\Xi^{\mu,\nu}$. The relevant properties are proved in subsection \ref{ra}.

In view of Lemma \ref{tech3} and the main hypothesis \eqref{sec2}, we have
\begin{equation}\label{nh5.5}
\sup_{t\in[0,T_0]}\|f_\mu(t)\|_{H^{N_0}\cap Z}\lesssim 1.
\end{equation}
for functions $f_\mu$ defined as in \eqref{nh5}. Letting
\begin{equation}\label{nh5.6}
Ef^\mu_{k',j'}(s):=e^{-is\widetilde{\Lambda}_\mu}f^\mu_{k',j'}(s),
\end{equation}
it follows from Lemma \ref{tech1.5} that for any $\mu\in\mathcal{I}_d$ and $s\in[0,T_0]$,
\begin{equation}\label{nh9}
\begin{split}
&\sum_{j'\geq \max(-k',0)}(\|Ef^\mu_{k',j'}(s)\|_{L^2}+\|f^\mu_{k',j'}(s)\|_{L^2})\lesssim \min(2^{-(N_0-1)k'},2^{(1+\beta-\alpha)k'}),\\
&\sum_{j'\geq \max(-k',0)}\|Ef^\mu_{k',j'}(s)\|_{L^\infty}\lesssim \min(2^{-6k'},2^{(1/2-\beta-\alpha)k'})(1+s)^{-1-\beta},\\
&\sup_{\xi\in\mathbb{R}^3}\big|D^\rho_\xi\widehat{f^\mu_{k',j'}}(\xi,s)\big|\lesssim_{|\rho|} (2^{\alpha k'}+2^{10k'})^{-1}\cdot 2^{-(1/2-\beta)\widetilde{k'}}2^{|\rho|j'}.
\end{split}
\end{equation}
Sometimes, we will also need the more precise bound
\begin{equation}\label{nh9.1}
\|Ef^\mu_{k',j'}(s)\|_{L^2}+\|f^\mu_{k',j'}(s)\|_{L^2}\lesssim (2^{\alpha k'}+2^{10k'})^{-1}2^{2\beta\widetilde{k'}}2^{-(1-\be)j'}\qquad\text{ for any }(k',j')\in\mathcal{J}.
\end{equation}

In addition to the bounds \eqref{nh5.5}--\eqref{nh9.1}, we will also need bounds on the derivatives $(\partial_sf^\mu_{k',j'})(s)$, in order to be able to integrate by parts in $s$. More precisely:

\begin{lemma}\label{ders}
(i) With $f^\mu_{k',j'}(s)$ as in \eqref{nh4.5} and \eqref{ok20}, for any $s\in[0,T_0]$, $\mu\in\mathcal{I}_d$, and $(k',j')\in\mathcal{J}$,
\begin{equation}\label{ok50}
\|(\partial_sf^\mu_{k',j'})(s)\|_{L^2}\lesssim \min[(1+s)^{-1-\beta},2^{3k'/2}]\cdot\min [1,2^{-(N_0-5)k'}],
\end{equation}

(ii) In addition, for any $\mu\in\mathcal{I}_d$, $(k',j')\in\mathcal{J}$ with $k'\in [-D/2,3D/2]$, and $s\in[0,T_0]$,
\begin{equation}\label{nxc1}
 \Vert (\partial_s\widehat{f^\mu_{k',j'}})(s)\Vert_{L^\infty}\lesssim (1+s)^{-1-\beta/10}.
 \end{equation}
\end{lemma}

\begin{proof}[Proof of Lemma \ref{ders}] (i) We may assume that $\mu=(\sigma+)$ for some $\sigma\in\{1,\ldots,d\}$, and use formula \eqref{norm3}. It follows that
\begin{equation}\label{ok51}
\begin{split}
&\|(\partial_sf^{(\sigma+)}_{k',j'})(s)\|_{L^2}\lesssim \delta_1^{-1}\sum_{\mu,\nu\in\mathcal{I}_d}\Big\|\varphi_{k'}(\xi)\int_{\mathbb{R}^3}e^{-is[\widetilde{\Lambda}_{\mu}(\xi-\eta)+\widetilde{\Lambda}_{\nu}(\eta)]}m_{\sigma;\mu,\nu}(\xi,\eta)\widehat{V_\mu}(\xi-\eta,s)\widehat{V_\nu}(\eta,s)\,d\eta\Big\|_{L^2_\xi}\\
&\lesssim \delta_1^{-1}\sum_{\mu,\nu\in\mathcal{I}_d}\sum_{(k_1,k_2)\in \mathcal{X}_{k'}}\Big\|\varphi_{k'}(\xi)\int_{\mathbb{R}^3}e^{-is[\widetilde{\Lambda}_{\mu}(\xi-\eta)+\widetilde{\Lambda}_{\nu}(\eta)]}m_{\sigma;\mu,\nu}(\xi,\eta)\widehat{P_{k_1}V_\mu}(\xi-\eta,s)\widehat{P_{k_2}V_\nu}(\eta,s)\,d\eta\Big\|_{L^2_\xi}.
\end{split}
\end{equation}
The main assumption \eqref{sec2} shows that 
\begin{equation*}
\|V_\mu(s)\|_{Z\cap H^{N_0}}\lesssim \delta_1,
\end{equation*}
for any $s\in [0,t]$ and $\mu\in\mathcal{I}_d$. Therefore, using \eqref{ok9}--\eqref{ok10},
\begin{equation}\label{ok52}
\begin{split}
&\|P_{k''}V_\mu(s)\|_{L^2}\lesssim \delta_1\min(2^{(1+\beta-\alpha)k''},2^{-N_0k''}),\\
&\|e^{-is\widetilde{\Lambda}_\mu}P_{k''}V_\mu(s)\|_{L^\infty}\lesssim \delta_1\min(2^{(1/2-\beta-\alpha)k''},2^{-6k''})(1+s)^{-1-\beta},
\end{split}
\end{equation}
for any $s\in [0,T_0]$, $\mu\in\mathcal{I}_d$, and $k''\in\mathbb{Z}$.

Using \eqref{ok51}, \eqref{ok52}, and the definition of the space $\mathcal{M}_{d'}$ in \eqref{symb2},
\begin{equation*}
\begin{split}
&\|(\partial_sf^{(\sigma+)}_{k',j'})(s)\|_{L^2}\\
&\lesssim \delta_1\sum_{(k_1,k_2)\in \mathcal{X}_{k'},\,k_1\leq k_2} \min(2^{(1+\beta-\alpha)k_2},2^{-(N_0-2)k_2})\cdot \min(2^{(1/2-\beta-\alpha)k_1},2^{-6k_1})(1+s)^{-1-\beta}\\
&\lesssim (1+s)^{-1-\beta}\min(1,2^{-(N_0-5)k'}).
\end{split}
\end{equation*}
Moreover, if $k'\leq 0$, then we can estimate, using again \eqref{ok51}, \eqref{ok52}, and the definition \eqref{symb2},
\begin{equation*}
\begin{split}
&\|(\partial_sf^{(\sigma+)}_{k',j'})(s)\|_{L^2}\\
&\lesssim \delta_1^{-1}\sum_{\mu,\nu\in\mathcal{I}_d}\sum_{(k_1,k_2)\in \mathcal{X}_{k'}}2^{3k'/2}\Big\|\int_{\mathbb{R}^3}e^{-is[\widetilde{\Lambda}_{\mu}(\xi-\eta)+\widetilde{\Lambda}_{\nu}(\eta)]}m_{\sigma;\mu,\nu}(\xi,\eta)\widehat{P_{k_1}V_\mu}(\xi-\eta,s)\widehat{P_{k_2}V_\nu}(\eta,s)\,d\eta\Big\|_{L^\infty_\xi}\\
&\lesssim \delta_12^{3k'/2}\sum_{(k_1,k_2)\in \mathcal{X}_{k'}}\min(2^{(1+\beta-\alpha)k_1},2^{-(N_0-2)k_1})\cdot \min(2^{(1+\beta-\alpha)k_2},2^{-(N_0-2)k_2})\\
&\lesssim 2^{3k'/2}.
\end{split}
\end{equation*}
The desired bound \eqref{ok50} follows.

To prove (ii) it suffices to prove that
\begin{equation*}
\Vert (\partial_s\widehat{P_{k'}V_{(\sigma+)}})(s)\Vert_{L^\infty}\lesssim \delta_1(1+s)^{-1-\beta/10}.
\end{equation*}
Using \eqref{norm3} it suffices to prove that
\begin{equation*}
\Big|\phi_{k'}(\xi)\int_{\mathbb{R}^3}e^{is[\Lambda_\sigma(\xi)-\widetilde{\Lambda}_{\mu}(\xi-\eta)-\widetilde{\Lambda}_{\nu}(\eta)]}m_{\sigma;\mu,\nu}(\xi,\eta)\widehat{V_\mu}(\xi-\eta,s)\widehat{V_\nu}(\eta,s)\,d\eta\Big|\lesssim \delta_1(1+s)^{-1-\beta/10},
\end{equation*}
for any $\xi\in\mathbb{R}^3$, $\mu,\nu\in\mathcal{I}_d$, $\sigma\in\{1,\ldots,d\}$, and $s\in[0,T_0]$. Recall that $\|V_\mu(s)\|_{Z\cap H^{N_0}}\lesssim \delta_1$, see \eqref{sec2}. Using the definition of the space $\mathcal{M}_{d'}$ in \eqref{symb2} and Lemma \ref{tech3}, it suffices to prove that if 
\begin{equation}\label{nxc2}
\|g_1\|_{Z\cap H^{N_0}}+\|g_2\|_{Z\cap H^{N_0}}\leq 1,
\end{equation}
and we decompose
\begin{equation*}
g_i=\sum_{(k_i,j_i)\in\mathcal{J}}g^i_{k_i,j_i},\qquad g^i_{k_i,j_i}:=P_{[k_i-2,k_i+2]}(\phii^{(k_i)}_{j_i}\cdot P_{k_i}g_i),\qquad i=1,2,
\end{equation*}
then
\begin{equation}\label{nxc3}
\sum_{(k_1,j_1),(k_2,j_2)\in\mathcal{J}}2^{\max(k_1,k_2)}\Big|\phi_{k'}(\xi)\int_{\mathbb{R}^3}e^{is[\Lambda_\sigma(\xi)-\widetilde{\Lambda}_{\mu}(\xi-\eta)-\widetilde{\Lambda}_{\nu}(\eta)]}\widehat{g^1_{k_1,j_1}}(\xi-\eta)\widehat{g^2_{k_2,j_2}}(\eta)\,d\eta\Big|\lesssim (1+s)^{-1-\beta/10},
\end{equation}
for any $\xi\in\mathbb{R}^3$, $\mu,\nu\in\mathcal{I}_d$, $\sigma\in\{1,\ldots,d\}$, $s\in\mathbb{R}$, and $k'\in\mathbb{Z}\cap [-D/2,3D/2]$.

We use first only the $L^2$ bounds
\begin{equation}\label{nxc3.5}
\|g^1_{k_1,j_1}\|_{L^2}\lesssim \min(2^{-N_0 k_1},2^{(2\beta-\alpha)\widetilde{k_1}}2^{-(1-\beta)j_1}), \qquad \|g^2_{k_2,j_2}\|_{L^2}\lesssim \min(2^{-N_0 k_2},2^{(2\beta-\alpha)\widetilde{k_2}}2^{-(1-\beta)j_2}),
\end{equation}
see \eqref{nxc2} and \eqref{mk15.5}, and estimate easily
\begin{equation*}
\sum_{((k_1,j_1),(k_2,j_2))\in J_1}2^{\max(k_1,k_2)}\Big|\phi_{k'}(\xi)\int_{\mathbb{R}^3}e^{is[\Lambda_\sigma(\xi)-\widetilde{\Lambda}_{\mu}(\xi-\eta)-\widetilde{\Lambda}_{\nu}(\eta)]}\widehat{g^1_{k_1,j_1}}(\xi-\eta)\widehat{g^2_{k_2,j_2}}(\eta)\,d\eta\Big|\lesssim (1+s)^{-1-\beta/10},
\end{equation*}
where
\begin{equation*}
J_1:=\{((k_1,j_1),(k_2,j_2))\in\mathcal{J}\times\mathcal{J}: (k_1,k_2)\in\mathcal{X}_{k'},\,2^{\max(k_1,k_2)}\geq (1+s)^{2/N_0}\text{ or }2^{\max(j_1,j_2)}\geq (1+s)^{1+4\beta}\}.
\end{equation*}
Also, the full bound \eqref{nxc3} follows easily if $s\leq 2^{D^2}$. We let 
\begin{equation*}
J_2:=\{((k_1,j_1),(k_2,j_2))\in\mathcal{J}\times\mathcal{J}: (k_1,k_2)\in\mathcal{X}_{k'},\,2^{\max(k_1,k_2)}\leq (1+s)^{2/N_0}\text{ and }2^{\max(j_1,j_2)}\leq (1+s)^{1+4\beta}\},
\end{equation*}
and notice that $J_2$ has at most $C\ln(2+s)^4$ elements. Therefore, for \eqref{nxc3} it suffices to prove that
\begin{equation}\label{nxc4}
\Big|\phi_{k'}(\xi)\int_{\mathbb{R}^3}e^{is[\Lambda_\sigma(\xi)-\widetilde{\Lambda}_{\mu}(\xi-\eta)-\widetilde{\Lambda}_{\nu}(\eta)]}\widehat{g^1_{k_1,j_1}}(\xi-\eta)\widehat{g^2_{k_2,j_2}}(\eta)\,d\eta\Big|\lesssim 2^{-\max(k_1,k_2)}s^{-1-\beta/9},
\end{equation}
for any $\xi\in\mathbb{R}^3$, $\mu,\nu\in\mathcal{I}_d$, $\sigma\in\{1,\ldots,d\}$, $s\geq 2^{D^2}$, $k'\in\mathbb{Z}\cap [-D/2,3D/2]$, and any $((k_1,j_1),(k_2,j_2))\in J_2$.

Without loss of generality, in proving \eqref{nxc4} we may assume that $j_1\leq j_2$. Assume first that
\begin{equation}\label{nxc5}
2^{j_2}\geq 2^{-D^2}(1+s)^{1-\beta/6}.
\end{equation}
Then, using \eqref{sec5.8}, \eqref{sec5.82} and the assumption \eqref{nxc2}, we have
\begin{equation*}
\|\widehat{g^2_{k_2,j_2}}\|_{L^1}\lesssim 2^{-(1+\beta)j_2}2^{3k_2/2}(2^{\alpha k_2}+2^{10k_2})^{-1}.
\end{equation*}
Using \eqref{mk15.55}, $\|\widehat{g^1_{k_1,j_1}}\|_{L^\infty}\lesssim 2^{-\widetilde{k_1}/2}$. Using also \eqref{nxc3.5} we estimate the left-hand side of \eqref{nxc4} by
\begin{equation*}
\begin{split}
C\min(\|\widehat{g^1_{k_1,j_1}}\|_{L^\infty}\|\widehat{g^2_{k_2,j_2}}\|_{L^1},\|\widehat{g^1_{k_1,j_1}}\|_{L^2}\|\widehat{g^2_{k_2,j_2}}\|_{L^2})&\lesssim \min(2^{-\widetilde{k_1}/2}2^{-(1+\beta)j_2},2^{\widetilde{k_1}(1+\beta-\alpha)}2^{-(1-\beta)j_2})\\
&\lesssim 2^{-(1+\beta/3)j_2}.
\end{split}
\end{equation*}
The desired bound \eqref{nxc4} follows if we assume \eqref{nxc5}.

Finally it remains to prove \eqref{nxc4} assuming that
\begin{equation}\label{nxc6}
j_1\leq j_2,\qquad 2^{j_2}\leq 2^{-D^2}(1+s)^{1-\beta/6}.
\end{equation}
In this case we would like to integrate by parts in $\eta$ to estimate the integral in \eqref{nxc4}. Let
\begin{equation*}
K=(1+s)^{\beta^2}\left[2^{j_2}+(1+s)^\frac{1}{2}\right],\quad \delta=K(1+s)^{-1},\quad \epsilon=\min(2^{-j_2},(1+s)^{-1/2}).
\end{equation*}
Recalling the definition \eqref{ra2}, using the bounds \eqref{ln1} and \eqref{mk15.55},
\begin{equation}\label{nxc7}
\Big|\int_{\mathbb{R}^3}[1-\varphi_{\leq 0}(\delta^{-1}\Xi^{\mu,\nu}(\xi,\eta))]e^{is[\Lambda_\sigma(\xi)-\widetilde{\Lambda}_{\mu}(\xi-\eta)-\widetilde{\Lambda}_{\nu}(\eta)]}\widehat{g^1_{k_1,j_1}}(\xi-\eta)\widehat{g^2_{k_2,j_2}}(\eta)\,d\eta\Big|\lesssim (1+s)^{-2}.
\end{equation}
Moreover, using \eqref{tik102}  (since $k^\prime\ge -D/2$, the last formula in \eqref{jb2} shows that the integral below is nontrivial only if $\min(k_1,k_2)\geq -D$)
\begin{equation}\label{nxc8}
\begin{split}
\Big|\phi_{k'}(\xi)\int_{\mathbb{R}^3}\varphi_{\leq 0}(\delta^{-1}\Xi^{\mu,\nu}(\xi,\eta))e^{is[\Lambda_\sigma(\xi)-\widetilde{\Lambda}_{\mu}(\xi-\eta)-\widetilde{\Lambda}_{\nu}(\eta)]}\widehat{g^1_{k_1,j_1}}(\xi-\eta)\widehat{g^2_{k_2,j_2}}(\eta)\,d\eta\Big|&\\
\lesssim \int_{\mathbb{R}^3}\mathbf{1}_{[0,C2^{4\max(k_1,k_2)}\delta]}(\eta-p^{\mu,\nu}(\xi))|\widehat{g^1_{k_1,j_1}}(\xi-\eta)|\,|\widehat{g^2_{k_2,j_2}}(\eta)|\,d\eta&.
\end{split}
\end{equation}
Using \eqref{sec5.8}, \eqref{sec5.815}, and \eqref{nxc2}, and recalling that we may assume that $\min(k_1,k_2)\geq -D$, we have
\begin{equation*}
\begin{split}
&\|\mathbf{1}_{[0,C2^{4\max(k_1,k_2)}\delta]}(\eta-p^{\mu,\nu}(\xi))\cdot \widehat{g^2_{k_2,j_2}}(\eta)\|_{L^1_\eta}\\
&\lesssim (2^{\alpha k_2}+2^{10k_2})^{-1}\min\left[2^{-(1+\beta)j_2}\cdot \delta^{3/2}2^{6\max(k_1,k_2)},\delta^32^{12\max(k_1,k_2)}\right].
\end{split}
\end{equation*}
Using \eqref{mk15.55} , we have $\|\widehat{g^1_{k_1,j_1}}\|_{L^\infty}\lesssim 2^{-10k_1}$. Therefore, we may estimate the right-hand side of \eqref{nxc8} by
\begin{equation*}
C\min(2^{-(1+\beta)j_2}\kappa^{3/2},2^{2\max(k_1,k_2)}\delta^3)\lesssim (1+s)^{-1-\beta}.
\end{equation*}
The desired bound \eqref{nxc4} follows, using also \eqref{nxc7} and the definition of the set $J_2$.
\end{proof}

\subsection{Proof of Proposition \ref{reduced1}} We will prove the key bound \eqref{nh3} in several steps. 
The main ingredients in the proof are the estimates \eqref{nh5.5}--\eqref{ok50} above. 

This proof constitutes the heart of the analysis. We proceed in three different times. Decomposing the solutions into atoms decomposes each interaction into a myriad of different ``elementary interactions''. The purpose of the first simplification is to get rid of most of the easier cases so as to only focus on the fewer that really affect the outcome. This reduces matters to proving Proposition \ref{reduced2} below, after which it suffices to bound each iteration independently in a uniform way, see \eqref{ok60}. In a second time, we reduce matters further to the core of the difficulty in Proposition \ref{reduced3}. This is done in Lemma \ref{BigBound3}, Lemma \ref{BigBound4} and Lemma \ref{BigBound5} by using in various ways the finite speed of propagation which morally forces the time to be the largest parameter in all the relevant interactions, and in Lemma \ref{BigBound6} and Lemma \ref{BigBound6.5} which use the absence of (time) resonances at $(0,0)$ or at infinity provided by the first condition in \eqref{bc}. The proof of Proposition \ref{reduced3} is harder and we postpone an explanation of its ingredients to after its statement.

In this subsection we start by considering some of the easier cases, and reduce matters to proving Proposition \ref{reduced2} below. 
In all the cases analyzed in this subsection we can in fact control the stronger norm $B^1_{k,j}$, see Definition \ref{MainDef}, instead of the required $B_{k,j}$ norm.

\begin{lemma}\label{BigBound1}
With $D=D(d,A,d')$ sufficiently large as in subsection \ref{MainSec2}, the estimate 
\begin{equation}\label{nh7.5}
\sum_{(k_1,j_1),(k_2,j_2)\in\mathcal{J}}(1+2^{k_1}+2^{k_2})\big\|\widetilde{\varphi}^{(k)}_j\cdot P_kT_{m}^{\sigma;\mu,\nu}(f^\mu_{k_1,j_1},f^\nu_{k_2,j_2})\big\|_{B^1_{k,j}}\lesssim 2^{-\beta^4 m}
\end{equation}
holds if
\begin{equation}\label{nh7}
j\leq \beta m/2+N'_0k_++D^2,\qquad\text{ where }\qquad N'_0:=2N_0/3-10.
\end{equation}
\end{lemma}

\begin{proof}[Proof of Lemma \ref{BigBound1}] We observe that, in view of Definition \ref{MainDef},
\begin{equation}\label{ok1}
\|\phii^{(k)}_j\cdot P_kh\|_{B_{k,j}}\lesssim (2^{\alpha k}+2^{10 k})\cdot 2^{3j/2}2^{(1/2-\beta)\widetilde{k}}\|\phii^{(k)}_j\cdot P_kh\|_{L^2}.
\end{equation}
Therefore, it suffices to prove that
\begin{equation}\label{nh8}
\sum_{(k_1,j_1),(k_2,j_2)\in\mathcal{J}}(1+2^{k_1}+2^{k_2})(2^{\alpha k}+2^{10 k})2^{3j/2}2^{(1/2-\beta)\widetilde{k}}\big\|P_kT_{m}^{\sigma;\mu,\nu}(f_{k_1,j_1}^\mu,f_{k_2,j_2}^\nu)\big\|_{L^2}\lesssim 2^{-\beta^4 m}.
\end{equation}

Recalling the definition \eqref{nh5.6}, it is easy to see that
\begin{equation*}
\mathcal{F}\big[P_kT_{m}^{\sigma;\mu,\nu}(f_{k_1,j_1}^\mu,f_{k_2,j_2}^\nu)\big](\xi)=\int_{\mathbb{R}}\int_{\mathbb{R}^3}\varphi_k(\xi)e^{is\Lambda_\sigma(\xi)}q_m(s)\widehat{Ef_{k_1,j_1}^\mu}(\xi-\eta,s)\widehat{Ef^\nu_{k_2,j_2}}(\eta,s)\,d\eta ds.
\end{equation*}
Therefore, using \eqref{mk6},
\begin{equation}\label{nh9.5}
\begin{split}
&\big\|P_kT_{m}^{\sigma;\mu,\nu}(f_{k_1,j_1}^\mu,f_{k_2,j_2}^\nu)\big\|_{L^2}\\
&\lesssim \min\Big(\int_{\mathbb{R}}q_m(s)\|Ef_{k_1,j_1}^\mu(s)\|_{L^2}\|Ef_{k_2,j_2}^\nu(s)\|_{L^\infty}\,ds,\int_{\mathbb{R}}q_m(s)\|Ef_{k_1,j_1}^\mu(s)\|_{L^\infty}\|Ef_{k_2,j_2}^\nu(s)\|_{L^2}\,ds\Big).
\end{split}
\end{equation}
Therefore, using \eqref{nh9} and recalling the properties of the functions $q_m$ (see \eqref{nh2}), 
\begin{equation}\label{nh10}
\sum_{(k_1,k_2)\in\mathcal{X}_k,\,(k_1,j_1),(k_2,j_2)\in\mathcal{J}}(1+2^{k_1}+2^{k_2})\big\|P_kT_{m}^{\sigma;\mu,\nu}(f_{k_1,j_1}^\mu,f_{k_2,j_2}^\nu)\big\|_{L^2}\lesssim 2^{-(N_0-4)k_+}2^{-\beta m}.
\end{equation}

It follows that the left-hand side of \eqref{nh8} is dominated by
\begin{equation*}
2^{-\beta m}2^{(1/2-\beta+\alpha)k}2^{3j/2}
\end{equation*}
when $k\leq 0$, and by
\begin{equation*}
2^{-(N_0-15)k}2^{-\beta m}2^{3j/2}
\end{equation*}
when $k\geq 0$. The bound \eqref{nh8} follows if $j\leq \beta m/2+(2N_0/3-10)k_++D^2$, as desired.
\end{proof}

\begin{lemma}\label{BigBound2}
Assume that
\begin{equation}\label{nh30}
j\geq \beta m/2+N'_0k_++D^2.
\end{equation}
Then, with the same notation as before,
\begin{equation}\label{nh31}
\sum_{(k_1,j_1),(k_2,j_2)\in\mathcal{J},\,\,\max(k_1,k_2)\geq j/N'_0}(1+2^{k_1}+2^{k_2})\big\|\widetilde{\varphi}^{(k)}_j\cdot P_kT_{m}^{\sigma;\mu,\nu}(f_{k_1,j_1}^\mu,f_{k_2,j_2}^\nu)\big\|_{B^1_{k,j}}\lesssim 2^{-\beta^4 m},
\end{equation}
\begin{equation}\label{nh30.7}
\sum_{(k_1,j_1),(k_2,j_2)\in\mathcal{J},\,\min(k_1,k_2)\leq -10j}(1+2^{k_1}+2^{k_2})\big\|\widetilde{\varphi}^{(k)}_j\cdot P_kT_{m}^{\sigma;\mu,\nu}(f_{k_1,j_1}^\mu,f_{k_2,j_2}^\nu)\big\|_{B^1_{k,j}}\lesssim 2^{-\beta^4 m},
\end{equation}
and
\begin{equation}\label{nh30.8}
\sum_{(k_1,j_1),(k_2,j_2)\in\mathcal{J},\,\max(j_1,j_2)\geq 10j}(1+2^{k_1}+2^{k_2})\big\|\widetilde{\varphi}^{(k)}_j\cdot P_kT_{m}^{\sigma;\mu,\nu}(f_{k_1,j_1}^\mu,f_{k_2,j_2}^\nu)\big\|_{B^1_{k,j}}\lesssim 2^{-\beta^4 m}.
\end{equation}
\end{lemma}

\begin{proof}[Proof of Lemma \ref{BigBound2}] Notice that if $(k_1,k_2)\in\mathcal{X}_k$, $\max(k_1,k_2)\geq j/N'_0$, and $j\geq N'_0k_++D^2$ (see \eqref{nh30}) then $|k_1-k_2|\leq 4$. Therefore, using \eqref{ok1}, \eqref{nh9}, and \eqref{nh9.5}, left-hand side of \eqref{nh31} is dominated by
\begin{equation*}
\begin{split}
\sum_{(k_1,j_1),(k_2,j_2)\in\mathcal{J},\,\,\max(k_1,k_2)\geq j/N'_0}&2^{\max(k_1,k_2,0)}(2^{\alpha k}+2^{10 k})2^{3j/2}2^{(1/2-\beta)\widetilde{k}}\big\|P_kT_{m}^{\sigma;\mu,\nu}(f_{k_1,j_1}^\mu,f_{k_2,j_2}^\nu)\big\|_{L^2}\\
&\lesssim 2^{-\beta m}2^{-N_0j/(2N_0/3-10)}\cdot(2^{\alpha k}+2^{10 k})2^{3j/2}2^{(1/2-\beta)\widetilde{k}},
\end{split}
\end{equation*}
which clearly suffices, in view of \eqref{nh30}. Similarly, the left-hand side of \eqref{nh30.7} is dominated by
\begin{equation*}
\begin{split}
\sum_{(k_1,j_1),(k_2,j_2)\in\mathcal{J},\,\,\min(k_1,k_2)\leq -10j}(1+2^{k_1}+2^{k_2})&(2^{\alpha k}+2^{10 k})\cdot 2^{3j/2}2^{(1/2-\beta)\widetilde{k}}\big\|P_kT_{m}^{\sigma;\mu,\nu}(f_{k_1,j_1}^\mu,f_{k_2,j_2}^\nu)\big\|_{L^2}\\
&\lesssim 2^{-\beta m}2^{-3j}\cdot(2^{\alpha k}+2^{10 k})2^{3j/2}2^{(1/2-\beta)\widetilde{k}},
\end{split}
\end{equation*}
which clearly suffices. Finally, using the more precise bound \eqref{nh9.1}, the left-hand side of \eqref{nh30.8} is dominated by
\begin{equation*}
\begin{split}
\sum_{(k_1,j_1),(k_2,j_2)\in\mathcal{J},\,\,\max(j_1,j_2)\geq 10j}(1+2^{k_1}+2^{k_2})&(2^{\alpha k}+2^{10 k})\cdot 2^{3j/2}2^{(1/2-\beta)\widetilde{k}}\big\|P_kT_{m}^{\sigma;\mu,\nu}(f_{k_1,j_1}^\mu,f_{k_2,j_2}^\nu)\big\|_{L^2}\\
&\lesssim 2^{-\beta m}2^{-3j}\cdot(2^{\alpha k}+2^{10 k})2^{3j/2}2^{(1/2-\beta)\widetilde{k}},
\end{split}
\end{equation*}
which clearly suffices.
\end{proof}

We examine the conclusions of Lemma \ref{BigBound1} and Lemma \ref{BigBound2}, and notice that Proposition \ref{reduced1} follows from Proposition \ref{reduced2} below.

\begin{proposition}\label{reduced2}
With the same notation as in Proposition \ref{reduced1}, we have
\begin{equation}\label{ok60}
(1+2^{k_1}+2^{k_2})\big\|\widetilde{\varphi}^{(k)}_j\cdot P_kT_{m}^{\sigma;\mu,\nu}(f_{k_1,j_1}^\mu,f_{k_2,j_2}^\nu)\big\|_{B_{k,j}}\lesssim 2^{-\beta^4 (m+j)},
\end{equation}
for any fixed $\mu,\nu\in\mathcal{I}_d$, $(k,j),(k_1,j_1),(k_2,j_2)\in\mathcal{J}$, and $m\in[0,L+1]\cap\mathbb{Z}$, satisfying
\begin{equation}\label{ok61}
j\geq \beta m/2+N'_0k_++D^2,\qquad -10j\leq k_1,\,k_2\leq j/N'_0,\qquad \max(j_1,j_2)\leq 10j.
\end{equation}
\end{proposition}

\subsection{Proof of Proposition \ref{reduced2}} In this subsection we will show that proving Proposition \ref{reduced2} can be further reduced 
to proving Proposition \ref{reduced3} below. The arguments are more complicated than before, and we need to examine our bilinear operators 
more carefully; however, in all cases discussed in this subsection we can still control the stronger $B^1_{k,j}$ norms. 

We notice that we are looking to prove the bound \eqref{ok60} for {\it{fixed}} $k,j,k_1,j_1,k_2,j_2,m$. 
We will consider several cases, depending on the relative sizes of these parameters. 

\begin{lemma}\label{BigBound3}
The bound \eqref{ok60} holds provided that \eqref{ok61} holds and, in addition,
\begin{equation}\label{nh32.5}
j\geq \max(m+\max(\widetilde{k_1},\widetilde{k_2})+D,-k(1+\beta^2)+D).
\end{equation}
\end{lemma}

\begin{proof}[Proof of Lemma \ref{BigBound3}] Using definition \eqref{sec5.3}, it suffices to prove that
\begin{equation}\label{ok30}
\begin{split}
&(1+2^{k_1}+2^{k_2})(2^{\alpha k}+2^{10k})\cdot 2^{(1+\beta)j}\big\|\widetilde{\varphi}^{(k)}_j\cdot P_kT_{m}^{\sigma;\mu,\nu}(f_{k_1,j_1}^\mu,f_{k_2,j_2}^\nu)\big\|_{L^2}\\
&+(1+2^{k_1}+2^{k_2})(2^{\alpha k}+2^{10k})\cdot 2^{(1/2-\beta)\widetilde{k}}\big\|\mathcal{F}[\widetilde{\varphi}^{(k)}_j\cdot P_kT_{m}^{\sigma;\mu,\nu}(f_{k_1,j_1}^\mu,f_{k_2,j_2}^\nu)]\big\|_{L^\infty}\lesssim 2^{-\beta^4 (m+j)}.
\end{split}
\end{equation}

Assume first that
\begin{equation}\label{ok31}
\min(j_1,j_2)\leq (1-\beta^2)j.
\end{equation}
By symmetry, we may assume that $j_1\leq (1-\beta^2)j$ and write
\begin{equation*}
\begin{split}
&\widetilde{\varphi}^{(k)}_j(x)\cdot P_kT_{m}^{\sigma;\mu,\nu}(f_{k_1,j_1}^\mu,f_{k_2,j_2}^\nu)(x)\\
&=c\widetilde{\varphi}^{(k)}_j(x)\int_{\mathbb{R}^3}\int_{\mathbb{R}}\int_{\mathbb{R}^3}\varphi_k(\xi)e^{ix\cdot\xi}e^{is[\Lambda_\sigma(\xi)-\widetilde{\Lambda}_{\mu}(\xi-\eta)-\widetilde{\Lambda}_{\nu}(\eta)]}q_m(s)\cdot \widehat{f_{k_1,j_1}^\mu}(\xi-\eta,s)\widehat{f_{k_2,j_2}^\nu}(\eta,s)\,d\eta ds d\xi.
\end{split}
\end{equation*}
We examine the integral in $\xi$ in the formula above. We recall the assumptions \eqref{ok61}, \eqref{nh32.5}, and \eqref{ok31}, and the last bound in \eqref{nh9}. Notice that, using only the assumption \eqref{nh32.5} and the definition \eqref{lambdas},
\begin{equation*}
\Big|\nabla_\xi\big[x\cdot\xi+s[\Lambda_\sigma(\xi)-\widetilde{\Lambda}_{\mu}(\xi-\eta)-\widetilde{\Lambda}_{\nu}(\eta)]\big]\Big|\geq |x|-s\big|\nabla_\xi[\Lambda_\sigma(\xi)-\widetilde{\Lambda}_{\mu}(\xi-\eta)\big]\big|\geq 2^{j-10},
\end{equation*}
as long as $|\xi|+|\xi-\eta|\leq 2^{\max(k_1,k_2)+10}$.
We apply Lemma \ref{tech5} (with $K\approx 2^j$, $\eps\approx 2^{-j_1}$) to conclude that
\begin{equation*}
\big|\widetilde{\varphi}^{(k)}_j(x)\cdot P_kT_{m}^{\sigma;\mu,\nu}(f_{k_1,j_1}^\mu,f_{k_2,j_2}^\nu)(x)\big|\lesssim 2^{-10j}|\widetilde{\varphi}^{(k)}_j(x)|,
\end{equation*}
and the desired bounds \eqref{ok30} follow easily.

Assume now that 
\begin{equation}\label{ok33}
\min(j_1,j_2)\geq (1-\beta^2)j.
\end{equation}
By symmetry, we may assume that $k_1\leq k_2$. We prove first the bound on the second term in the left-hand side of \eqref{ok30}: using \eqref{nh9.1} we estimate
\begin{equation*}
\begin{split}
&(1+2^{k_1}+2^{k_2})(2^{\alpha k}+2^{10k})\cdot 2^{(1/2-\beta)\widetilde{k}}\|\mathcal{F}[\widetilde{\varphi}^{(k)}_j\cdot P_kT_{m}^{\sigma;\mu,\nu}(f_{k_1,j_1}^\mu,f_{k_2,j_2}^\nu)]\big\|_{L^\infty}\\
&\lesssim (2^{k_2}+1)(2^{\alpha k}+2^{10k})2^{(1/2-\beta)\widetilde{k}}\cdot 2^m\sup_{s\in[2^{m-1},2^{m+1}]}\|f_{k_1,j_1}^\mu(s)\|_{L^2}\|f_{k_2,j_2}^\nu(s)\|_{L^2}\\
&\lesssim (2^{k_2}+1)(2^{\alpha k}+2^{10k})2^{(1/2-\beta)\widetilde{k}}2^{j-\widetilde{k_2}}\cdot(2^{\alpha k_1}+2^{10k_1})^{-1}2^{2\beta\widetilde{k_1}}2^{-(1-\be)j_1}\cdot(2^{\alpha k_2}+2^{10k_2})^{-1}2^{2\beta\widetilde{k_2}}2^{-(1-\be)j_2}\\
&\lesssim (2^{k_2}+1)2^j2^{-(1/2+\beta)\widetilde{k_2}}\cdot 2^{-\alpha k_1}\min(2^{(1+\beta)k_1},2^{-(1-\beta-\beta^2)j})\cdot 2^{-(1-\be-\beta^2)j}.
\end{split}
\end{equation*}
This suffices to prove the desired bound in \eqref{ok30}, as it can be easily seen by considering the cases $k_1\leq -j$ and $k_1\geq -j$.

Some more care is needed to prove the bound on the first term in the left-hand side of \eqref{ok30}. We recall that
\begin{equation*}
f_{k_1,j_1}^\mu=P_{[k_1-2,k_1+2]}(\phii_{j_1}^{(k_1)}\cdot P_{k_1}f_\mu)\quad\text{ and }\quad f_{k_2,j_2}^\nu=P_{[k_2-2,k_2+2]}(\phii_{j_2}^{(k_2)}\cdot P_{k_2}f_\nu).
\end{equation*}
Since $\|\phii_{j_1}^{(k_1)}\cdot P_{k_1}f_\mu(s)\|_{B_{k_1,j_1}}+\|\phii_{j_2}^{(k_2)}\cdot P_{k_2}f_\mu(s)\|_{B_{k_2,j_2}}\lesssim 1$, see \eqref{nh5.5}, we use \eqref{sec5.8}--\eqref{sec5.82} to decompose
\begin{equation}\label{ok36}
\begin{split}
&\phii^{(k_1)}_{j_1}\cdot P_{k_1}f_\mu(s)=(2^{\alpha k_1}+2^{10k_1})^{-1}[g_{k_1,j_1}^\mu(s)+h_{k_1,j_1}^\mu(s)],\\
&g_{k_1,j_1}^\mu(s)=g_{k_1,j_1}^\mu(s)\cdot \widetilde{\varphi}^{(k_1)}_{[j_1-2,j_1+2]},\qquad h_{k_1,j_1}^\mu(s)=h_{k_1,j_1}^\mu(s)\cdot \widetilde{\varphi}^{(k_1)}_{[j_1-2,j_1+2]},\\
&2^{(1+\be)j_1}\|g_{k_1,j_1}^\mu(s)\|_{L^2}+2^{(1/2-\beta)\widetilde{k_1}}\|\widehat{g_{k_1,j_1}^\mu}(s)\|_{L^\infty}\lesssim 1,\\
&2^{-2\beta\widetilde{k_1}}2^{(1-\be)j_1}\|h_{k_1,j_1}^\mu(s)\|_{L^2}+2^{(1/2-\beta)\widetilde{k_1}}\|\widehat{h_{k_1,j_1}^\mu}(s)\|_{L^\infty}+2^{(\gamma-\beta-5/2)\widetilde{k_1}}2^{\gamma j_1}\|\widehat{h_{k_1,j_1}^\mu}(s)\|_{L^1}\lesssim 1,
\end{split}
\end{equation}
and
\begin{equation}\label{ok37}
\begin{split}
&\phii^{(k_2)}_{j_2}\cdot P_{k_2}f_\nu(s)=(2^{\alpha k_2}+2^{10k_2})^{-1}[g_{k_2,j_2}^\nu(s)+h_{k_2,j_2}^\nu(s)],\\
&g_{k_2,j_2}^\nu(s)=g_{k_2,j_2}^\nu(s)\cdot \widetilde{\varphi}^{(k_2)}_{[j_2-2,j_2+2]},\qquad h_{k_2,j_2}^\nu(s)=h_{k_2,j_2}^\nu(s)\cdot \widetilde{\varphi}^{(k_2)}_{[j_2-2,j_2+2]},\\
&2^{(1+\be)j_2}\|g_{k_2,j_2}^\nu(s)\|_{L^2}+2^{(1/2-\beta)\widetilde{k_2}}\|\widehat{g_{k_2,j_2}^\nu}(s)\|_{L^\infty}\lesssim 1,\\
&2^{-2\beta\widetilde{k_2}}2^{(1-\be)j_2}\|h_{k_2,j_2}^\nu(s)\|_{L^2}+2^{(1/2-\beta)\widetilde{k_2}}\|\widehat{h_{k_2,j_2}^\nu}(s)\|_{L^\infty}+2^{(\gamma-\beta-5/2)\widetilde{k_2}}2^{\gamma j_2}\|\widehat{h_{k_2,j_2}^\nu}(s)\|_{L^1}\lesssim 1.
\end{split}
\end{equation}
Using these decompositions and recalling the definition \eqref{nh9.2}, to prove the desired bound on the first term in the left-hand side of \eqref{ok30}, it suffices to prove that for any $s\in[2^{m-1},2^{m+1}]$
\begin{equation}\label{ok38}
\begin{split}
(1+2^{k_1}+2^{k_2})&(2^{\alpha k}+2^{10k})2^{(1+\beta)j}\cdot(2^{\alpha k_1}+2^{10k_1})^{-1}(2^{\alpha k_2}+2^{10k_2})^{-1}2^m\\
\Big[&\big\|\widetilde{\varphi}^{(k)}_j\cdot P_k\widetilde{T}_{s}^{\sigma;\mu,\nu}(P_{[k_1-2,k_1+2]}g_{k_1,j_1}^\mu(s),P_{[k_2-2,k_2+2]}g_{k_2,j_2}^\nu(s))\big\|_{L^2}\\
+&\big\|\widetilde{\varphi}^{(k)}_j\cdot P_k\widetilde{T}_{s}^{\sigma;\mu,\nu}(P_{[k_1-2,k_1+2]}g_{k_1,j_1}^\mu(s),P_{[k_2-2,k_2+2]}h_{k_2,j_2}^\nu(s))\big\|_{L^2}\\
+&\big\|\widetilde{\varphi}^{(k)}_j\cdot P_k\widetilde{T}_{s}^{\sigma;\mu,\nu}(P_{[k_1-2,k_1+2]}h_{k_1,j_1}^\mu(s),P_{[k_2-2,k_2+2]}g_{k_2,j_2}^\nu(s))\big\|_{L^2}\\
+&\big\|\widetilde{\varphi}^{(k)}_j\cdot P_k\widetilde{T}_{s}^{\sigma;\mu,\nu}(P_{[k_1-2,k_1+2]}h_{k_1,j_1}^\mu(s),P_{[k_2-2,k_2+2]}h_{k_2,j_2}^\nu(s))\big\|_{L^2}\Big]\lesssim 2^{-\beta^4(m+j)}.
\end{split}
\end{equation}

Recall that we assumed $k_1\leq k_2$; therefore we may also assume that $k\leq k_2+4$. Using \eqref{ok36}--\eqref{ok37} and recalling \eqref{ok33}, we estimate
\begin{equation*}
\begin{split}
\big\|P_k\widetilde{T}_{s}^{\sigma;\mu,\nu}(P_{[k_1-2,k_1+2]}g_{k_1,j_1}^\mu(s),P_{[k_2-2,k_2+2]}g_{k_2,j_2}^\nu(s))\big\|_{L^2}&\lesssim \|\mathcal{F}(P_{[k_1-2,k_1+2]}g_{k_1,j_1}^\mu)(s)\|_{L^1}\|g_{k_2,j_2}^\nu(s)\|_{L^2}\\
&\lesssim 2^{3k_1/2}2^{-(1+\beta)j_1}2^{-(1+\beta)j_2}\\
&\lesssim 2^{3k_1/2}2^{-(2+2\beta)(1-\beta^2)j},
\end{split}
\end{equation*}
\begin{equation*}
\begin{split}
\big\|P_k\widetilde{T}_{s}^{\sigma;\mu,\nu}(P_{[k_1-2,k_1+2]}h_{k_1,j_1}^\mu(s),P_{[k_2-2,k_2+2]}h_{k_2,j_2}^\nu(s))\big\|_{L^2}&\lesssim \|\widehat{h_{k_1,j_1}^\mu}(s)\|_{L^1}\|\widehat{h_{k_2,j_2}^\nu}(s)\|_{L^2}\\
&\lesssim 2^{-\gamma j_1}2^{(5/2+\beta-\gamma)\widetilde{k_1}}2^{-(1-\beta)j_2}2^{2\beta\widetilde{k_2}}\\
&\lesssim 2^{(3/2-2\beta)\widetilde{k_1}}2^{2\beta\widetilde{k_2}}2^{-(2+2\beta)(1-\beta^2)j},
\end{split}
\end{equation*}
\begin{equation*}
\begin{split}
\big\|P_k\widetilde{T}_{s}^{\sigma;\mu,\nu}(P_{[k_1-2,k_1+2]}h_{k_1,j_1}^\mu(s),P_{[k_2-2,k_2+2]}g_{k_2,j_2}^\nu(s)\big\|_{L^2}&\lesssim \|\widehat{h_{k_1,j_1}^\mu}(s)\|_{L^1}\|\widehat{g_{k_2,j_2}^\nu}(s)\|_{L^2}\\
&\lesssim 2^{-\gamma j_1}2^{(5/2+\beta-\gamma)\widetilde{k_1}}2^{-(1+\beta)j_2}\\
&\lesssim 2^{3\widetilde{k_1}/2}2^{-(2+2\beta)(1-\beta^2)j},
\end{split}
\end{equation*}
and
\begin{equation*}
\begin{split}
\big\|P_k\widetilde{T}_{s}^{\sigma;\mu,\nu}&(P_{[k_1-2,k_1+2]}g_{k_1,j_1}^\mu(s),P_{[k_2-2,k_2+2]}h_{k_2,j_2}^\nu(s))\big\|_{L^2}\\
&\lesssim \min\big(2^{3k_1/2}\|\widehat{g_{k_1,j_1}^\mu}(s)\|_{L^2}\|\widehat{h_{k_2,j_2}^\nu}(s)\|_{L^2},\|\widehat{g_{k_1,j_1}^\mu}(s)\|_{L^2}\|\widehat{h_{k_2,j_2}^\nu}(s)\|_{L^1}\big)\\
&\lesssim 2^{-(1+\beta)j_1}\min\big(2^{-(1-\beta)j_2}2^{2\beta\widetilde{k_2}}2^{3k_1/2},2^{-\gamma j_2}2^{(5/2+\beta-\gamma)\widetilde{k_2}}\big)\\
&\lesssim 2^{-(1+\beta)j_1}2^{-(1+\beta)j_2}2^{3\widetilde{k_2}/2}\min\big(2^{2\beta(j_2+\widetilde{k_2})}2^{3(k_1-\widetilde{k_2})/2},2^{(1+\beta-\gamma)(j_2+\widetilde{k_2})}\big)\\
&\lesssim 2^{-(2+2\beta)(1-\beta^2)j}2^{3k_1/4}2^{3\widetilde{k_2}/4}.
\end{split}
\end{equation*}
Therefore, since $2^m\lesssim 2^{j-\widetilde{k_2}}$ and $(2^{\alpha k}+2^{10k})(2^{\alpha k_2}+2^{10k_2})^{-1}\lesssim 1$, the left-hand side of \eqref{ok38} is dominated by
\begin{equation*}
\begin{split}
(1+2^{k_1}+2^{k_2})&2^{(1+\beta)j}\cdot(2^{\alpha k_1}+2^{10k_1})^{-1}2^{j-\widetilde{k_2}}\cdot 2^{-(2+2\beta)(1-\beta^2)j}(2^{3k_1/2}+2^{3k_1/4}2^{3\widetilde{k_2}/4})\\
&\lesssim 2^{-2\beta j/3}(2^{k_2}+1),
\end{split}
\end{equation*}
which suffices since $2^{k_2}\lesssim 2^{j/N'_0}$. This completes the proof of the lemma.
\end{proof}

\begin{lemma}\label{BigBound4}
The bound \eqref{ok60} holds provided that \eqref{ok61} holds and, in addition,
\begin{equation}\label{ok40}
m+\max(\widetilde{k_1},\widetilde{k_2})+D\leq j\leq -k(1+\beta^2)+D.
\end{equation}
\end{lemma}

\begin{proof}[Proof of Lemma \ref{BigBound4}]
In view of the restrictions \eqref{ok40} and \eqref{ok61}, we may assume that $k\leq -D^2/2$. Using the definition, it is easy to see that
\begin{equation*}
\|\widetilde{\varphi}^{(k)}_j\cdot P_kh\|_{B_{k,j}}\lesssim (2^{\alpha k}+2^{10 k})2^{(1+\beta)j}2^{3k/2}\|\widehat{P_kh}\|_{L^\infty}.
\end{equation*}
Therefore, it suffices to prove that
\begin{equation}\label{ok43}
(1+2^{k_1}+2^{k_2})2^{\alpha k}2^{(1+\beta)j}2^{3k/2}\big\|\mathcal{F}P_kT_{m}^{\sigma;\mu,\nu}(f_{k_1,j_1}^\mu,f_{k_2,j_2}^\nu)\big\|_{L^\infty}\lesssim 2^{-\beta^4 (m+j)}.
\end{equation}

Recall the definition 
\begin{equation}\label{ok40.5}
\mathcal{F}P_kT_{m}^{\sigma;\mu,\nu}(f_{k_1,j_1}^\mu,f_{k_2,j_2}^\nu)(\xi)=\varphi_k(\xi)\int_{\mathbb{R}}\int_{\mathbb{R}^3}e^{is\Phi^{\sigma;\mu,\nu}(\xi,\eta)}q_m(s)\widehat{f_{k_1,j_1}^\mu}(\xi-\eta,s)\widehat{f_{k_2,j_2}^\nu}(\eta,s)\,d\eta ds,
\end{equation}
where
\begin{equation}\label{ok40.6}
\Phi^{\sigma;\mu,\nu}(\xi,\eta)=\Lambda_\sigma(\xi)-\widetilde{\Lambda}_{\mu}(\xi-\eta)-\widetilde{\Lambda}_{\nu}(\eta).
\end{equation}
Using \eqref{nh9.1} and recalling that $\alpha\leq 2\beta$, it follows that
\begin{equation*}
\begin{split}
\big\|\mathcal{F}P_kT_{m}^{\sigma;\mu,\nu}&(f_{k_1,j_1}^\mu,f_{k_2,j_2}^\nu)\big\|_{L^\infty}\lesssim \int_{\mathbb{R}}q_m(s)\|f_{k_1,j_1}^\mu(s)\|_{L^2}\|f_{k_2,j_2}^\nu(s)\|_{L^2}\,ds\\
&\lesssim \|q_m\|_{L^1(\mathbb{R})}(2^{\alpha k_1}+2^{10k_1})^{-1}2^{2\beta\widetilde{k_1}}2^{-(1-\beta)j_1}\cdot (2^{\alpha k_2}+2^{10k_2})^{-1}2^{2\beta\widetilde{k_2}}2^{-(1-\beta)j_2}\\
&\lesssim \|q_m\|_{L^1(\mathbb{R})}\min(1,2^{-5k_1})2^{-(1-\beta)j_1}\cdot \min(1,2^{-5k_2})2^{-(1-\beta)j_2}.
\end{split}
\end{equation*}
Recalling the definitions \eqref{sec5.6} and the assumptions, the desired bound \eqref{ok43} follows if
\begin{equation*}
m=L+1\,\,\,\text{ or }\,\,\,m\leq (1-\beta)(j_1+j_2)-(1/2-\beta)k. 
\end{equation*}

It remains to prove the bound \eqref{ok43} in the case
\begin{equation}\label{ok45}
m\in[1,L]\cap\mathbb{Z}\,\,\,\text{ and }\,\,\,m\geq -(1/2-\beta)k+(1-\beta)(j_1+j_2).
\end{equation}
Since $j_1+k_1\geq 0$, $j_2+k_2\geq 0$, and $k\leq -D^2/2$, the conditions \eqref{ok40} and \eqref{ok45} show that $k_1,k_2\geq k+10$. In particular, we may assume that $|k_1-k_2|\leq 10$. Using also \eqref{ok40}, for \eqref{ok43} it suffices to prove that, assuming \eqref{ok45},
\begin{equation}\label{ok46}
(1+2^{k_2})\big\|\mathcal{F}P_kT_{m}^{\sigma;\mu,\nu}(f_{k_1,j_1}^\mu,f_{k_2,j_2}^\nu)\big\|_{L^\infty}\lesssim 2^{-k(1/2+\alpha-\beta-2\beta^2)}.
\end{equation}

To prove \eqref{ok46} we would like to integrate by parts in $\eta$ and $s$ in the formula \eqref{ok40.5}. Recall the definitions \eqref{ok40.5} and \eqref{ok40.6}, and decompose
\begin{equation*}
\begin{split}
&\mathcal{F}P_kT_{m}^{\sigma;\mu,\nu}(f_{k_1,j_1}^\mu,f_{k_2,j_2}^\nu)(\xi)=G(\xi)+H(\xi),\\
&G(\xi):=\varphi_k(\xi)\int_{\mathbb{R}}\int_{\mathbb{R}^3}e^{is\Phi^{\sigma;\mu,\nu}(\xi,\eta)}\varphi(2^D\Phi^{\sigma;\mu,\nu}(\xi,\eta))q_m(s)\widehat{f_{k_1,j_1}^\mu}(\xi-\eta,s)\widehat{f_{k_2,j_2}^\nu}(\eta,s)\,d\eta ds,\\
&H(\xi):=\varphi_k(\xi)\int_{\mathbb{R}}\int_{\mathbb{R}^3}e^{is\Phi^{\sigma;\mu,\nu}(\xi,\eta)}[1-\varphi(2^D\Phi^{\sigma;\mu,\nu}(\xi,\eta))]q_m(s)\widehat{f_{k_1,j_1}^\mu}(\xi-\eta,s)\widehat{f_{k_2,j_2}^\nu}(\eta,s)\,d\eta ds.
\end{split}
\end{equation*}
The function $H$ can be estimated using integration by parts in $s$, Lemma \ref{ders}, the assumptions \eqref{nh2}, and the bounds \eqref{nh9.1}. Indeed,
\begin{equation*}
\begin{split}
|H(\xi)|&\lesssim \sup_{s\in [2^{m-1},2^{m+1}]}\big[\big\|\widehat{f_{k_1,j_1}^\mu}(s)\big\|_{L^2}\big\|\widehat{f_{k_2,j_2}^\nu}(s)\big\|_{L^2}\\
&+2^m\big\|(\partial_s\widehat{f_{k_1,j_1}^\mu})(s)\big\|_{L^2}\big\|\widehat{f_{k_2,j_2}^\nu}(s)\big\|_{L^2}+2^m\big\|\widehat{f_{k_1,j_1}^\mu}(s)\big\|_{L^2}\big\|(\partial_s\widehat{f_{k_2,j_2}^\nu})(s)\big\|_{L^2}\big]\\
&\lesssim \min(1,2^{-(N_0-5)k_2}).
\end{split}
\end{equation*}
Therefore, for \eqref{ok46} it suffices to prove that
\begin{equation}\label{ok48}
(1+2^{k_2})\big\|G\big\|_{L^\infty}\lesssim 2^{-k(1/2+\alpha-\beta-2\beta^2)}.
\end{equation}

Recalling the definitions \eqref{lambdas} and \eqref{ra2},
\begin{equation}\label{ok47}
\Xi^{\mu,\nu}(\xi,\eta)=(\nabla_\eta\Phi^{\sigma;\mu,\nu})(\xi,\eta)=-\iota_1\frac{c^2_{\sigma_1}(\eta-\xi)}{(b^2_{\sigma_1}+c^2_{\sigma_1}|\eta-\xi|^2)^{1/2}}
-\iota_2\frac{c^2_{\sigma_2}\eta}{(b^2_{\sigma_2}+c^2_{\sigma_2}|\eta|^2)^{1/2}},
\end{equation}
where
\begin{equation*}
\mu=(\sigma_1\iota_1),\qquad \mu=(\sigma_2\iota_2),\qquad \sigma_1,\sigma_2\in\{1,\ldots,d\},\qquad \iota_1,\iota_2\in\{+,-\}.
\end{equation*}
In view of the first assumption in \eqref{abcond}, we may assume that
\begin{equation}\label{ok49}
k_1,k_2\geq -D/10,
\end{equation}
since otherwise $G=0$. For $l\in\mathbb{Z}$ let
\begin{equation}\label{ba1}
\begin{split}
G_{\leq l}(\xi):=\varphi_k(\xi)&\int_{\mathbb{R}}\int_{\mathbb{R}^3}\varphi_{(-\infty,l]}(\Xi^{\mu,\nu}(\xi,\eta))\cdot  e^{is\Phi^{\sigma;\mu,\nu}(\xi,\eta)}\\
&\varphi(2^D\Phi^{\sigma;\mu,\nu}(\xi,\eta))q_m(s)\widehat{f_{k_1,j_1}^\mu}(\xi-\eta,s)\widehat{f_{k_2,j_2}^\nu}(\eta,s)\,d\eta ds.
\end{split}
\end{equation}
Let $G_l:=G_{\leq l}-G_{\leq l-1}$. In proving \eqref{ok48} we may assume that $j_1\leq j_2$. If $l\geq \max(j_2,m/2)-(1-\beta^2)m$ then we integrate by parts in $\eta$, using Lemma \ref{tech5} with $K\approx 2^{m+l}$ and $\eps\approx 2^{-j_2}$. Using also the last bound in \eqref{nh9} and recalling that $k_1,k_2\geq -D/10$, it follows that
\begin{equation}\label{ba2}
\sum_{l\geq l_0+1}\|G_l\|_{L^\infty}\lesssim 2^{-5k_2},\qquad\text{ where }l_0=\lfloor \max(j_2,m/2)-m+\beta^2m\rfloor.
\end{equation}
It remains to estimate $\|G_{\leq l_0}\|_{L^\infty}$. It follows from Lemma \ref{desc1} that $G_{\leq l_0}\equiv 0$, provided that
\begin{equation*}
2^{l_0+k_2}\leq 2^{-D/10}.
\end{equation*}
This last inequality is an easy algebraic consequence of the assumptions \eqref{ok61}, \eqref{ok40}, and \eqref{ok45}, which completes the proof of the lemma.
\end{proof}

\begin{lemma}\label{BigBound5}
The bound \eqref{ok60} holds provided that \eqref{ok61} holds and, in addition,
\begin{equation}\label{ok110}
j\leq m+\max(\widetilde{k_1},\widetilde{k_2})+D\quad\text{ and }\quad \max(j_1,j_2)\geq (1-\beta/10)(m+\max(\widetilde{k_1},\widetilde{k_2})).
\end{equation}
\end{lemma}

\begin{proof}[Proof of Lemma \ref{BigBound5}] Using definition \eqref{sec5.3}, it suffices to prove that
\begin{equation}\label{ok112}
\begin{split}
&(1+2^{k_1}+2^{k_2})(2^{\alpha k}+2^{10k})\cdot 2^{(1+\beta)j}\big\|\widetilde{\varphi}^{(k)}_j\cdot P_kT_{m}^{\sigma;\mu,\nu}(f_{k_1,j_1}^\mu,f_{k_2,j_2}^\nu)\big\|_{L^2}\\
&+(1+2^{k_1}+2^{k_2})(2^{\alpha k}+2^{10k})\cdot 2^{(1/2-\beta)\widetilde{k}}\big\|\mathcal{F}[\widetilde{\varphi}^{(k)}_j\cdot P_kT_{m}^{\sigma;\mu,\nu}(f_{k_1,j_1}^\mu,f_{k_2,j_2}^\nu)]\big\|_{L^\infty}\lesssim 2^{-\beta^4 (m+j)}.
\end{split}
\end{equation}
By symmetry, we may assume $k_1\leq k_2$.

We prove first the bounds \eqref{ok112} in the case
\begin{equation}\label{ba6}
k_1\leq-5m/6.
\end{equation}
Using \eqref{nh9}, for any $s\in[0,t]$,
\begin{equation*}
\|\widehat{f_{k_1,j_1}^\mu}(s)\|_{L^1}\lesssim 2^{3k_1}\|\widehat{f_{k_1,j_1}^\mu}(s)\|_{L^\infty}\lesssim 2^{(5/2-\alpha+\beta)k_1}.
\end{equation*}
Therefore, using \eqref{nh9} again, it follows that
\begin{equation*}
\begin{split}
\big\|\mathcal{F}[T_{m}^{\sigma;\mu,\nu}(f_{k_1,j_1}^\mu,f_{k_2,j_2}^\nu)]\big\|_{L^2}&\lesssim 2^m\sup_{s\in[2^{m-1},2^{m+1}]}\|\widehat{f_{k_1,j_1}^\mu}(s)\|_{L^1}\|\widehat{f_{k_2,j_2}^\nu}(s)\|_{L^2}\\
&\lesssim 2^m2^{(5/2-\alpha+\beta)k_1}\min(2^{-(N_0-1)k_2},2^{(1+\beta-\alpha)k_2})
\end{split}
\end{equation*}
and
\begin{equation}\label{ba6.5}
\begin{split}
\big\|\mathcal{F}[T_{m}^{\sigma;\mu,\nu}(f_{k_1,j_1}^\mu,f_{k_2,j_2}^\nu)\big\|_{L^\infty}&\lesssim 2^m\sup_{s\in[2^{m-1},2^{m+1}]}\|\widehat{f_{k_1,j_1}^\mu}(s)\|_{L^1}\|\widehat{f_{k_2,j_2}^\nu}(s)\|_{L^\infty}\\
&\lesssim 2^m2^{(5/2-\alpha+\beta)k_1}\cdot(2^{\alpha k_2}+2^{10k_2})^{-1}2^{-(1/2-\beta)\widetilde{k_2}}.
\end{split}
\end{equation}
Therefore, recalling \eqref{ba6}, if $k\leq 0$ then the left-hand side of \eqref{ok112} is dominated by
\begin{equation*}
C2^{(2+\beta)m}2^{(5/2-\alpha+\beta)k_1}\lesssim 2^{(-1/12+5\al/6+\be/6)m},
\end{equation*}
which suffices. Similarly, if $k\geq 0$ then the left-hand side of \eqref{ok112} is dominated by
\begin{equation*}
C2^{(2+\beta)m}2^{(5/2-\alpha+\beta)k_1}2^{-(N_0-15)k}+C2^{2k_2}2^m2^{(5/2-\alpha+\beta)k_1}\lesssim 2^{-10k}2^{(-1/12+5\al/6+\be/6)m},
\end{equation*}
which also suffices.

To prove the bound \eqref{ok112} when $-5m/6\leq k_1\leq k_2$ we decompose, as in \eqref{ok36}--\eqref{ok37}, for any $s\in[2^{m-1},2^{m+1}]$,
\begin{equation}\label{ok113}
\begin{split}
&\phii^{(k_1)}_{j_1}\cdot P_{k_1}f_\mu(s)=(2^{\alpha k_1}+2^{10k_1})^{-1}[g_{k_1,j_1}^\mu(s)+h_{k_1,j_1}^\mu(s)],\\
&g_{k_1,j_1}^\mu(s)=g_{k_1,j_1}^\mu(s)\cdot \widetilde{\varphi}^{(k_1)}_{[j_1-2,j_1+2]},\qquad h_{k_1,j_1}^\mu(s)=h_{k_1,j_1}^\mu(s)\cdot \widetilde{\varphi}^{(k_1)}_{[j_1-2,j_1+2]},\\
&2^{(1+\be)j_1}\|g_{k_1,j_1}^\mu(s)\|_{L^2}+2^{(1/2-\beta)\widetilde{k_1}}\|\widehat{g_{k_1,j_1}^\mu}(s)\|_{L^\infty}\lesssim 1,\\
&2^{-2\beta\widetilde{k_1}}2^{(1-\be)j_1}\|h_{k_1,j_1}^\mu(s)\|_{L^2}+2^{(1/2-\beta)\widetilde{k_1}}\|\widehat{h_{k_1,j_1}^\mu}(s)\|_{L^\infty}+2^{(\gamma-\beta-5/2)\widetilde{k_1}}2^{\gamma j_1}\|\widehat{h_{k_1,j_1}^\mu}(s)\|_{L^1}\lesssim 1,
\end{split}
\end{equation}
and
\begin{equation}\label{ok114}
\begin{split}
&\phii^{(k_2)}_{j_2}\cdot P_{k_2}f_\nu(s)=(2^{\alpha k_2}+2^{10k_2})^{-1}[g_{k_2,j_2}^\nu(s)+h_{k_2,j_2}^\nu(s)],\\
&g_{k_2,j_2}^\nu(s)=g_{k_2,j_2}^\nu(s)\cdot \widetilde{\varphi}^{(k_2)}_{[j_2-2,j_2+2]},\qquad h_{k_2,j_2}^\nu(s)=h_{k_2,j_2}^\nu(s)\cdot \widetilde{\varphi}^{(k_2)}_{[j_2-2,j_2+2]},\\
&2^{(1+\be)j_2}\|g_{k_2,j_2}^\nu(s)\|_{L^2}+2^{(1/2-\beta)\widetilde{k_2}}\|\widehat{g_{k_2,j_2}^\nu}(s)\|_{L^\infty}\lesssim 1,\\
&2^{-2\beta\widetilde{k_2}}2^{(1-\be)j_2}\|h_{k_2,j_2}^\nu(s)\|_{L^2}+2^{(1/2-\beta)\widetilde{k_2}}\|\widehat{h_{k_2,j_2}^\nu}(s)\|_{L^\infty}+2^{(\gamma-\beta-5/2)\widetilde{k_2}}2^{\gamma j_2}\|\widehat{h_{k_2,j_2}^\nu}(s)\|_{L^1}\lesssim 1.
\end{split}
\end{equation}

We will prove now the $L^2$ bound
\begin{equation}\label{ba8}
(1+2^{k_2})(2^{\alpha k}+2^{10k})\cdot 2^{(2+\beta)m}2^{\widetilde{k_2}}\big\|P_k\widetilde{T}_{s}^{\sigma;\mu,\nu}(f_{k_1,j_1}^\mu(s),f_{k_2,j_2}^\nu(s))\big\|_{L^2}\lesssim 2^{-2\beta^4 m},
\end{equation}
for any $s\in[2^{m-1},2^{m+1}]$, see \eqref{nh9.2} for the definition of the bilinear operators $\widetilde{T}_{s}^{\sigma;\mu,\nu}$. In view of the assumption \eqref{ok110}) this would clearly imply the desired $L^2$ bound in \eqref{ok112}. 

Assume first that $\min(j_1,j_2)\leq m(1-9\beta)$, i.e.
\begin{equation}\label{ba9}
\min(j_1,j_2)\leq m(1-9\beta),\qquad \max(j_1,j_2)\geq (1-\beta/10)(m+\widetilde{k_2}),\qquad k_2\geq k_1\geq-5m/6.
\end{equation}
Using \eqref{mk15.6} and \eqref{mk15.65}, and recalling that $\alpha\in[0,\beta]$, we notice that
\begin{equation*}
\begin{split}
&\|Ef_{k_1,j_1}^\mu(s)\|_{L^\infty}\lesssim \min(2^{\beta k_1},2^{-6k_1})2^{-3m/2}2^{(1/2+\beta)j_1},\\
&\|Ef_{k_2,j_2}^\nu(s)\|_{L^\infty}\lesssim \min(2^{\beta k_2},2^{-6k_2})2^{-3m/2}2^{(1/2+\beta)j_2},
\end{split}
\end{equation*}
for any $s\in[2^{m-1},2^{m+1}]$. Therefore, using also \eqref{nh9.1},
\begin{equation*}
\begin{split}
\|P_k\widetilde{T}_{s}^{\sigma;\mu,\nu}&(f_{k_1,j_1}^\mu(s),f_{k_2,j_2}^\nu(s))\big\|_{L^2}\lesssim \min(\|Ef_{k_1,j_1}^\mu(s)\|_{L^\infty}\|Ef_{k_2,j_2}^\nu(s)\|_{L^2}, \|Ef_{k_1,j_1}^\mu(s)\|_{L^2}\|Ef_{k_2,j_2}^\nu(s)\|_{L^\infty})\\
&\lesssim \min(2^{\beta k_1},2^{-6k_1})\min(2^{\beta k_2},2^{-6k_2})\cdot 2^{-3m/2}2^{(1/2+\beta)\min(j_1,j_2)}2^{-(1-\beta)\max(j_1,j_2)}\\
&\lesssim (1+2^{k_2})^{-6}2^{-\widetilde{k_2}}2^{-(2+2\beta)m},
\end{split}
\end{equation*}
which suffices to prove \eqref{ba8}.

Assume now that $\min(j_1,j_2)\geq m(1-9\beta)$, i.e.
\begin{equation}\label{ba10}
\min(j_1,j_2)\geq m(1-9\beta),\qquad \max(j_1,j_2)\geq (1-\beta/10)(m+\widetilde{k_2}),\qquad k_2\geq k_1\geq-5m/6.
\end{equation}
We recall that 
\begin{equation}\label{ba11}
\begin{split}
&f_{k_1,j_1}^\mu=P_{[k_1-2,k_1+2]}(\phii^{(k_1)}_{j_1}\cdot P_{k_1}f_\mu)=(2^{\alpha k_1}+2^{10k_1})^{-1}[P_{[k_1-2,k_1+2]}g_{k_1,j_1}^\mu+P_{[k_1-2,k_1+2]}h_{k_1,j_1}^\mu],\\
&f_{k_2,j_2}^\nu=P_{[k_2-2,k_2+2]}(\phii^{(k_2)}_{j_2}\cdot P_{k_2}f_\nu)=(2^{\alpha k_2}+2^{10k_2})^{-1}[P_{[k_2-2,k_2+2]}g_{k_2,j_2}^\nu+P_{[k_2-2,k_2+2]}h_{k_2,j_2}^\nu],
\end{split}
\end{equation}
and use the decompositions \eqref{ok113}--\eqref{ok114}. Then we estimate, using also \eqref{ba10},
\begin{equation*}
\begin{split}
\big\|P_k\widetilde{T}_{s}^{\sigma;\mu,\nu}&(P_{[k_1-2,k_1+2]}h_{k_1,j_1}^\mu(s),P_{[k_2-2,k_2+2]}h_{k_2,j_2}^\nu(s))\big\|_{L^2}\\
&\lesssim \min(\|\widehat{h_{k_1,j_1}^\mu}(s)\|_{L^1}\|\widehat{h_{k_2,j_2}^\nu}(s)\|_{L^2},\|\widehat{h_{k_1,j_1}^\mu}(s)\|_{L^1}\|\widehat{h_{k_2,j_2}^\nu}(s)\|_{L^2})\\
&\lesssim 2^{-\gamma\max(j_1,j_2)}2^{-(1-\beta)\min(j_1,j_2)}2^{2\beta\widetilde{k_1}}2^{(5/2+\beta-\gamma)\widetilde{k_2}}\\
&\lesssim 2^{-m(\gamma+1-11\beta)}2^{(5/2+\beta-2\gamma)\widetilde{k_2}}2^{2\beta\widetilde{k_1}},
\end{split}
\end{equation*}
\begin{equation*}
\begin{split}
\big\|P_k\widetilde{T}_{s}^{\sigma;\mu,\nu}(P_{[k_1-2,k_1+2]}h_{k_1,j_1}^\mu(s),&P_{[k_2-2,k_2+2]}g_{k_2,j_2}^\nu(s))\big\|_{L^2}\lesssim \|\widehat{h_{k_1,j_1}^\mu}(s)\|_{L^1}\|\widehat{g_{k_2,j_2}^\nu}(s)\|_{L^2}\\
&\lesssim 2^{-\gamma j_1}2^{-(1+\beta)j_2}2^{2\beta\widetilde{k_1}}2^{(5/2+\beta-\gamma)\widetilde{k_2}}\\
&\lesssim 2^{-m(\gamma+1-11\beta)}2^{(5/2+\beta-2\gamma)\widetilde{k_2}}2^{2\beta\widetilde{k_1}},
\end{split}
\end{equation*}
\begin{equation*}
\begin{split}
\big\|P_k\widetilde{T}_{s}^{\sigma;\mu,\nu}(P_{[k_1-2,k_1+2]}g_{k_1,j_1}^\mu(s),&P_{[k_2-2,k_2+2]}h_{k_2,j_2}^\nu(s))\big\|_{L^2}\lesssim \|\widehat{g_{k_1,j_1}^\mu}(s)\|_{L^2}\|\widehat{h_{k_2,j_2}^\nu}(s)\|_{L^1}\\
&\lesssim 2^{-(1+\beta)j_1}2^{-\gamma j_2}2^{2\beta\widetilde{k_1}}2^{(5/2+\beta-\gamma)\widetilde{k_2}}\\
&\lesssim 2^{-m(\gamma+1-11\beta)}2^{(5/2+\beta-2\gamma)\widetilde{k_2}}2^{2\beta\widetilde{k_1}},
\end{split}
\end{equation*}
and, using also \eqref{cucu1} and \eqref{cucu3},
\begin{equation*}
\begin{split}
&\big\|P_k\widetilde{T}_{s}^{\sigma;\mu,\nu}(P_{[k_1-2,k_1+2]}g_{k_1,j_1}^\mu(s),P_{[k_2-2,k_2+2]}g_{k_2,j_2}^\nu(s))\big\|_{L^2}\\
&\lesssim \min\big(\|e^{-is\widetilde{\Lambda}_\mu}P_{[k_1-2,k_1+2]}(g_{k_1,j_1}^\mu(s))\|_{L^\infty}\|g_{k_2,j_2}^\nu(s)\|_{L^2},\|g_{k_1,j_1}^\mu(s)\|_{L^2}\|e^{-is\widetilde{\Lambda}_\nu}P_{[k_2-2,k_2+2]}(g_{k_2,j_2}^\nu(s))\|_{L^\infty}\big)\\
&\lesssim 2^{-(1+\beta)\max(j_1,j_2)}\cdot 2^{-3m/2}2^{(1/2-\beta)\min(j_1,j_2)}(1+2^{3k_2})\\
&\lesssim 2^{-m(2+19\beta/10)}2^{-3\widetilde{k_2}/4}(1+2^{3k_2}).
\end{split}
\end{equation*}
Therefore, using also $\al\in[0,\be/2]$ and $k_1\geq -5m/6$, the left-hand side of \eqref{ba8} is dominated by 
\begin{equation*}
C(1+2^{4k_2})2^{-\alpha k_1}2^{-9\beta m/10}\lesssim (1+2^{4k_2})2^{-29m\beta/60}. 
\end{equation*}
This completes the proof of \eqref{ba8}.

To complete the proof of \eqref{ok112} it remains to prove the $L^\infty$ bound
\begin{equation}\label{ba12}
(1+2^{k_2})(2^{\alpha k}+2^{10k})\cdot 2^{(1/2-\beta)\widetilde{k}}\big\|\mathcal{F}P_kT_m^{\sigma;\mu,\nu}(f_{k_1,j_1}^\mu,f_{k_2,j_2}^\nu)\big\|_{L^\infty}\lesssim 2^{-2\beta^4m}.
\end{equation}
If $k_2\leq -D/10$ then $\max(k,k_1)\leq-D/10+10$ and $1\lesssim |\Phi^{\sigma;\mu,\nu}(\xi,\eta)|$ whenever $|\xi|\approx 2^k,|\xi-\eta|\approx 2^{k_1},|\eta|\approx 2^{k_2}$. Therefore, we integrate by parts in $s$ and use \eqref{nh9.1} and \eqref{ok50} to estimate
\begin{equation*}
\begin{split}
\big\|\mathcal{F}P_kT_m^{\sigma;\mu,\nu}(f_{k_1,j_1}^\mu,&f_{k_2,j_2}^\nu)\big\|_{L^\infty}\lesssim\sup_{s\in[2^{m-1},2^{m+1}]}\Big[\|f_{k_1,j_1}^\mu(s)\|_{L^2}\|f_{k_2,j_2}^\nu(s)\|_{L^2}\\
&+2^m\|(\partial_sf_{k_1,j_1}^\mu)(s)\|_{L^2}\|f_{k_2,j_2}^\nu(s)\|_{L^2}+2^m\|f_{k_1,j_1}^\mu(s)\|_{L^2}\|(\partial_sf_{k_2,j_2}^\nu)(s)\|_{L^2}\Big]\\
&\lesssim 2^{-\beta m}.
\end{split}
\end{equation*}
The desired estimate \eqref{ba12} follows easily in this case.

Assume now that $k_2\geq -D/10$. For \eqref{ba12} it suffices to prove that
\begin{equation}\label{ba14}
2^{k_2}(2^{\alpha k}+2^{10k})\cdot 2^{(1/2-\beta)\widetilde{k}}2^m\big\|\mathcal{F}P_k\widetilde{T}_s^{\sigma;\mu,\nu}(f_{k_1,j_1}^\mu,f_{k_2,j_2}^\nu)\big\|_{L^\infty}\lesssim 2^{-2\beta^4m},
\end{equation}
for any $s\in[2^{m-1},2^{m+1}]$. If, in addition, $k_1\leq -2m/5$ then, as in \eqref{ba6.5},
\begin{equation*}
\big\|\mathcal{F}P_k\widetilde{T}_s^{\sigma;\mu,\nu}(f_{k_1,j_1}^\mu,f_{k_2,j_2}^\nu)\big\|_{L^\infty}\lesssim \|\widehat{f_{k_1,j_1}^\mu}(s)\|_{L^1}\|\widehat{f_{k_2,j_2}^\nu}(s)\|_{L^\infty}\lesssim 2^{(5/2-\alpha+\beta)k_1}2^{-10k_2},
\end{equation*}
and the desired bound \eqref{ba14} follows since $\alpha\in[0,\beta/2]$. 

It remains to prove the bound \eqref{ba14} in the case
\begin{equation}\label{ba13}
k_2\geq -D/10,\qquad k_1\geq -2m/5.
\end{equation}
We decompose $f_{k_1,j_1}^\mu,f_{k_2,j_2}^\nu$ as in \eqref{ok113}, \eqref{ok114}, \eqref{ba11}. If $j_1\leq j_2$ we estimate
\begin{equation*}
\begin{split}
\big\|\mathcal{F}P_k\widetilde{T}_s^{\sigma;\mu,\nu}(P_{[k_1-2,k_1+2]}(g_{k_1,j_1}^\mu(s)+h_{k_1,j_1}^\mu(s)),&P_{[k_2-2,k_2+2]}g_{k_2,j_2}^\nu(s))\big\|_{L^\infty}\\
&\lesssim \big(\|\widehat{g_{k_1,j_1}^\mu}(s)\|_{L^2}+\|\widehat{h_{k_1,j_1}^\mu}(s)\|_{L^2}\big)\|\widehat{g_{k_2,j_2}^\nu}(s)\|_{L^2}\\
&\lesssim 2^{-(1+\beta)j_2},
\end{split}
\end{equation*}
and
\begin{equation*}
\begin{split}
\big\|\mathcal{F}P_k\widetilde{T}_s^{\sigma;\mu,\nu}(P_{[k_1-2,k_1+2]}(g_{k_1,j_1}^\mu(s)+h_{k_1,j_1}^\mu(s)),&P_{[k_2-2,k_2+2]}h_{k_2,j_2}^\nu(s))\big\|_{L^\infty}\\
&\lesssim \big(\|\widehat{g_{k_1,j_1}^\mu}(s)\|_{L^\infty}+\|\widehat{h_{k_1,j_1}^\mu}(s)\|_{L^\infty}\big)\|\widehat{h_{k_2,j_2}^\nu}(s)\|_{L^1}\\
&\lesssim 2^{-(1/2-\beta)\widetilde{k_1}}2^{-\gamma j_2}.
\end{split}
\end{equation*}
Since $-\widetilde{k_1}\leq 2m/5$ and $2^{j_2}\gtrsim 2^{m(1-\beta/10)}$ it follows that
\begin{equation}\label{ba15}
\text{ if }\,\,\,j_1\leq j_2\,\,\,\text{ then }\,\,\,\big\|\mathcal{F}P_k\widetilde{T}_s^{\sigma;\mu,\nu}(f_{k_1,j_1}^\mu,f_{k_2,j_2}^\nu)\big\|_{L^\infty}\lesssim 2^{-(1+\beta)(1-\beta/10)m}\cdot (2^{\alpha k_1}+2^{10k_1})^{-1}2^{-10k_2}.
\end{equation}

Similarly, if $j_1\geq j_2$ we estimate
\begin{equation*}
\begin{split}
\big\|\mathcal{F}P_k\widetilde{T}_s^{\sigma;\mu,\nu}(P_{[k_1-2,k_1+2]}g_{k_1,j_1}^\mu(s),&P_{[k_2-2,k_2+2]}(g_{k_2,j_2}^\nu(s)+h_{k_2,j_2}^\nu(s)))\big\|_{L^\infty}\\
&\lesssim \|\widehat{g_{k_1,j_1}^\mu}(s)\|_{L^2}\big(\|\widehat{g_{k_2,j_2}^\nu}(s)\|_{L^2}+\|\widehat{h_{k_2,j_2}^\nu}(s)\|_{L^2}\big)\\
&\lesssim 2^{-(1+\beta)j_1},
\end{split}
\end{equation*}
and
\begin{equation*}
\begin{split}
\big\|\mathcal{F}P_k\widetilde{T}_s^{\sigma;\mu,\nu}(P_{[k_1-2,k_1+2]}g_{k_1,j_1}^\mu(s),&P_{[k_2-2,k_2+2]}(g_{k_2,j_2}^\nu(s)+h_{k_2,j_2}^\nu(s)))\big\|_{L^\infty}\\
&\lesssim \|\widehat{g_{k_1,j_1}^\mu}(s)\|_{L^1}\big(\|\widehat{g_{k_2,j_2}^\nu}(s)\|_{L^\infty}+\|\widehat{h_{k_2,j_2}^\nu}(s)\|_{L^\infty}\big)\\
&\lesssim 2^{-\gamma j_1},
\end{split}
\end{equation*}
Since $2^{j_1}\gtrsim 2^{m(1-\beta/10)}$ it follows that
\begin{equation}\label{ba16}
\text{ if }\,\,\,j_1\geq j_2\,\,\,\text{ then }\,\,\,\big\|\mathcal{F}P_k\widetilde{T}_s^{\sigma;\mu,\nu}(f_{k_1,j_1}^\mu,f_{k_2,j_2}^\nu)\big\|_{L^\infty}\lesssim 2^{-(1+\beta)(1-\beta/10)m}\cdot (2^{\alpha k_1}+2^{10k_1})^{-1}2^{-10k_2}.
\end{equation}

Using \eqref{ba15} and \eqref{ba16}, the left-hand side of \eqref{ba14} is dominated by
\begin{equation*}
C2^{2k_2}2^{-\alpha k_1}2^{-4\beta m/5},
\end{equation*}
which suffices. This completes the proof of the lemma.
\end{proof}

\begin{lemma}\label{BigBound6}
The bound \eqref{ok60} holds provided that \eqref{ok61} holds and, in addition,
\begin{equation}\label{ok100}
j\leq m+\max(\widetilde{k_1},\widetilde{k_2})+D,\quad \max(j_1,j_2)\leq (1-\beta/10)(m+\max(\widetilde{k_1},\widetilde{k_2})),\quad \min(k,k_1,k_2)\leq-D/10.
\end{equation}
\end{lemma}

\begin{proof}[Proof of Lemma \ref{BigBound6}] Using definition \eqref{sec5.3}, it suffices to prove that
\begin{equation}\label{ba20}
\begin{split}
&(1+2^{k_1}+2^{k_2})(2^{\alpha k}+2^{10k})\cdot 2^{(1+\beta)j}\big\|\widetilde{\varphi}^{(k)}_j\cdot P_kT_{m}^{\sigma;\mu,\nu}(f_{k_1,j_1}^\mu,f_{k_2,j_2}^\nu)\big\|_{L^2}\\
&+(1+2^{k_1}+2^{k_2})(2^{\alpha k}+2^{10k})\cdot 2^{(1/2-\beta)\widetilde{k}}\big\|\mathcal{F}[\widetilde{\varphi}^{(k)}_j\cdot P_kT_{m}^{\sigma;\mu,\nu}(f_{k_1,j_1}^\mu,f_{k_2,j_2}^\nu)]\big\|_{L^\infty}\lesssim 2^{-2\beta^4m}.
\end{split}
\end{equation}
By symmetry, we may assume $k_1\leq k_2$.

As in the proof of Lemma \ref{BigBound4} we decompose
\begin{equation*}
\begin{split}
&\mathcal{F}P_kT_{m}^{\sigma;\mu,\nu}(f_{k_1,j_1}^\mu,f_{k_2,j_2}^\nu)(\xi)=G(\xi)+H(\xi),\\
&G(\xi):=\varphi_k(\xi)\int_{\mathbb{R}}\int_{\mathbb{R}^3}e^{is\Phi^{\sigma;\mu,\nu}(\xi,\eta)}\varphi(2^D\Phi^{\sigma;\mu,\nu}(\xi,\eta))q_m(s)\widehat{f_{k_1,j_1}^\mu}(\xi-\eta,s)\widehat{f_{k_2,j_2}^\nu}(\eta,s)\,d\eta ds,\\
&H(\xi):=\varphi_k(\xi)\int_{\mathbb{R}}\int_{\mathbb{R}^3}e^{is\Phi^{\sigma;\mu,\nu}(\xi,\eta)}[1-\varphi(2^D\Phi^{\sigma;\mu,\nu}(\xi,\eta))]q_m(s)\widehat{f_{k_1,j_1}^\mu}(\xi-\eta,s)\widehat{f_{k_2,j_2}^\nu}(\eta,s)\,d\eta ds.
\end{split}
\end{equation*}

We show first that
\begin{equation}\label{b30}
(1+2^{k_2})(2^{\alpha k}+2^{10k})\cdot 2^{(1+\beta)(m+\widetilde{k_2})}\|H\|_{L^2}+(1+2^{k_2})(2^{\alpha k}+2^{10k})\cdot 2^{(1/2-\beta)\widetilde{k}}\|H\|_{L^\infty}\lesssim 2^{-2\beta^4m}.
\end{equation}
For this we integrate by parts in $s$ and use the bound \eqref{mk6.6}. It follows that
\begin{equation}\label{b31}
\begin{split}
\|H\|_{L^2}\lesssim (1+2^{3k_1})&(1+2^{3k_2})\sup_{s\in[2^{m-1},2^{m+1}]}\Big[2^m\|Ef_{k_1,j_1}^\mu(s)\|_{L^\infty}\|(\partial_sf_{k_2,j_2}^\nu)(s)\|_{L^2}\\
&+2^m\|(\partial_sf_{k_1,j_1}^\mu)(s)\|_{L^2}\|Ef_{k_2,j_2}^\nu(s)\|_{L^\infty}\\
&+\min\big(\|f_{k_1,j_1}^\mu(s)\|_{L^2}\|Ef_{k_2,j_2}^\nu(s)\|_{L^\infty},\|Ef_{k_1,j_1}^\mu(s)\|_{L^\infty}\|f_{k_2,j_2}^\nu(s)\|_{L^2}\big)\Big]
\end{split}
\end{equation}
and
\begin{equation}\label{b32}
\begin{split}
\|H\|_{L^\infty}\lesssim \sup_{s\in[2^{m-1},2^{m+1}]}&\Big[2^m\|f_{k_1,j_1}^\mu(s)\|_{L^2}\|(\partial_sf_{k_2,j_2}^\nu)(s)\|_{L^2}\\
&+2^m\|(\partial_sf_{k_1,j_1}^\mu)(s)\|_{L^2}\|f_{k_2,j_2}^\nu(s)\|_{L^2}+\|H_1(s)\|_{L^\infty}\Big],
\end{split}
\end{equation}
where
\begin{equation}\label{b32.5}
H_1(\xi,s):=\varphi_k(\xi)\int_{\mathbb{R}^3}e^{is\Phi^{\sigma;\mu,\nu}(\xi,\eta)}\frac{[1-\varphi(2^D\Phi^{\sigma;\mu,\nu}(\xi,\eta))]}{\Phi^{\sigma;\mu,\nu}(\xi,\eta)}\widehat{f_{k_1,j_1}^\mu}(\xi-\eta,s)\widehat{f_{k_2,j_2}^\nu}(\eta,s)\,d\eta.
\end{equation}

Using \eqref{nh9} and Lemma \ref{ders}, for any $s\in[2^{m-1},2^{m+1}]$
\begin{equation}\label{b33}
\begin{split}
2^m\|Ef_{k_1,j_1}^\mu(s)\|_{L^\infty}&\|(\partial_sf_{k_2,j_2}^\nu)(s)\|_{L^2}+2^m\|(\partial_sf_{k_1,j_1}^\mu)(s)\|_{L^2}\|Ef_{k_2,j_2}^\nu(s)\|_{L^\infty}\\
&\lesssim (1+2^{k_1})^{-6}(1+2^{k_2})^{-6}2^{-(1+2\beta)m}.
\end{split}
\end{equation}
Moreover, using again \eqref{nh9} and \eqref{nh9.1}, if $s\in[2^{m-1},2^{m+1}]$ and $\max(j_1,j_2)\geq 4\beta m$ then
\begin{equation*}
\min\big(\|f_{k_1,j_1}^\mu(s)\|_{L^2}\|Ef_{k_2,j_2}^\nu(s)\|_{L^\infty},\|Ef_{k_1,j_1}^\mu(s)\|_{L^\infty}\|f_{k_2,j_2}^\nu(s)\|_{L^2}\big)\lesssim (1+2^{k_1})^{-6}(1+2^{k_2})^{-6}2^{-(1+2\beta)m}.
\end{equation*}
On the other hand, using also \eqref{mk15.6}--\eqref{mk15.65}, if $s\in[2^{m-1},2^{m+1}]$ and $\max(j_1,j_2)\leq 4\beta m$ then
\begin{equation*}
\min\big(\|f_{k_1,j_1}^\mu(s)\|_{L^2}\|Ef_{k_2,j_2}^\nu(s)\|_{L^\infty},\|Ef_{k_1,j_1}^\mu(s)\|_{L^\infty}\|f_{k_2,j_2}^\nu(s)\|_{L^2}\big)\lesssim (1+2^{k_1})^{-6}(1+2^{k_2})^{-6}2^{-(1+2\beta)m}.
\end{equation*}
Therefore, using also \eqref{b31} and \eqref{b33} it follows that
\begin{equation}\label{b36}
(1+2^{k_2})(2^{\alpha k}+2^{10k})\cdot 2^{(1+\beta)(m+\widetilde{k_2})}\|H\|_{L^2}\lesssim 2^{-2\beta^4m},
\end{equation}
as desired.

To prove the $L^\infty$ bound in \eqref{b30} we use \eqref{nh9} and Lemma \ref{ders} to estimate
\begin{equation}\label{b34}
2^m\|f_{k_1,j_1}^\mu(s)\|_{L^2}\|(\partial_sf_{k_2,j_2}^\nu)(s)\|_{L^2}+2^m\|(\partial_sf_{k_1,j_1}^\mu)(s)\|_{L^2}\|f_{k_2,j_2}^\nu(s)\|_{L^2}\lesssim (1+2^{k_2})^{-6}2^{-\beta m}.
\end{equation}
Then we estimate, using \eqref{nh9.1},
\begin{equation*}
\|H_1(s)\|_{L^\infty}\lesssim \|f_{k_1,j_1}^\mu(s)\|_{L^2}\|f_{k_2,j_2}^\nu(s)\|_{L^2}\lesssim 2^{-(j_1+j_2)/2}(1+2^{k_2})^{-10}.
\end{equation*}
The desired $L^\infty$ estimate in \eqref{b30},
\begin{equation}\label{b35}
(1+2^{k_2})(2^{\alpha k}+2^{10k})\cdot 2^{(1/2-\beta)\widetilde{k}}\|H\|_{L^\infty}\lesssim 2^{-2\beta^4m}
\end{equation}
follows from \eqref{b32} unless
\begin{equation}\label{b37}
\max(j_1,j_2,-k,-k_1,-k_2)\leq 2\beta m.
\end{equation}

On the other hand, assuming \eqref{b37}, we need to improve slightly on the $L^\infty$ bound on $H_1(s)$. We decompose $H_1(\xi,s)=H_2(\xi,s)+H_3(\xi,s)$ where 
\begin{equation*}
\begin{split}
H_2(\xi,s):=\varphi_k(\xi)&\int_{\mathbb{R}^3}\varphi_{(-\infty,-(1/2-\beta^2)m]}(\Xi^{\mu,\nu}(\xi,\eta))e^{is\Phi^{\sigma;\mu,\nu}(\xi,\eta)}\\
&\frac{[1-\varphi(2^D\Phi^{\sigma;\mu,\nu}(\xi,\eta))]}{\Phi^{\sigma;\mu,\nu}(\xi,\eta)}\widehat{f_{k_1,j_1}^\mu}(\xi-\eta,s)\widehat{f_{k_2,j_2}^\nu}(\eta,s)\,d\eta,
\end{split}
\end{equation*}
and
\begin{equation*}
\begin{split}
H_3(\xi,s):=\varphi_k(\xi)&\int_{\mathbb{R}^3}[1-\varphi_{(-\infty,-(1/2-\beta^2)m]}(\Xi^{\mu,\nu}(\xi,\eta))]e^{is\Phi^{\sigma;\mu,\nu}(\xi,\eta)}\\
&\frac{[1-\varphi(2^D\Phi^{\sigma;\mu,\nu}(\xi,\eta))]}{\Phi^{\sigma;\mu,\nu}(\xi,\eta)}\widehat{f_{k_1,j_1}^\mu}(\xi-\eta,s)\widehat{f_{k_2,j_2}^\nu}(\eta,s)\,d\eta.
\end{split}
\end{equation*}
Using Lemma \ref{tech5} (with $K\approx 2^{m(1/2+\beta^2)}$, $\eps\approx 2^{-m/2}$), the restriction \eqref{b37}, and the bound \eqref{nh9}, it follows that $|H_3(\xi,s)|\lesssim 2^{-m}$. At the same time, using the explicit formula \eqref{ok47}, and the simple equality
\begin{equation*}
\vert \vec{A}-\vec{B}\vert^2=\vert \vert\vec{A}\vert-\vert\vec{B}\vert\vert^2+\vert\vec{A}\vert\cdot\vert\vec{B}\vert(1-\cos\theta),\quad \theta=\angle(\vec{A},\vec{B})
\end{equation*}
it is easy to see that if $|\xi|\approx 2^k$, $|\xi-\eta|\approx 2^{k_1}$, $|\eta|\approx 2^{k_2}$, where $\max(|k|,|k_1|,|k_2|)\leq 2\beta m$, and if $|\Xi^{\mu,\nu}(\xi,\eta)|\lesssim 2^{-m/3}$ then
\begin{equation*}
\min\Big(\big|\eta-\xi|\eta|/|\xi|\big|,\big|\eta+\xi|\eta|/|\xi|\big|\Big)\lesssim 2^{-m/4}.
\end{equation*}
Therefore, using the last bound in \eqref{nh9}, $|H_2(\xi,s)|\lesssim 2^{-m/5}$. As a result, assuming \eqref{b37}, it follows that $|H_1(\xi,s)|\lesssim 2^{-m/5}$. The desired bound \eqref{b35} follows using also \eqref{b32} and \eqref{b34}. This completes the proof of the main estimate \eqref{b30}.

We show now that
\begin{equation}\label{b40}
(1+2^{k_2})(2^{\alpha k}+2^{10k})\cdot 2^{(1+\beta)(m+\widetilde{k_2})}\|G\|_{L^2}+(1+2^{k_2})(2^{\alpha k}+2^{10k})\cdot 2^{(1/2-\beta)\widetilde{k}}\|G\|_{L^\infty}\lesssim 2^{-2\beta^4m}.
\end{equation}
Notice that $G=0$ unless
\begin{equation}\label{b41}
k_2\geq -D/20.
\end{equation}
As in the proof of Lemma \ref{BigBound4}, for any $l\in\mathbb{Z}$ we define
\begin{equation*}
\begin{split}
G_{\leq l}(\xi) :=\varphi_k(\xi)&\int_{\mathbb{R}}\int_{\mathbb{R}^3}\varphi_{(-\infty,l]}(\Xi^{\mu,\nu}(\xi,\eta))\cdot e^{is\Phi^{\sigma;\mu,\nu}(\xi,\eta)}\\
&\varphi(2^D\Phi^{\sigma;\mu,\nu}(\xi,\eta))q_m(s)\widehat{f_{k_1,j_1}^\mu}(\xi-\eta,s)\widehat{f_{k_2,j_2}^\nu}(\eta,s)\,d\eta ds.
\end{split}
\end{equation*}
Let $G_l:=G_{\leq l}-G_{\leq l-1}$. Recalling the assumption $\max(j_1,j_2)\leq(1-\beta/10)m$, we notice that if $l\geq -\beta m/11$ then we may apply Lemma \ref{tech5} (with $K\approx 2^{(1-\beta/11)m}$, $\varepsilon\approx 2^{-(1-\beta/10)m}$) and use the bounds \eqref{nh9} to conclude that
\begin{equation*}
\|G_l\|_{L^\infty}\lesssim 2^{-4m}\qquad\text{ if }\qquad l\geq l_0:=\lfloor-\beta m/11\rfloor.
\end{equation*}
On the other hand, recalling that $\min(k,k_1,k_2)\leq-D/10$ and the inequality \eqref{b41}, we notice that
\begin{equation*}
G_{\leq l_0}=0\qquad\text{ if }\qquad k_1\leq -D/10.
\end{equation*}
Finally, if $k\leq -D/10$ and $k_2\le j/N_0^\prime$ and using Lemma \ref{desc1} (i), $G_{\leq l_0}=0$. The desired estimate \eqref{b40} follows easily.
\end{proof}

\begin{lemma}\label{BigBound6.5}
The bound \eqref{ok60} holds provided that \eqref{ok61} holds and, in addition,
\begin{equation}\label{tik20}
j\leq m+\max(\widetilde{k_1},\widetilde{k_2})+D,\quad \max(j_1,j_2)\leq (1-\beta/10)(m+\max(\widetilde{k_1},\widetilde{k_2})),\quad \max(k,k_1,k_2)\geq D.
\end{equation}
\end{lemma}

\begin{proof}[Proof of Lemma \ref{BigBound6.5}] This is similar to the proof of Lemma \ref{BigBound6}, using Lemma \ref{desc1} (ii) instead of Lemma \ref{desc1} (i). Using definition \eqref{sec5.3}, it suffices to prove that
\begin{equation}\label{tik21}
\begin{split}
&2^{\max(k_1,k_2)}(2^{\alpha k}+2^{10k})\cdot 2^{(1+\beta)j}\big\|\widetilde{\varphi}^{(k)}_j\cdot P_kT_{m}^{\sigma;\mu,\nu}(f_{k_1,j_1}^\mu,f_{k_2,j_2}^\nu)\big\|_{L^2}\\
&+2^{\max(k_1,k_2)}(2^{\alpha k}+2^{10k})\cdot 2^{(1/2-\beta)\widetilde{k}}\big\|\mathcal{F}[\widetilde{\varphi}^{(k)}_j\cdot P_kT_{m}^{\sigma;\mu,\nu}(f_{k_1,j_1}^\mu,f_{k_2,j_2}^\nu)]\big\|_{L^\infty}\lesssim 2^{-2\beta^4m}.
\end{split}
\end{equation}
The inequalities in \eqref{tik20} show that
\begin{equation*}
\max(k_1,k_2)\geq D-10,\quad j\leq m+D,\quad \max(j_1,j_2)\leq (1-\beta/10)m.
\end{equation*}
By symmetry we may assume that $k_1\leq k_2$.

As in the proof of Lemma \ref{BigBound6} we decompose
\begin{equation*}
\begin{split}
&\mathcal{F}P_kT_{m}^{\sigma;\mu,\nu}(f_{k_1,j_1}^\mu,f_{k_2,j_2}^\nu)(\xi)=G(\xi)+H(\xi),\\
&G(\xi):=\varphi_k(\xi)\int_{\mathbb{R}}\int_{\mathbb{R}^3}e^{is\Phi^{\sigma;\mu,\nu}(\xi,\eta)}\varphi(2^{2D+2k_2}\Phi^{\sigma;\mu,\nu}(\xi,\eta))q_m(s)\widehat{f_{k_1,j_1}^\mu}(\xi-\eta,s)\widehat{f_{k_2,j_2}^\nu}(\eta,s)\,d\eta ds,\\
&H(\xi):=\varphi_k(\xi)\int_{\mathbb{R}}\int_{\mathbb{R}^3}e^{is\Phi^{\sigma;\mu,\nu}(\xi,\eta)}[1-\varphi(2^{2D+2k_2}\Phi^{\sigma;\mu,\nu}(\xi,\eta))]q_m(s)\widehat{f_{k_1,j_1}^\mu}(\xi-\eta,s)\widehat{f_{k_2,j_2}^\nu}(\eta,s)\,d\eta ds.
\end{split}
\end{equation*}

As in the proof of Lemma \ref{BigBound6} we integrate by parts in $s$ to estimate the contributions of the function $H$, and integrate by parts in $\eta$ to estimate the contributions of the function $G$. More precisely, we argue as in the proof of Lemma \ref{BigBound6}, using Lemma \ref{desc1} (ii) instead of Lemma \ref{desc1} (i), to conclude that
\begin{equation*}
2^{k_2}(2^{\alpha k}+2^{10k})\cdot 2^{(1+\beta)m}\|H\|_{L^2}+2^{k_2}(2^{\alpha k}+2^{10k})\cdot 2^{(1/2-\beta)\widetilde{k}}\|H\|_{L^\infty}+2^{2m}\|G\|_{L^{\infty}}\lesssim 2^{-2\beta^4m}.
\end{equation*}
Clearly, this suffices to prove the desired estimate \eqref{tik21}.
\end{proof}

We examine now the conclusions of Lemma \ref{BigBound3}, Lemma \ref{BigBound4}, Lemma \ref{BigBound5}, Lemma \ref{BigBound6}, and Lemma \ref{BigBound6.5}, and notice that to complete the proof of Proposition \ref{reduced2}, it suffices to prove Proposition \ref{reduced3} below. 

\begin{proposition}\label{reduced3}
With the same notation as in Proposition \ref{reduced1}, we have
\begin{equation}\label{ok90}
(1+2^{k_1}+2^{k_2})\big\|\widetilde{\varphi}^{(k)}_j\cdot P_kT_{m}^{\sigma;\mu,\nu}(f_{k_1,j_1}^\mu,f_{k_2,j_2}^\nu)\big\|_{B_{k,j}}\lesssim 2^{-2\beta^4m},
\end{equation}
for any fixed $\mu,\nu\in\mathcal{I}_d$, $(k,j),(k_1,j_1),(k_2,j_2)\in\mathcal{J}$, and $m\in[0,L+1]\cap\mathbb{Z}$, satisfying
\begin{equation}\label{ok91}
\beta m/2+D^2\leq j\leq m+D,\qquad \max(j_1,j_2)\leq (1-\beta/10)m,\qquad -D/10\leq k,k_1,k_2\leq D.
\end{equation}
\end{proposition}

The most delicate part of the analysis is done to prove Proposition \ref{reduced3} and corresponds to the resonant interaction at time $T$ and at location $X\simeq T$ of inputs located at position $Y\lesssim T$. This forms the bulk of the nonlinear stationary phase argument. We separate two cases.

(i) when the inputs are located close to the origin $1\lesssim Y\lesssim T^\frac{1}{2}$. In this case, essentially no parameter in the norm can give additional control and we {\it must} understand the result of the interaction. This is what sets the ``weak norm''. On the positive side, in this case, the inputs have essentially smooth Fourier transforms and allow for efficient stationary phase analysis, which gives a good description of the output.

(ii) when at least one input is located further away from the origin $T^\frac{1}{2}\lesssim Y\lesssim T$. In this case, the stationary phase analysis gets less and less efficient as $Y$ increases and we have access to less information on the output. However, this is compensated for by the fact that the parameters in the norm (and in particular the appropriate choice of $\beta$) start to give stronger control as $Y$ increases. In our situation, this is enough and we can always control the outcome of this interaction in the strong norm.

%
%
%
%

\subsection{Proof of Proposition \ref{reduced3}} In this subsection we prove Proposition \ref{reduced3}. The arguments are more complicated than before; to control some of the more difficult spacetime resonances we need to use the more refined $B_{k,j}$ norms. We also need additional $L^2$ orthogonality arguments.

\begin{lemma}\label{BigBound7}
The bound \eqref{ok90} holds provided that \eqref{ok91} holds and, in addition,
\begin{equation}\label{cle0}
\max(j_1,j_2)\leq m(1/2-\beta^2).
\end{equation}
\end{lemma}

\begin{proof}[Proof of Lemma \ref{BigBound7}] Let 
\begin{equation}\label{cle1}
\ka_1:=2^{-m/2}2^{\beta^2m},
\end{equation}
and decompose first
\begin{equation*}
\begin{split}
&\mathcal{F}P_kT_{m}^{\sigma;\mu,\nu}(f_{k_1,j_1}^\mu,f_{k_2,j_2}^\nu)(\xi)=G(\xi)+H(\xi),\\
&G(\xi):=\varphi_k(\xi)\int_{\mathbb{R}}\int_{\mathbb{R}^3}e^{is\Phi^{\sigma;\mu,\nu}(\xi,\eta)}\varphi_{\leq 0}(\Xi^{\mu,\nu}(\xi,\eta)/\ka_1)q_m(s)\widehat{f_{k_1,j_1}^\mu}(\xi-\eta,s)\widehat{f_{k_2,j_2}^\nu}(\eta,s)\,d\eta ds,\\
&H(\xi):=\varphi_k(\xi)\int_{\mathbb{R}}\int_{\mathbb{R}^3}e^{is\Phi^{\sigma;\mu,\nu}(\xi,\eta)}[1-\varphi_{\leq 0}(\Xi^{\mu,\nu}(\xi,\eta)/\ka_1)]q_m(s)\widehat{f_{k_1,j_1}^\mu}(\xi-\eta,s)\widehat{f_{k_2,j_2}^\nu}(\eta,s)\,d\eta ds.
\end{split}
\end{equation*}
Using Lemma \ref{tech5} (with $K\approx 2^m\ka_1$ and $\eps\approx \ka_1$) and the last bound in \eqref{nh9} it is easy to see 
that $\|H\|_{L^\infty}\lesssim 2^{-10m}$. Therefore it remains to prove that
\begin{equation}\label{cle2}
\big\|\widetilde{\varphi}^{(k)}_j\cdot \mathcal{F}^{-1}(G)\|_{B_{k,j}}\lesssim 2^{-2\beta^4m}.
\end{equation}

Using the $L^\infty$ bounds in \eqref{nh9} and Lemma \ref{desc2}, we see easily that
\begin{equation}\label{cle3}
\|G\|_{L^\infty}\lesssim \ka_1^3\cdot 2^m\lesssim 2^{-m/2}2^{3\beta^2 m}.
\end{equation}
This suffices to prove the desired bound \eqref{cle2} if, for example, $j\leq m(1/2-4\beta)$. To cover the entire range $j\leq m+D$ we need 
more refined bounds on $|G(\xi)|$, which we prove using integration by parts in $s$. 

In the argument below we may assume that $G\neq 0$; in particular this guarantees that the main assumptions \eqref{kx10} 
and \eqref{kxz10} are satisfied. With $\Psi^{\si;\mu,\nu}(|\xi|)=\Phi^{\si;\mu,\nu}(\xi,p^{\mu,\nu}(\xi))$, defined as in \eqref{kxz11}, 
assume that
\begin{equation}\label{cle6}
2^m|\Psi^{\si;\mu,\nu}(|\xi|)|\in [2^l,2^{l+1}],\,\,l\in[\beta m,\infty)\cap\mathbb{Z}.
\end{equation}
Then, using Lemma \ref{desc2}, we see that
\begin{equation*}
|\Phi^{\si;\mu,\nu}(\xi,\eta)-\Psi^{\si;\mu,\nu}(\vert\xi\vert)|\le\vert\eta-p^{\mu,\nu}(\xi)\vert\cdot\sup_{\vert \zeta-p^{\mu,\nu}(\xi)\vert\le 2^{10D}\kappa_1}\vert\Xi^{\mu,\nu}(\xi,\zeta)\vert\lesssim 2^{30D}\kappa_1\vert\eta-p^{\mu,\nu}(\xi)\vert
\end{equation*}
since $\Xi^{\mu,\nu}(\xi,p^{\mu,\nu}(\xi))=0$. Therefore
\begin{equation*}
 2^m|\Phi^{\si;\mu,\nu}(\xi,\eta)|\in[2^{l-3},2^{l+4}]\qquad\text{ if }\qquad \Xi^{\mu,\nu}(\xi,\eta)\leq 100\ka_1.
\end{equation*}
After integration by parts in $s$ it follows that
\begin{equation*}
 \begin{split}
 |G(\xi)|\lesssim 2^{m-l}|\varphi_k(\xi)|\int_{\mathbb{R}}\int_{\mathbb{R}^3}
&|\varphi_{\leq 0}(\Xi^{\mu,\nu}(\xi,\eta)/\ka_1)|\,|q'_m(s)|\,|\widehat{f_{k_1,j_1}^\mu}(\xi-\eta,s)|\,|\widehat{f_{k_2,j_2}^\nu}(\eta,s)|\\
&+|\varphi_{\leq 0}(\Xi^{\mu,\nu}(\xi,\eta)/\ka_1)|\,|q_m(s)|\,|(\partial_s\widehat{f_{k_1,j_1}^\mu})(\xi-\eta,s)|\,|\widehat{f_{k_2,j_2}^\nu}(\eta,s)|\\
&+|\varphi_{\leq 0}(\Xi^{\mu,\nu}(\xi,\eta)/\ka_1)|\,|q_m(s)|\,|\widehat{f_{k_1,j_1}^\mu}(\xi-\eta,s)|\,|(\partial_s\widehat{f_{k_2,j_2}^\nu})(\eta,s)|\,d\eta ds.
 \end{split}
\end{equation*}
We use now \eqref{nh2}, the last bound in \eqref{nh9}, \eqref{nxc1}, and Lemma \ref{desc2}. It follows that
\begin{equation}\label{cle7}
 |G(\xi)|\lesssim 2^{m-l}|\varphi_k(\xi)|\cdot \ka_1^3\lesssim |\varphi_k(\xi)|\cdot 
2^{-l}2^{-m/2}2^{3\beta^2 m}
\end{equation}
provided that \eqref{cle6} holds.

We can now prove the desired bound \eqref{cle2}. To apply \eqref{cle6}--\eqref{cle7} we need a good description of the 
level sets of the functions $\Psi^{\si;\mu,\nu}$. Let
\begin{equation*}
 \begin{split}
&l_0:=\lfloor\beta m+2\rfloor,\qquad D_{l_0}:=\{\xi\in\mathbb{R}^3:2^m|\Psi^{\si;\mu,\nu}(|\xi|)|\leq 2^{l_0}\},\\
&D_l:=\{\xi\in\mathbb{R}^3:2^m|\Psi^{\si;\mu,\nu}(|\xi|)|\in(2^{l-1},2^{l+1}]\},\qquad l\in[l_0+1,m+D]\cap\mathbb{Z},\\
&G=\sum_{l=l_0}^{m+D}G_l,\qquad G_l(\xi):=G(\xi)\cdot\mathbf{1}_{D_l}(\xi).
 \end{split}
\end{equation*}
For \eqref{cle2} it remains to prove that for any $l\in[l_0,m+D]\cap\mathbb{Z}$
\begin{equation}\label{cle8}
\big\|\widetilde{\varphi}^{(k)}_j\cdot \mathcal{F}^{-1}(G_l)\|_{B_{k,j}}\lesssim 2^{-3\beta^4m}.
\end{equation}

Using Lemma \ref{desc4}, it follows that there is $r^{\sigma;\mu,\nu}=r^{\sigma;\mu,\nu}(\mu,\nu,\sigma,k,k_1,k_2,l)\in [2^{-D},\infty)$ 
with the property that
\begin{equation}\label{cle9}
 D_l\subseteq\{\xi\in\mathbb{R}^3:\big||\xi|-r^{\sigma;\mu,\nu}\big|\lesssim 2^{l-m}\}.
\end{equation}
Therefore, using also \eqref{cle7},
\begin{equation*}
\begin{split}
 \big\|\widetilde{\varphi}^{(k)}_j\cdot \mathcal{F}^{-1}(G_l)\|_{B^1_{k,j}}
&\lesssim 2^{(1+\beta)j}\|G_l\|_{L^2}+\|G_l\|_{L^\infty}\\
&\lesssim 2^{-l}2^{-m/2}2^{3\beta^2 m}\cdot \big(2^{(1+\beta)j}2^{(l-m)/2}+1)\\
&\lesssim 2^{j-m}2^{-l/2}2^{\beta m+3\beta^2m}+2^{-l}2^{-m/2}2^{3\beta^2 m}.
\end{split}
\end{equation*}
This clearly suffices to prove \eqref{cle8} if $l\ge 6\beta m$ or $j\leq m-3\beta m$.

It remains to prove \eqref{cle8} in the remaining case
\begin{equation}\label{cle10}
l\in[\beta m, 6\beta m]\cap\mathbb{Z}\qquad\text{ and }\qquad j\in[m-3\beta m,m+D]\cap\mathbb{Z}. 
\end{equation}
For this we need to use the norms $B^2_{k,j}$ defined in \eqref{sec5.4}. Assume first that $l\geq l_0+1$. As before we estimate easily
\begin{equation*}
\begin{split}
 2^{(1-\beta)j}\|G_l\|_{L^2}+\|G_l\|_{L^\infty}&\lesssim 2^{-l}2^{-m/2}2^{3\beta^2 m}\cdot \big(2^{(1-\beta)m}2^{(l-m)/2}+1)\\
&\lesssim 2^{-l/2}2^{-\beta m+3\beta^2m}+2^{-l}2^{-m/2}2^{3\beta^2 m}.
\end{split}
\end{equation*}
Therefore, for \eqref{cle8} it suffices to prove that
\begin{equation}\label{cle13}
2^{\gamma j}\sup_{R\in[2^{-j},2^{k}],\,\xi_0\in\mathbb{R}^3}R^{-2}
\big\|\mathcal{F}\big[\widetilde{\varphi}^{(k)}_j\cdot \mathcal{F}^{-1}(G_l)\big]\big\|_{L^1(B(\xi_0,R))}\lesssim 2^{-3\beta^4m}.
\end{equation}
Since $\big|\mathcal{F}(\widetilde{\varphi}^{(k)}_j)(\xi)\big|\lesssim 2^{3j}(1+2^j|\xi|)^{-6}$, it follows from \eqref{cle7} that
\begin{equation*}
\begin{split}
\big|\mathcal{F}\big[\widetilde{\varphi}^{(k)}_j\cdot \mathcal{F}^{-1}(G_l)\big](\xi)\big|&\lesssim 
\int_{\mathbb{R}^3}|G_l(\xi-\eta)|\cdot 2^{3j}(1+2^j|\eta|)^{-6}\,d\eta\\
&\lesssim 2^{-l}2^{-m/2}2^{3\beta^2 m}\int_{\mathbb{R}^3}\mathbf{1}_{D_l}(\xi-\eta)\cdot 2^{3j}(1+2^j|\eta|)^{-6}\,d\eta.
\end{split}
\end{equation*}
Therefore, using now \eqref{cle9}, for any $R\in[2^{-j},2^k]$ and $\xi_0\in\mathbb{R}^3$,
\begin{equation*}
R^{-2}\big\|\mathcal{F}\big[\widetilde{\varphi}^{(k)}_j\cdot \mathcal{F}^{-1}(G_l)\big]\big\|_{L^1(B(\xi_0,R))}\lesssim 2^{-l}2^{-m/2}2^{3\beta^2 m}\cdot 2^{l-m}\lesssim 2^{-3m/2}2^{3\beta^2 m}.
\end{equation*}

Similarly, using \eqref{cle3} and \eqref{cle9},
\begin{equation*}
2^{(1-\beta)j}\Vert G_{l_0}\Vert_{L^2}+\Vert G_{l_0}\Vert_{L^\infty} \lesssim 2^{(1-\beta)(j-m)}2^{-\beta m+l_0/2+3\beta^2m}+2^{-m/4}\lesssim 2^{-3\beta^4m}
\end{equation*}
and
\begin{equation*}
\begin{split}
\big|\mathcal{F}\big[\widetilde{\varphi}^{(k)}_j\cdot \mathcal{F}^{-1}(G_{l_0})\big](\xi)\big|&\lesssim 
\int_{\mathbb{R}^3}|G_{l_0}(\xi-\eta)|\cdot 2^{3j}(1+2^j|\eta|)^{-6}\,d\eta\\
&\lesssim 2^{-m/2}2^{3\beta^2 m}\int_{\mathbb{R}^3}\mathbf{1}_{D_{l_0}}(\xi-\eta)\cdot 2^{3j}(1+2^j|\eta|)^{-6}\,d\eta
\end{split}
\end{equation*}
from where we conclude that, for any $R\in[2^{-j},2^k]$ and $\xi_0\in\mathbb{R}^3$,
\begin{equation*}
R^{-2}\big\|\mathcal{F}\big[\widetilde{\varphi}^{(k)}_j\cdot \mathcal{F}^{-1}(G_{l_0})\big]\big\|_{L^1(B(\xi_0,R))}\lesssim 2^{-m/2}2^{3\beta^2 m}\cdot 2^{l_0-m}\lesssim 2^{-3m/2}2^{2\beta m}.
\end{equation*}
The desired bound \eqref{cle13} follows, which completes the proof of the lemma.
\end{proof}

\begin{lemma}\label{BigBound8}
The bound \eqref{ok90} holds provided that \eqref{ok91} holds and, in addition,
\begin{equation}\label{cle20}
\max(j_1,j_2)\geq m(1/2-\beta^2).
\end{equation}
\end{lemma}

\begin{proof}[Proof of Lemma \ref{BigBound8}] Using definition \eqref{sec5.3}, it suffices to prove that
\begin{equation}\label{cle21}
2^{(1+\beta)j}\big\|\widetilde{\varphi}^{(k)}_j\cdot P_kT_{m}^{\sigma;\mu,\nu}(f_{k_1,j_1}^\mu,f_{k_2,j_2}^\nu)\big\|_{L^2}+\big\|\mathcal{F}[\widetilde{\varphi}^{(k)}_j\cdot P_kT_{m}^{\sigma;\mu,\nu}(f_{k_1,j_1}^\mu,f_{k_2,j_2}^\nu)]\big\|_{L^\infty}\lesssim 2^{-2\beta^4m}.
\end{equation}
Let 
\begin{equation}\label{cle22}
j'':=\max(j_1,j_2)+\lfloor3\beta^2m\rfloor\in[m(1/2+\beta^2),m(1-\beta/20)],
\end{equation}
and decompose
\begin{equation*}
\mathcal{F}P_kT_{m}^{\sigma;\mu,\nu}(f_{k_1,j_1}^\mu,f_{k_2,j_2}^\nu)(\xi)=G(\xi)+H_1(\xi)+H_2(\xi),
\end{equation*}
where
\begin{equation*}
H_2(\xi):=\varphi_k(\xi)\int_{\mathbb{R}}\int_{\mathbb{R}^3}e^{is\Phi^{\sigma;\mu,\nu}(\xi,\eta)}[1-\varphi(2^{30D}\Phi^{\sigma;\mu,\nu}(\xi,\eta))]q_m(s)\widehat{f_{k_1,j_1}^\mu}(\xi-\eta,s)\widehat{f_{k_2,j_2}^\nu}(\eta,s)\,d\eta ds,
\end{equation*}
\begin{equation*}
\begin{split}
G(\xi):=\varphi_k(\xi)\int_{\mathbb{R}}\int_{\mathbb{R}^3}&e^{is\Phi^{\sigma;\mu,\nu}(\xi,\eta)}\varphi(2^{30D}\Phi^{\sigma;\mu,\nu}(\xi,\eta))\varphi_{\leq 0}
(2^{m-j''}\Xi^{\mu,\nu}(\xi,\eta))\\
&\times q_m(s)\widehat{f_{k_1,j_1}^\mu}(\xi-\eta,s)\widehat{f_{k_2,j_2}^\nu}(\eta,s)\,d\eta ds,
\end{split}
\end{equation*}
and
\begin{equation*}
\begin{split}
H_1(\xi):=\varphi_k(\xi)\int_{\mathbb{R}}\int_{\mathbb{R}^3}&e^{is\Phi^{\sigma;\mu,\nu}(\xi,\eta)}\varphi(2^{30D}\Phi^{\sigma;\mu,\nu}(\xi,\eta))[1-\varphi_{\leq 0}
(2^{m-j''}\Xi^{\mu,\nu}(\xi,\eta))]\\
&\times q_m(s)\widehat{f_{k_1,j_1}^\mu}(\xi-\eta,s)\widehat{f_{k_2,j_2}^\nu}(\eta,s)\,d\eta ds.
\end{split}
\end{equation*}

Using Lemma \ref{tech5} (with $K\approx 2^{j''}$ and $\eps\approx 2^{-\max(j_1,j_2)}$) and the last bound in \eqref{nh9} it is easy to see 
that $\|H_1\|_{L^\infty}\lesssim 2^{-10m}$. Moreover, the same argument as in the first part of the proof of Lemma \ref{BigBound6} (which does not use the assumption $\min(k,k_1,k_2)\leq-D/10$) shows that
\begin{equation*}
2^{(1+\beta)m}\|H_2\|_{L^2}+\|H_2\|_{L^\infty}\lesssim 2^{-2\beta^4 m}.
\end{equation*}
Therefore it remains to prove that
\begin{equation}\label{cle23}
2^{(1+\beta)m}\big\|G\big\|_{L^2}+\|G\|_{L^\infty}\lesssim 2^{-2\beta^4m}.
\end{equation}

In proving \eqref{cle23} we may assume that $G\neq 0$; in particular this guarantees that the main assumption \eqref{kx10} is satisfied. 
We prove first the $L^\infty$ bound in \eqref{cle23}. Assume that $j_1\leq j_2$ (the case $j_1\geq j_2$ is similar). Then, see \eqref{nh9} and \eqref{sec5.8}--\eqref{sec5.815},
\begin{equation*}
 \begin{split}
&\|\widehat{f_{k_1,j_1}^\mu}(s)\|_{L^\infty}\lesssim 1,\\
&\sup_{\xi_0\in\mathbb{R}^3}\|\widehat{f_{k_2,j_2}^\nu}(s)\|_{L^1(B(\xi_0,R))}\lesssim 2^{-(1+\beta)j_2}R^{3/2},\qquad\text{ for any }R\leq 1.
 \end{split}
\end{equation*}
Using Lemma \ref{desc2} and \eqref{cle22} it follows that
\begin{equation*}
\|G\|_{L^\infty}\lesssim 2^m\cdot 2^{-(1+\beta)j_2}(2^{j''-m})^{3/2}\lesssim 2^{-m/2}2^{4\beta^2 m}2^{(1/2-\beta)j''}\lesssim 2^{-2\beta^4m},
\end{equation*}
as desired.

To prove the $L^2$ bound in \eqref{cle23} it suffices to show that
\begin{equation}\label{cle24}
2^{(2+2\beta)m}\|G\|^2_{L^2}\lesssim 2^{-4\beta^4m}.
\end{equation}
To prove this we need first an orthogonality argument. Let $\chi:\mathbb{R}\to[0,1]$ denote a smooth function supported in the interval $[-2,2]$ with the 
property that
\begin{equation*}
\sum_{n\in\mathbb{Z}}\chi(x-n)=1\qquad\text{ for any }x\in\mathbb{R}.
\end{equation*}
We define the smooth function $\chi':\mathbb{R}^3\to[0,1]$, $\chi'(x,y,z):=\chi(x)\chi(y)\chi(z)$. Recall the functions $\Psi^{\sigma;\mu,\nu}$
defined in \eqref{kxz11}. We define, for any $v\in\mathbb{Z}^3$ and $n\in\mathbb{Z}$,
\begin{equation}\label{cle25}
\begin{split}
G_{v,n}(\xi):=\chi'(2^{m-j''}\xi-v)\cdot \varphi_k(\xi)\int_{\mathbb{R}}\int_{\mathbb{R}^3}&e^{is\Phi^{\sigma;\mu,\nu}(\xi,\eta)}
\varphi(2^{30D}\Phi^{\sigma;\mu,\nu}(\xi,\eta))\varphi_{\leq 0}(2^{m-j''}\Xi^{\mu,\nu}(\xi,\eta))\\
&\times\chi(2^{-j''}s-n)q_m(s)\widehat{f_{k_1,j_1}^\mu}(\xi-\eta,s)\widehat{f_{k_2,j_2}^\nu}(\eta,s)\,d\eta ds.
\end{split}
\end{equation}
and notice that $G=\sum_{v\in\mathbb{Z}^3}\sum_{n\in\mathbb{Z}}G_{v,n}$. 

We show now that
\begin{equation}\label{cle26}
 \|G\|^2_{L^2}\lesssim \sum_{v\in\mathbb{Z}^3}\sum_{n\in\mathbb{Z}}\|G_{v,n}\|_{L^2}^2+2^{-10m}.
\end{equation}
Indeed, we clearly have
\begin{equation*}
\|G\|^2_{L^2}\lesssim \sum_{v\in\mathbb{Z}^3}\Big\|\sum_{n\in\mathbb{Z}}G_{v,n}\Big\|_{L^2}^2
\lesssim \sum_{v\in\mathbb{Z}^3}\sum_{n_1,n_2\in\mathbb{Z}}|\langle G_{v,n_1},G_{v,n_2}\rangle|.
\end{equation*}
Therefore, for \eqref{cle26} it suffices to prove that
\begin{equation}\label{cle27}
|\langle G_{v,n_1},G_{v,n_2}\rangle|\lesssim 2^{-20m}\qquad \text{ if }v\in\mathbb{Z}^3\text{ and }|n_1-n_2|\geq 2^{100D}.
\end{equation}
To prove this, we notice that, since $\vert\nabla_\eta\Phi^{\sigma;\mu,\nu}\vert\le 2^{j^{\prime\prime}-m}$ and $\vert\partial^\rho\Phi^{\sigma;\mu,\nu}\vert\lesssim 1$ for $\vert\rho\vert=2$, after repeated integration by parts in $\xi$, for any $n\in\mathbb{Z}$,
\begin{equation*}
\begin{split}
&|\mathcal{F}^{-1}(G_{v,n})(x)|\lesssim |x+w_n|^{-200}\qquad\text{ if }|x+w_n|\geq 2^{50D}2^{j''},\\
&w_n:=n2^{j''}(\Psi^{\sigma;\mu,\nu})'(2^{j''-m}|v|)\cdot v/|v|.
\end{split}
\end{equation*}
Moreover, $G_{v,n}$ is nontrivial only if $|\Psi^{\sigma;\mu,\nu}(2^{j''-m}|v|)|\leq 2^{-25D}$. We can therefore apply Lemma \ref{desc4} to conclude that $|(\Psi^{\sigma;\mu,\nu})'(2^{j''-m}|v|)|\geq 2^{-20D}$. Therefore if $|n_1-n_2|\geq 2^{100D}$ then $|w_{n_1}-w_{n_2}|\geq 2^{70D}2^{j''}$ and the desired bound \eqref{cle27} follows. This completes the proof of \eqref{cle26}.

In view of \eqref{cle26}, for \eqref{cle24} it remains to prove that
\begin{equation}\label{cle35}
2^{(2+2\beta)m}\sum_{2^{-k}|v|,\,n\in[2^{m-j''-4},2^{m-j''+4}]}\|G_{v,n}\|^2_{L^2}\lesssim 2^{-4\beta^4m}.
\end{equation}
Assuming $v,n$ fixed, the variables in the definition of the function $G_{v,n}$ are naturally restricted as follows:
\begin{equation*}
|\xi-2^{j''-m}v|\lesssim 2^{j''-m},\qquad |\eta-p^{\mu,\nu}(2^{j''-m}v)|\lesssim 2^{j''-m},\qquad |s-2^{j''}n|\lesssim 2^{j''},
\end{equation*}
where $p^{\mu,\nu}$ is defined as in Lemma \ref{desc2}. More precisely, we define the functions $f_1^{v,n}$ and $f_2^{v,n}$ by the formulas
\begin{equation}\label{cle36}
\begin{split}
&\widehat{f_1^{v,n}}(\theta,s):=\mathbf{1}_{[n-4,n+4]}(2^{-j''}s)\varphi_{\leq 0}[2^{-50D}2^{m-j''}(\theta-2^{j''-m}v+p^{\mu,\nu}(2^{j''-m}v))]\cdot \widehat{f_{k_1,j_1}^\mu}(\theta,s),\\
&\widehat{f_2^{v,n}}(\theta,s):=\mathbf{1}_{[n-4,n+4]}(2^{-j''}s)\varphi_{\leq 0}[2^{-50D}2^{m-j''}(\theta-p^{\mu,\nu}(2^{j''-m}v))]\cdot \widehat{f_{k_2,j_2}^\nu}(\theta,s).
\end{split}
\end{equation}
Since $|p^{\mu,\nu}(2^{j''-m}v_1)-p^{\mu,\nu}(2^{j''-m}v_2)|\geq 2^{80D}2^{j''-m}$ and $\big|[2^{j''-m}v_1-p^{\mu,\nu}(2^{j''-m}v_1)]-[2^{j''-m}v_2-p^{\mu,\nu}(2^{j''-m}v_2)]\big|\geq 2^{80D}2^{j''-m}$ whenever $|v_1-v_2|\gtrsim 1$ (these inequalities are consequences of the lower bounds in the first line of \eqref{r1to1}), it follows by orthogonality that, for any $s\in\mathbb{R}$,
\begin{equation}\label{cle36.5}
\begin{split}
&\sum_{2^{-k}|v|\in[2^{m-j''-4},2^{m-j''+4}]}\|f_1^{v,n}(s)\|_{L^2}^2\lesssim \|f_{k_1,j_1}^\mu(s)\|_{L^2}^2\lesssim 2^{-2j_1+2\beta j_1},\\
&\sum_{2^{-k}|v|\in[2^{m-j''-4},2^{m-j''+4}]}\|f_2^{v,n}(s)\|_{L^2}^2\lesssim \|f_{k_2,j_2}^\nu(s)\|_{L^2}^2\lesssim 2^{-2j_2+2\beta j_2}.
\end{split}
\end{equation}

Using the definition \eqref{cle25} and Lemma \ref{desc2} we notice that, for any $(v,n)\in\mathbb{Z}^3\times\mathbb{Z}$,
\begin{equation}\label{cle37}
\begin{split}
G_{v,n}(\xi)=\chi'(2^{m-j''}\xi-v)\cdot \varphi_k(\xi)\int_{\mathbb{R}}\int_{\mathbb{R}^3}&e^{is\Phi^{\sigma;\mu,\nu}(\xi,\eta)}
\varphi(2^{30D}\Phi^{\sigma;\mu,\nu}(\xi,\eta))\varphi_{\leq 0}(2^{m-j''}\Xi^{\mu,\nu}(\xi,\eta))\\
&\times\chi(2^{-j''}s-n)q_m(s)\widehat{f_1^{v,n}}(\xi-\eta,s)\widehat{f_2^{v,n}}(\eta,s)\,d\eta ds.
\end{split}
\end{equation}
Letting, as in \eqref{nh5.6}, $(Ef_1^{v,n})(s):=e^{-is\widetilde{\Lambda}_\mu}(f_1^{v,n}(s))$ and $(Ef_2^{v,n})(s):=e^{-is\widetilde{\Lambda}_\nu}(f_2^{v,n}(s))$, it follows that
\begin{equation*}
\|G_{v,n}\|_{L^2}\lesssim \int_{\mathbb{R}}\chi(2^{-j''}s-n)q_m(s)\|A_{v}(Ef_1^{v,n}(s),Ef_2^{v,n}(s))\|_{L^2}\,ds,
\end{equation*}
where, by definition,
\begin{equation}\label{cle38}
\begin{split}
A_{v}(g_1,g_2)(\xi):=\chi'(2^{m-j''}\xi-v)\varphi_k(\xi)\int_{\mathbb{R}^3}&\varphi(2^{30D}\Phi^{\sigma;\mu,\nu}(\xi,\eta))\varphi_{\leq 0}(2^{m-j''}\Xi^{\mu,\nu}(\xi,\eta))\\
&\times \mathcal{F}(P_{[k_1-4,k_1+4]}g_1)(\xi-\eta)\mathcal{F}(P_{[k_2-4,k_2+4]}g_2)(\eta)\,d\eta.
\end{split}
\end{equation}
Therefore
\begin{equation*}
\|G_{v,n}\|_{L^2}^2\lesssim 2^{j''}\int_{\mathbb{R}}q_m(s)\|A_{v}(Ef_1^{v,n}(s),Ef_2^{v,n}(s))\|^2_{L^2}\,ds,
\end{equation*}
and for \eqref{cle35} it suffices to prove that
\begin{equation}\label{cle39}
2^{2m+2\beta m}2^{j''}\sum_{2^{-k}|v|,\,n\in[2^{m-j''-4},2^{m-j''+4}]}\int_{\mathbb{R}}\|A_{v}(Ef_1^{v,n}(s),Ef_2^{v,n}(s))\|^2_{L^2}\,ds\lesssim 2^{-4\beta^4m}.
\end{equation}

We notice now that if $p,q\in[2,\infty]$, $1/p+1/q=1/2$, then
\begin{equation}\label{cle40}
\|A_{v}(g_1,g_2)\|_{L^2}\lesssim \|g_1\|_{L^p}\|g_2\|_{L^q}.
\end{equation}
Indeed, as in the proof of Lemma \ref{tech2}, we write
\begin{equation*}
\mathcal{F}^{-1}(A_{v}(g_1,g_2))(x)=c\int_{\mathbb{R}^3\times\mathbb{R}^3}g_1(y)g_2(z)K_v(x;y,z)\,dydz,
\end{equation*}
where
\begin{equation*}
\begin{split}
K_v(x;y,z):=\int_{\mathbb{R}^3\times\mathbb{R}^3}&e^{i(x-y)\cdot \xi}e^{i(y-z)\cdot \eta}\chi'(2^{m-j''}\xi-v)\varphi_{\leq 0}(2^{m-j''}\Xi^{\mu,\nu}(\xi,\eta))\\
&\times \varphi_k(\xi)\varphi(2^{30D}\Phi^{\sigma;\mu,\nu}(\xi,\eta))\varphi_{[k_1-4,k_1+4]}(\xi-\eta)\varphi_{[k_2-4,k_2+4]}(\eta)\,d\xi d\eta.
\end{split}
\end{equation*}
We recall that $k,k_1,k_2\in[-D/10,D]$ and integrate by parts in $\xi$ and $\eta$. Using also Lemma \ref{desc2}, it follows that
\begin{equation*}
|K_v(x;y,z)|\lesssim 2^{3(j''-m)}(1+2^{j''-m}|x-y|)^{-4}\cdot 2^{3(j''-m)}(1+2^{j''-m}|y-z|)^{-4},
\end{equation*}
and the desired estimate \eqref{cle40} follows.

We can now prove the main estimate \eqref{cle39}. Assume first that
\begin{equation}\label{cle41}
\max(j_1,j_2)-\min(j_1,j_2)\geq 10\beta m.
\end{equation}
By symmetry, we may assume that $j_1\leq j_2$ and estimate, using \eqref{mk15.6}--\eqref{mk15.65},
\begin{equation*}
\sup_{s\in\mathbb{R}}\|Ef_1^{v,n}(s)\|_{L^\infty}\lesssim 2^{-3m/2}2^{(1/2+\beta)j_1}.
\end{equation*}
Therefore, using \eqref{cle40} and \eqref{cle36.5}, the left-hand side of \eqref{cle39} is dominated by
\begin{equation*}
\begin{split}
C2^{2m+2\beta m}2^{j''}&\sum_{2^{-k}|v|,\,n\in[2^{m-j''-4},2^{m-j''+4}]}2^{-3m}2^{(1+2\beta)j_1}\int_{\mathbb{R}}\|Ef_2^{v,n}(s)\|^2_{L^2}\,ds\\
&\lesssim C2^{2m+2\beta m}2^{j''}\cdot 2^{-3m}2^{(1+2\beta)j_1}\cdot 2^{-2j_2+2\beta j_2}\cdot 2^m\\
&\lesssim 2^{j_1-j_2}2^{2\beta m}2^{2\beta j_1}2^{2\beta j_2}2^{j''-j_2},
\end{split}
\end{equation*}
and the desired bound \eqref{cle39} follows provided that \eqref{cle41} holds.

Assume now that
\begin{equation}\label{cle45}
\max(j_1,j_2)\leq (3/5-2\beta) m.
\end{equation}
By symmetry, we may assume again that $j_1\leq j_2$ and estimate 
\begin{equation*}
\sup_{s\in\mathbb{R}}\|Ef_1^{v,n}(s)\|_{L^\infty}\lesssim \sup_{s\in\mathbb{R}}\|\widehat{f_1^{v,n}}(s)\|_{L^1}\lesssim 2^{3j''-3m}.
\end{equation*}
Therefore, using \eqref{cle40} and \eqref{cle36.5}, the left-hand side of \eqref{cle39} is dominated by
\begin{equation*}
\begin{split}
C2^{2m+2\beta m}2^{j''}&\sum_{2^{-k}|v|,\,n\in[2^{m-j''-4},2^{m-j''+4}]}2^{-6m}2^{6j''}\int_{\mathbb{R}}\|Ef_2^{v,n}(s)\|^2_{L^2}\,ds\\
&\lesssim C2^{2m+2\beta m}2^{j''}\cdot 2^{-6m}2^{6j''}\cdot 2^{-2j_2+2\beta j_2}\cdot 2^m\\
&\lesssim 2^{-3m}2^{5j_2}2^{7(j''-j_2)}2^{2\beta m}2^{2\beta j_2},
\end{split}
\end{equation*}
and the desired bound \eqref{cle39} follows provided that \eqref{cle45} holds.

Finally, assume that
\begin{equation}\label{cle47}
\max(j_1,j_2)-\min(j_1,j_2)\leq 10\beta m\qquad\text{ and }\qquad\max(j_1,j_2)\geq (3/5-2\beta) m.
\end{equation}
In this case we need the more refined decomposition in \eqref{sec5.8}--\eqref{sec5.815}. More precisely, using the definitions we decompose
\begin{equation*}
f^\mu_{k_1,j_1}(s)=P_{[k_1-2,k_1+2]}(g_1(s)+h_1(s)),\qquad f^\nu_{k_2,j_2}(s)=P_{[k_2-2,k_2+2]}(g_2(s)+h_2(s)),
\end{equation*}
where\footnote{The decomposition in \eqref{sec5.8}--\eqref{sec5.815} provides some more information about the functions $g_1,h_1,g_2,h_2$, but only \eqref{cle48} and \eqref{cle49} are being used in the proof.}
\begin{equation}\label{cle48}
g_1(s)=g_1(s)\cdot \widetilde{\varphi}^{(k_1)}_{[j_1-2,j_1+2]},\quad g_2(s)=g_2(s)\cdot \widetilde{\varphi}^{(k_2)}_{[j_2-2,j_2+2]},
\end{equation}
and
\begin{equation}\label{cle49}
\begin{split}
&2^{(1+\beta)j_1}\|g_1(s)\|_{L^2}+2^{(1-\beta)j_1}\|h_1(s)\|_{L^2}+2^{\gamma j_1}\sup_{R\in[2^{-j_1},2^{k_1}],\theta_0\in\mathbb{R}^3}R^{-2}\|\widehat{h_1}(s)\|_{L^1(B(\theta_0,R))}\lesssim 1,\\
&2^{(1+\beta)j_2}\|g_2(s)\|_{L^2}+2^{(1-\beta)j_2}\|h_2(s)\|_{L^2}+2^{\gamma j_2}\sup_{R\in[2^{-j_2},2^{k_2}],\theta_0\in\mathbb{R}^3}R^{-2}\|\widehat{h_2}(s)\|_{L^1(B(\theta_0,R))}\lesssim 1.
\end{split}
\end{equation}
Then, we define the functions $g_1^{v,n}, h_1^{v,n}, g_2^{v,n}, h_2^{v,n}$ by the formulas (compare with \eqref{cle36}),
\begin{equation}\label{cle50}
\begin{split}
&\widehat{g_1^{v,n}}(\theta,s):=\mathbf{1}_{[n-4,n+4]}(2^{-j''}s)\varphi_{\leq 0}[2^{-50D}2^{m-j''}(\theta-2^{j''-m}v+p^{\mu,\nu}(2^{j''-m}v))]\cdot \mathcal{F}(P_{[k_1-2,k_1+2]}g_1)(\theta,s),\\
&\widehat{h_1^{v,n}}(\theta,s):=\mathbf{1}_{[n-4,n+4]}(2^{-j''}s)\varphi_{\leq 0}[2^{-50D}2^{m-j''}(\theta-2^{j''-m}v+p^{\mu,\nu}(2^{j''-m}v))]\cdot \mathcal{F}(P_{[k_1-2,k_1+2]}h_1)(\theta,s),\\
&\widehat{g_2^{v,n}}(\theta,s):=\mathbf{1}_{[n-4,n+4]}(2^{-j''}s)\varphi_{\leq 0}[2^{-50D}2^{m-j''}(\theta-p^{\mu,\nu}(2^{j''-m}v))]\cdot \mathcal{F}(P_{[k_2-2,k_2+2]}g_2)(\theta,s),\\
&\widehat{h_2^{v,n}}(\theta,s):=\mathbf{1}_{[n-4,n+4]}(2^{-j''}s)\varphi_{\leq 0}[2^{-50D}2^{m-j''}(\theta-p^{\mu,\nu}(2^{j''-m}v))]\cdot \mathcal{F}(P_{[k_2-2,k_2+2]}h_2)(\theta,s).
\end{split}
\end{equation}
As in \eqref{cle36.5}, using $L^2$ orthogonality and \eqref{cle49}, for any $s\in\mathbb{R}$ we have
\begin{equation}\label{cle50.5}
\begin{split}
&\sum_{2^{-k}|v|\in[2^{m-j''-4},2^{m-j''+4}]}\|g_1^{v,n}(s)\|_{L^2}^2\lesssim 2^{-2j_1-2\beta j_1},\quad \sum_{2^{-k}|v|\in[2^{m-j''-4},2^{m-j''+4}]}\|h_1^{v,n}(s)\|_{L^2}^2\lesssim 2^{-2j_1+2\beta j_1},\\
&\sum_{2^{-k}|v|\in[2^{m-j''-4},2^{m-j''+4}]}\|g_2^{v,n}(s)\|_{L^2}^2\lesssim 2^{-2j_2-2\beta j_2}, \quad \sum_{2^{-k}|v|\in[2^{m-j''-4},2^{m-j''+4}]}\|h_2^{v,n}(s)\|_{L^2}^2\lesssim 2^{-2j_2+2\beta j_2}.
\end{split}
\end{equation}

Using \eqref{mk15} and \eqref{cle48}--\eqref{cle49}, we derive the $L^\infty$ bounds
\begin{equation}\label{cle51}
\begin{split}
&\|Eg_1^{v,n}(s)\|_{L^\infty}\lesssim 2^{-3m/2}\|g_1(s)\|_{L^1}\lesssim 2^{-3m/2}2^{(1/2-\beta)j_1},\\
&\|Eh_1^{v,n}(s)\|_{L^\infty}\lesssim \|\widehat{h^{v,n}_1}(s)\|_{L^1}\lesssim 2^{2j''-2m}2^{-\gamma j_1},\\
&\|Eg_2^{v,n}(s)\|_{L^\infty}\lesssim 2^{-3m/2}\|g_2(s)\|_{L^1}\lesssim 2^{-3m/2}2^{(1/2-\beta)j_2},\\
&\|Eh_2^{v,n}(s)\|_{L^\infty}\lesssim \|\widehat{h^{v,n}_2}(s)\|_{L^1}\lesssim 2^{2j''-2m}2^{-\gamma j_2},
\end{split}
\end{equation}
for any $v,n,s$. Using \eqref{cle40} and \eqref{cle50.5}--\eqref{cle51}, we estimate, assuming $j_1\leq j_2$,
\begin{equation*}
\begin{split}
& 2^{2m+2\beta m}2^{j''}\sum_{2^{-k}|v|,\,n\in[2^{m-j''-4},2^{m-j''+4}]}\int_{\mathbb{R}}\|A_{v}(Ef_1^{v,n}(s),Eg_2^{v,n}(s))\|^2_{L^2}\,ds\\
&\lesssim 2^{2m+2\beta m}2^{j''}\sum_{2^{-k}|v|,\,n\in[2^{m-j''-4},2^{m-j''+4}]}\int_{\mathbb{R}}\|g_2^{v,n}(s)\|^2_{L^2}(\|Eg_1^{v,n}(s)\|_{L^\infty}^2+\|Eh_1^{v,n}(s)\|_{L^\infty}^2)\,ds\\
&\lesssim 2^{2m+2\beta m}2^{j''}\cdot 2^{m}2^{-2j_2-2\beta j_2}\cdot [2^{-3m}2^{(1-2\beta)j_1}+2^{4j''-4m}2^{-2\gamma j_1}]\\
&\lesssim 2^{2\beta m}2^{j''}2^{-(1+4\beta) j_2}+2^{3\beta m}2^{2j_2}2^{-2\gamma j_1}.
\end{split}
\end{equation*}
Similarly, we estimate
\begin{equation*}
\begin{split}
& 2^{2m+2\beta m}2^{j''}\sum_{2^{-k}|v|,\,n\in[2^{m-j''-4},2^{m-j''+4}]}\int_{\mathbb{R}}\|A_{v}(Ef_1^{v,n}(s),Eh_2^{v,n}(s))\|^2_{L^2}\,ds\\
&\lesssim 2^{2m+2\beta m}2^{j''}\sum_{2^{-k}|v|,\,n\in[2^{m-j''-4},2^{m-j''+4}]}\int_{\mathbb{R}}(\|g_1^{v,n}(s)\|_{L^2}^2+\|h_1^{v,n}(s)\|_{L^2}^2)\|Eh_2^{v,n}(s)\|^2_{L^\infty}\,ds\\
&\lesssim 2^{2m+2\beta m}2^{j''}\cdot 2^m2^{-2 j_1+2\beta j_1}\cdot 2^{4j''-4m}2^{-2\gamma j_2}\\
&\lesssim 2^{5\beta m}2^{-2j_1}2^{(4-2\gamma)j_2}.
\end{split}
\end{equation*}
The desired estimate \eqref{cle39} follows from the last two bounds and the restriction \eqref{cle47}. This completes the proof of the lemma.
\end{proof}

\section{Technical estimates}\label{technical}

In this section we collect several technical estimates that are used at various stages of the argument. 

\subsection{Linear and bilinear estimates}

We prove now some important linear and bilinear estimates, which are used repeatedly in the paper. We show first that our main spaces constructed in Definition \ref{MainDef} are 
compatible with normalized Calderon--Zygmund operators.

\begin{lemma}\label{tech3}
If $Q$ is a normalized Calderon--Zygmund operator (see \eqref{symb}--\eqref{CZ}) then
\begin{equation}\label{compat}
\|Qf\|_{Z}\lesssim \|f\|_{Z},\qquad\text{ for any }f\in Z.
\end{equation}
\end{lemma}

\begin{proof}[Proof of Lemma \ref{tech3}] We may assume that $\|f\|_{Z}\leq 1$ and it suffices to prove that
\begin{equation}\label{compat2}
\|\widetilde{\varphi}^{(k)}_j\cdot P_kQf\|_{B_{k,j}}\lesssim 1,
\end{equation}
for any $(k,j)\in\mathcal{J}$ fixed.

We have
\begin{equation}\label{compat2.1}
\widetilde{\varphi}^{(k)}_j(x)\cdot P_kQf(x)=\widetilde{\varphi}^{(k)}_j(x)\int_{\mathbb{R}^3}P_kf(y)\cdot K_k(x-y)\,dy,
\end{equation}
where
\begin{equation*}
K_k(z)=c\int_{\mathbb{R}^3}e^{iz\cdot\xi}q(\xi)\varphi_{[k-1,k+1]}(\xi)\,d\xi.
\end{equation*}
Clearly,
\begin{equation}\label{compat2.5}
|K_k(z)|\lesssim 2^{3k}(1+2^k|z|)^{-6}.
\end{equation}

As before, let $\widetilde{k}=\min(k,0)$, $k_+=\max(k,0)$. Since $\|\widetilde{\varphi}^{(k)}_{j'}\cdot P_kf\|_{B_{k,j'}}\leq 1$ for any $j'\geq-\widetilde{k}$, we can 
decompose, as in \eqref{sec5.8}--\eqref{sec5.82},
\begin{equation}\label{compat6}
\begin{split}
&\widetilde{\varphi}^{(k)}_{j'}\cdot P_kf=g_{1,j'}+g_{2,j'},\qquad g_{1,j'}=g_{1,j'}\cdot \widetilde{\varphi}^{(k)}_{[j'-2,j'+2]},\qquad g_{2,j'}=g_{2,j'}\cdot \widetilde{\varphi}^{(k)}_{[j'-2,j'+2]},\\
&2^{(1+\be)j'}\|g_{1,j'}\|_{L^2}+2^{(1/2-\beta)\widetilde{k}}\|\widehat{g_{1,j'}}\|_{L^\infty}\lesssim (2^{\alpha k}+2^{10k})^{-1},\\
&2^{-2\beta\widetilde{k}}2^{(1-\be)j'}\|g_{2,j'}\|_{L^2}+2^{(1/2-\beta)\widetilde{k}}\|\widehat{g_{2,j'}}\|_{L^\infty}
+2^{(\gamma-\beta-5/2)\widetilde{k}}2^{\gamma j'}\|\widehat{g_{2,j'}}\|_{L^1}\lesssim (2^{\alpha k}+2^{10k})^{-1},
\end{split}
\end{equation}
and, moreover,
\begin{equation}\label{compa6}
2^{(\gamma-\beta-1/2)\widetilde{k}}2^{2k_+}2^{\gamma j'}\sup_{R\in[2^{-j'},2^k],\,\xi_0\in\mathbb{R}^3}
R^{-2}\|\widehat{g_{2,j'}}\|_{L^1(B(\xi_0,R))}\lesssim (2^{\alpha k}+2^{10k})^{-1}.
\end{equation}

Then we decompose, using the formulas \eqref{compat2.1} and \eqref{compat6}, 
\begin{equation}\label{compat6.1}
\begin{split}
&\widetilde{\varphi}^{(k)}_j(x)\cdot P_kQf(x)=G_1+G_2,\\
&G_1(x):=\sum_{j'\geq -\widetilde{k}}\widetilde{\varphi}^{(k)}_j(x)\cdot (g_{1,j'}\ast K_k)(x)+\sum_{j'\geq-\widetilde{k},\,|j'-j|\geq 4}\widetilde{\varphi}^{(k)}_j(x)\cdot (g_{2,j'}\ast K_k)(x),\\
&G_2(x):=\sum_{j'\geq-\widetilde{k},\,|j'-j|\leq 3}\widetilde{\varphi}^{(k)}_j(x)\cdot (g_{2,j'}\ast K_k)(x).
\end{split}
\end{equation}
In view of the definitions, for \eqref{compat2} it suffices to prove that
\begin{equation}\label{compat6.2}
\|G_1\|_{B^1_{k,j}}+\|G_2\|_{B^2_{k,j}}\lesssim 1.
\end{equation}

To prove the bound $\|G_1\|_{B^1_{k,j}}\lesssim 1$ we notice first that
\begin{equation*}
\begin{split}
&\sum_{j'\geq-\widetilde{k},\,|j'-j|\leq 3}\|\widetilde{\varphi}^{(k)}_j\cdot (g_{1,j'}\ast K_k)\|_{L^2}
\lesssim \sum_{j'\geq-\widetilde{k},\,|j'-j|\leq 3}\|g_{1,j'}\|_{L^2}\lesssim (2^{\alpha k}+2^{10k})^{-1}2^{-(1+\beta)j},\\
&\sum_{j'\geq-\widetilde{k},\,|j'-j|\leq 3}\|\mathcal{F}[\widetilde{\varphi}_j^{(k)}\cdot (g_{1,j'}\ast K_k)]\|_{L^\infty}
\lesssim \sum_{j'\geq-\widetilde{k},\,|j'-j|\leq 3}\|\widehat{g_{1,j'}}\|_{L^\infty}\lesssim (2^{\alpha k}+2^{10k})^{-1}2^{-(1/2-\beta)\widetilde{k}}.
\end{split}
\end{equation*}
Therefore it remains to prove that
\begin{equation}\label{compat6.3}
 \begin{split}
&\sum_{j'\geq-\widetilde{k},\,|j'-j|\geq 4}\Big[\|\widetilde{\varphi}_j^{(k)}\cdot (g_{1,j'}\ast K_k)\|_{L^2}+\|\widetilde{\varphi}^{(k)}_j\cdot (g_{2,j'}\ast K_k)\|_{L^2}\Big]\lesssim (2^{\alpha k}+2^{10k})^{-1}2^{-(1+\beta)j},\\
&\sum_{j'\geq-\widetilde{k},\,|j'-j|\geq 4}\Big[\|\mathcal{F}[\widetilde{\varphi}^{(k)}_j\cdot (g_{1,j'}\ast K_k)]\|_{L^\infty}+\|\mathcal{F}[\widetilde{\varphi}^{(k)}_j\cdot (g_{2,j'}\ast K_k)]\|_{L^\infty}\Big]
\lesssim (2^{\alpha k}+2^{10k})^{-1}2^{-(1/2-\beta)\widetilde{k}}.
 \end{split}
\end{equation}
Since
\begin{equation*}
\|\mathcal{F}[\widetilde{\varphi}^{(k)}_j\cdot h]\|_{L^\infty}\lesssim \|\widetilde{\varphi}^{(k)}_j\cdot h\|_{L^1}
\lesssim 2^{3j/2}\|\widetilde{\varphi}^{(k)}_j\cdot h\|_{L^2},
\end{equation*}
for \eqref{compat6.3} it suffices to prove that 
\begin{equation}\label{compat6.4}
\sum_{j'\geq-\widetilde{k},\,|j'-j|\geq 4}\Big[\|\widetilde{\varphi}_j^{(k)}\cdot (g_{1,j'}\ast K_k)\|_{L^2}+\|\widetilde{\varphi}_j^{(k)}\cdot (g_{2,j'}\ast K_k)\|_{L^2}\Big]
\lesssim (2^{\alpha k}+2^{10k})^{-1}2^{-3j/2}2^{-(1/2-\beta)\widetilde{k}}.
\end{equation}
Notice that if $|j-j'|\geq 4$ and $\mu\in\{1,2\}$ then
\begin{equation*}
\widetilde{\varphi}^{(k)}_j(x)\cdot (g_{\mu,j'}\ast K_k)(x)=\widetilde{\varphi}_j^{(k)}(x)\cdot (g_{\mu,j'}\ast K_{k,j,j'})(x)\quad\text{ where }
\quad K_{k,j,j'}(z):=K_k(z)\cdot \varphi_{[\max(j,j')-10,\infty)}(z).
\end{equation*}
Therefore, using \eqref{compat2.5}, 
\begin{equation*}
\begin{split}
 \sum_{j'\geq-\widetilde{k},\,|j'-j|\geq 4}&\Big[\|\widetilde{\varphi}^{(k)}_j\cdot (g_{1,j'}\ast K_k)\|_{L^2}+\|\widetilde{\varphi}^{(k)}_j\cdot (g_{2,j'}\ast K_k)\|_{L^2}\Big]\\
 &\lesssim 2^{3j/2}\sum_{j'\geq-\widetilde{k},\,|j'-j|\geq 4}\Big[(\|g_{1,j'}\|_{L^1}+\|g_{2,j'}\|_{L^1})\|K_{k,j,j'}\|_{L^\infty}\Big]\\
&\lesssim 2^{3j/2}\sum_{j'\geq-\widetilde{k}}2^{3j'/2}(2^{\alpha k}+2^{10k})^{-1}\cdot 2^{-(1-\beta)j'}2^{2\beta\widetilde{k}}\cdot 2^{3k}(1+2^k2^{\max(j,j')})^{-6}\\
&\lesssim (2^{\alpha k}+2^{10k})^{-1}2^{-3j/2}2^{-(1/2-\beta)\widetilde{k}}\cdot 2^{-|k+j|},
\end{split}
\end{equation*}
which suffices to prove the desired bound \eqref{compat6.4}.

To prove the bound $\|G_2\|_{B^2_{k,j}}\lesssim 1$ in \eqref{compat6.2} we notice first the
\begin{equation*}
\begin{split}
&\|G_2\|_{L^2}\lesssim \sum_{j'\geq-\widetilde{k},\,|j'-j|\leq 3}\|g_{2,j'}\|_{L^2}\lesssim (2^{\alpha k}+2^{10k})^{-1}2^{-(1-\beta)j}2^{2\beta\widetilde{k}},\\
&\|\widehat{G_2}\|_{L^\infty}\lesssim \sum_{j'\geq-\widetilde{k},\,|j'-j|\leq 3}\|\widehat{g_{2,j'}}\|_{L^\infty}
\lesssim (2^{\alpha k}+2^{10k})^{-1}2^{-(1/2-\beta)\widetilde{k}},
\end{split}
\end{equation*}
using the assumptions on $g_{2,j'}$ in \eqref{compat6}. Therefore it remains to prove that
\begin{equation}\label{compa7}
 2^{(\gamma-\beta-1/2)\widetilde{k}}2^{2k_+}2^{\gamma j'}R^{-2}\|\mathcal{F}[\widetilde{\varphi}^{(k)}_j\cdot (g_{2,j'}\ast K_k)]\|_{L^1(B(\xi_0,R))}\lesssim (2^{\alpha k}+2^{10k})^{-1}    
\end{equation}
for any $R\in[2^{-j},2^k]$, $\xi_0\in\mathbb{R}^3$, and $j'\in[j-3,j+3]\cap\mathbb{Z}$.

To prove \eqref{compa7} we notice that, for any $\xi\in B(\xi_0,R)$,
\begin{equation*}
\Big|\mathcal{F}[\widetilde{\varphi}^{(k)}_j\cdot (g_{2,j'}\ast K_k)](\xi)\Big|
\lesssim \int_{\mathbb{R}^3}|\widehat{g_{2,j'}}(\xi-\eta)|\,|\mathcal{F}(\widetilde{\varphi}^{(k)}_j)(\eta)|\,d\eta
\lesssim \int_{\mathbb{R}^3}|\widehat{g_{2,j'}}(\xi-\eta)|2^{3j}(1+2^j|\eta|)^{-6}\,d\eta.
\end{equation*}
Therefore
\begin{equation*}
\|\mathcal{F}[\widetilde{\varphi}^{(k)}_j\cdot (g_{2,j'}\ast K_k)]\|_{L^1(B(\xi_0,R))}\lesssim \sup_{\xi_1\in\mathbb{R}^3}
\|\widehat{g_{2,j'}}\|_{L^1(B(\xi_1,R))},
\end{equation*}
and the desired bound \eqref{compa7} follows from \eqref{compa6}.
\end{proof}

We prove now several dispersive estimates.

\begin{lemma}\label{tech1.5} (i) For any $k\in\mathbb{Z}$, $t\in\mathbb{R}$, $\sigma\in\{1,\ldots,d\}$, and $g\in L^1(\mathbb{R}^3)$ we have
\begin{equation}\label{mk15}
\|P_{(-\infty,k]}e^{it\Lambda_\sigma}g\|_{L^\infty}\lesssim (1+|t|)^{-3/2}2^{3k_+}\|g\|_{L^1}.
\end{equation}

(ii) Assume $\|f\|_{Z}\leq 1$, $t\in\mathbb{R}$, $(k,j)\in\mathcal{J}$, and let $\widetilde{k}=\min(k,0)$ and
\begin{equation*}
f_{k,j}:=P_{[k-2,k+2]}[\widetilde{\varphi}^{(k)}_j\cdot P_kf].
\end{equation*}
Then
\begin{equation}\label{mk15.5}
\|f_{k,j}\|_{L^2}\lesssim (2^{\alpha k}+2^{10k})^{-1}\cdot 2^{2\beta\widetilde{k}}2^{-(1-\be)j}
\end{equation}
and
\begin{equation}\label{mk15.55}
\sup_{\xi\in\mathbb{R}^3}\big|D^\rho_\xi\widehat{f_{k,j}}(\xi)\big|\lesssim_{|\rho|} (2^{\alpha k}+2^{10k})^{-1}\cdot 2^{-(1/2-\beta)\widetilde{k}}2^{|\rho|j}.
\end{equation}
Moreover, for $\sigma\in\{1,\ldots,d\}$, if $k\leq 0$ then
\begin{equation}\label{mk15.6}
\begin{split}
\big\|e^{it\Lambda_\sigma}f_{k,j}\big\|_{L^\infty}&\lesssim 2^{-\al k}\min(2^{-(1+\be)j}2^{3k/2},(1+|t|)^{-3/2}2^{(1/2-\beta)j})\\
&+2^{-\al k}\min(2^{(-\gamma+\beta+5/2)k}2^{-\gamma j},(1+|t|)^{-3/2}2^{(1/2+\beta)j}2^{2\beta k}).
\end{split}
\end{equation}
If $k\geq 0$ then
\begin{equation}\label{mk15.65}
\big\|e^{it\Lambda_\sigma}f_{k,j}\big\|_{L^\infty}\lesssim 2^{-6k}\min(2^{-(1+\be)j},(1+|t|)^{-3/2}2^{(1/2-\beta)j})+2^{-6k}\min(2^{-\gamma j},(1+|t|)^{-3/2}2^{(1/2+\beta)j}).
\end{equation}

(iii) As a consequence
\begin{equation}\label{ok9}
\sum_{j\geq \max(-k,0)}\|f_{k,j}\|_{L^2}\lesssim \min(2^{(1+\beta-\alpha)k},2^{-10k})
\end{equation}
and\footnote{In many places we will be able to use the simpler bound \eqref{ok10}, instead of the more precise bounds \eqref{mk15.6} and \eqref{mk15.65}.}
\begin{equation}\label{ok10}
\sum_{j\geq \max(-k,0)}\big\|e^{it\Lambda_\sigma}f_{k,j}\big\|_{L^\infty}\lesssim \min(2^{(1/2-\beta-\alpha)k},2^{-6k})(1+|t|)^{-1-\beta}.
\end{equation}
\end{lemma}

\begin{proof}[Proof of Lemma \ref{tech1.5}] The dispersive bound \eqref{mk15} is well-known. To prove the bounds in (ii), 
we start by decomposing, as in \eqref{sec5.8}--\eqref{sec5.82},
\begin{equation}\label{compat70}
\begin{split}
&\widetilde{\varphi}^{(k)}_{j}\cdot P_kf=g_{1,j}+g_{2,j},\qquad g_{1,j}=g_{1,j}\cdot \widetilde{\varphi}^{(k)}_{[j-2,j+2]},\qquad g_{2,j}=g_{2,j}\cdot \widetilde{\varphi}^{(k)}_{[j-2,j+2]},\\
&2^{(1+\be)j}\|g_{1,j}\|_{L^2}+2^{(1/2-\beta)\widetilde{k}}\|\widehat{g_{1,j}}\|_{L^\infty}\lesssim (2^{\alpha k}+2^{10k})^{-1},\\
&2^{-2\beta\widetilde{k}}2^{(1-\be)j}\|g_{2,j}\|_{L^2}+2^{(1/2-\beta)\widetilde{k}}\|\widehat{g_{2,j}}\|_{L^\infty}+2^{(\gamma-\beta-5/2)\widetilde{k}}2^{\gamma j}\|\widehat{g_{2,j}}\|_{L^1}\lesssim (2^{\alpha k}+2^{10k})^{-1}.
\end{split}
\end{equation}
The bound \eqref{mk15.5} follows easily. To prove \eqref{mk15.55} we use the formulas in the first line of \eqref{compat70} to write, for $\mu=1,2$,
\begin{equation*}
\widehat{g_{\mu,j}}(\xi)=c\int_{\mathbb{R}^3}\widehat{g_{\mu,j}}(\eta)\mathcal{F}(\widetilde{\varphi}^{(k)}_{[j-2,j+2]})(\xi-\eta)\,d\eta.
\end{equation*}
Therefore
\begin{equation*}
D^\rho_\xi\widehat{g_{\mu,j}}(\xi)=c\int_{\mathbb{R}^3}\widehat{g_{\mu,j}}(\eta)\mathcal{F}(x^\rho\cdot\widetilde{\varphi}^{(k)}_{[j-2,j+2]})(\xi-\eta)\,d\eta.
\end{equation*}
The desired bounds \eqref{mk15.55} follow using the bounds $\|\widehat{g_{\mu,j}}\|_{L^\infty}\lesssim(2^{\alpha k}+2^{10k})^{-1}2^{-(1/2-\beta)\widetilde{k}}$, see \eqref{compat70}.

We prove now the bounds \eqref{mk15.6}. Assuming $k\leq 0$ we estimate
\begin{equation*}
\big\|e^{it\Lambda_\sigma}P_{[k-2,k+2]}g_{1,j}\big\|_{L^\infty}\lesssim 2^{3k/2}\|g_{1,j}\|_{L^2}\lesssim 2^{3k/2}\cdot 2^{-\alpha k}2^{-(1+\be)j},
\end{equation*}
and, using \eqref{mk15},
\begin{equation*}
\begin{split}
\big\|e^{it\Lambda_\sigma}P_{[k-2,k+2]}g_{1,j}\big\|_{L^\infty}&\lesssim (1+|t|)^{-3/2}\|g_{1,j}\|_{L^1}\\
&\lesssim (1+|t|)^{-3/2}2^{3j/2}\|g_{1,j}\|_{L^2}\\
&\lesssim (1+|t|)^{-3/2}2^{3j/2}\cdot 2^{-\alpha k}2^{-(1+\be)j}.
\end{split}
\end{equation*}
Therefore
\begin{equation}\label{cucu1}
\big\|e^{it\Lambda_\sigma}P_{[k-2,k+2]}g_{1,j}\big\|_{L^\infty}\lesssim 2^{-\al k}\min(2^{-(1+\be)j}2^{3k/2},(1+|t|)^{-3/2}2^{(1/2-\beta)j}).
\end{equation}
Similarly,
\begin{equation*}
\big\|e^{it\Lambda_\sigma}P_{[k-2,k+2]}g_{2,j}\big\|_{L^\infty}\lesssim \|\widehat{g_{2,j}}\|_{L^1}\lesssim 2^{-\al k}2^{(-\gamma+\beta+5/2)k}2^{-\gamma j},
\end{equation*}
and, using \eqref{mk15},
\begin{equation*}
\begin{split}
\big\|e^{it\Lambda_\sigma}P_{[k-2,k+2]}g_{2,j}\big\|_{L^\infty}&\lesssim (1+|t|)^{-3/2}\|g_{2,j}\|_{L^1}\\
&\lesssim (1+|t|)^{-3/2}2^{3j/2}\|g_{2,j}\|_{L^2}\\
&\lesssim (1+|t|)^{-3/2}2^{3j/2}\cdot 2^{-\alpha k}2^{2\beta k}2^{-(1-\be)j}.
\end{split}
\end{equation*}
Therefore
\begin{equation}\label{cucu2}
\big\|e^{it\Lambda_\sigma}P_{[k-2,k+2]}g_{2,j}\big\|_{L^\infty}\lesssim 2^{-\al k}\min(2^{(-\gamma+\beta+5/2)k}2^{-\gamma j},(1+|t|)^{-3/2}2^{(1/2+\beta)j}2^{2\beta k}).
\end{equation}

Similarly, if $k\geq 0$ then we estimate
\begin{equation*}
\big\|e^{it\Lambda_\sigma}P_{[k-2,k+2]}g_{1,j}\big\|_{L^\infty}\lesssim 2^{3k/2}\|g_{1,j}\|_{L^2}\lesssim 2^{3k/2}\cdot 2^{-10 k}2^{-(1+\be)j},
\end{equation*}
and, using \eqref{mk15},
\begin{equation*}
\begin{split}
\big\|e^{it\Lambda_\sigma}P_{[k-2,k+2]}g_{1,j}\big\|_{L^\infty}&\lesssim (1+|t|)^{-3/2}2^{3k}\|g_{1,j}\|_{L^1}\\
&\lesssim (1+|t|)^{-3/2}2^{3k}2^{3j/2}\|g_{1,j}\|_{L^2}\\
&\lesssim (1+|t|)^{-3/2}2^{3k}2^{3j/2}\cdot 2^{-10 k}2^{-(1+\be)j}.
\end{split}
\end{equation*}
Therefore,
\begin{equation}\label{cucu3}
\big\|e^{it\Lambda_\sigma}P_{[k-2,k+2]}g_{1,j}\big\|_{L^\infty}\lesssim 2^{-6k}\min(2^{-(1+\be)j},(1+|t|)^{-3/2}2^{(1/2-\beta)j}).
\end{equation}
Similarly,
\begin{equation*}
\big\|e^{it\Lambda_\sigma}P_{[k-2,k+2]}g_{2,j}\big\|_{L^\infty}\lesssim \|\widehat{g_{2,j}}\|_{L^1}\lesssim 2^{-10 k}2^{-\gamma j},
\end{equation*}
and, using \eqref{mk15},
\begin{equation*}
\begin{split}
\big\|e^{it\Lambda_\sigma}P_{[k-2,k+2]}g_{2,j}\big\|_{L^\infty}&\lesssim (1+|t|)^{-3/2}2^{3k}\|g_{2,j}\|_{L^1}\\
&\lesssim (1+|t|)^{-3/2}2^{3k}2^{3j/2}\|g_{2,j}\|_{L^2}\\
&\lesssim (1+|t|)^{-3/2}2^{3k}2^{3j/2}\cdot 2^{-10k}2^{-(1-\be)j}.
\end{split}
\end{equation*}
Therefore
\begin{equation}\label{cucu4}
\big\|e^{it\Lambda_\sigma}P_{[k-2,k+2]}g_{2,j}\big\|_{L^\infty}\lesssim 2^{-6k}\min(2^{-\gamma j},(1+|t|)^{-3/2}2^{(1/2+\beta)j}).
\end{equation}
The last bound in \eqref{mk15.6} follows from \eqref{cucu3} and \eqref{cucu4}.

(iii) The desired bounds follow directly from \eqref{mk15.5}, \eqref{mk15.6}, and \eqref{mk15.65}, by summation over $j$. 
\end{proof}

\begin{lemma}\label{tech2}
Assume that $k,k_1,k_2\in\mathbb{Z}$, and $p,q\in[2,\infty]$ satisfy $1/p+1/q=1/2$. Then
\begin{equation}\label{mk6}
\Big\|\int_{\mathbb{R}^3}\varphi_k(\xi)\varphi_{k_1}(\xi-\eta)\varphi_{k_2}(\eta)\cdot\widehat{f}(\xi-\eta)\widehat{g}(\eta)\,d\eta\Big\|_{L^2_\xi}\lesssim\|f\|_{L^p}\|g\|_{L^q}.
\end{equation}
More generally, if $k_1\leq k_2$ and $A_{k;k_1,k_2}:\mathbb{R}\times\mathbb{R}\to\mathbb{C}$ satisfies
\begin{equation}\label{mk6.3}
\sup_{|x|\in[2^{k-1},2^{k+1}],\,|y|\in[2^{k_1-1},2^{k_1+1}]}\sup_{|\rho|,|\sigma|\in[0,4]}\lambda^{-|\rho|}\lambda_1^{-|\sigma|}|D^\rho_xD^\sigma_yA_{k;k_1,k_2}(x,y)|\leq 1,
\end{equation}
for some $\lambda,\lambda_1\in(0,\infty)$, then
\begin{equation}\label{mk6.6}
\begin{split}
\Big\|\int_{\mathbb{R}^3}A_{k;k_1,k_2}(\xi,\xi-\eta)\varphi_k(\xi)&\varphi_{k_1}(\xi-\eta)\varphi_{k_2}(\eta)\cdot\widehat{f}(\xi-\eta)\widehat{g}(\eta)\,d\eta\Big\|_{L^2_\xi}\\
&\lesssim(1+2^{3k}\lambda^3)(1+2^{3k_1}\lambda_1^3)\|f\|_{L^p}\|g\|_{L^q}.
\end{split}
\end{equation}
\end{lemma}

\begin{proof}[Proof of Lemma \ref{tech2}] The bound \eqref{mk6} follows from Plancherel theorem. To prove \eqref{mk6.6}, letting
\begin{equation*}
F(\xi):=\int_{\mathbb{R}^3}A_{k;k_1,k_2}(\xi,\xi-\eta)\varphi_k(\xi)\varphi_{k_1}(\xi-\eta)\varphi_{k_2}(\eta)\cdot\widehat{f}(\xi-\eta)\widehat{g}(\eta)\,d\eta
\end{equation*}
we calculate
\begin{equation*}
(\mathcal{F}^{-1}F)(x)=c\int_{\mathbb{R}^3\times\mathbb{R}^3}f(y)g(z)K_{k;k_1,k_2}(x;y,z)\,dydz,
\end{equation*}
where
\begin{equation*}
K_{k;k_1,k_2}(x;y,z):=\int_{\mathbb{R}^3\times\mathbb{R}^3}e^{i(x-z)\cdot\xi}e^{i(z-y)\cdot\eta}A(\xi,\eta)\varphi_k(\xi)\varphi_{k_1}(\eta)\varphi_{k_2}(\xi-\eta)\,d\xi d\eta.
\end{equation*}
By integration by parts, using \eqref{mk6.3},
\begin{equation*}
|K_{k;k_1,k_2}(x;y,z)|\lesssim 2^{3k}\Big(1+\frac{|x-z|}{2^{-k}+\lambda}\Big)^{-4}\cdot 2^{3k_1}\Big(1+\frac{|z-y|}{2^{-k_1}+\lambda_1}\Big)^{-4},
\end{equation*}
and the desired bound \eqref{mk6.6} follows.
\end{proof}

The following general oscillatory integral estimate is used repeatedly in the proofs. 

\begin{lemma}\label{tech5} Assume that $0<\eps\leq 1/\eps\leq K$, $N\geq 1$ is an integer, and $f,g\in C^N(\mathbb{R}^n)$. Then
\begin{equation}\label{ln1}
\Big|\int_{\mathbb{R}^n}e^{iKf}g\,dx\Big|\lesssim_N (K\eps)^{-N}\big[\sum_{|\rho|\leq N}\eps^{|\rho|}\|D^\rho_xg\|_{L^1}\big],
\end{equation}
provided that $f$ is real-valued, 
\begin{equation}\label{ln2}
|\nabla_x f|\geq \mathbf{1}_{{\mathrm{supp}}\,g},\quad\text{ and }\quad\|D_x^\rho f \cdot\mathbf{1}_{{\mathrm{supp}}\,g}\|_{L^\infty}\lesssim_N\eps^{1-|\rho|},\,2\leq |\rho|\leq N.
\end{equation}
\end{lemma}

\begin{proof}[Proof of Lemma \ref{tech5}] We localize first to balls of size $\approx \eps$. Using the assumptions in \eqref{ln2} we may assume that inside each small ball, one of the directional derivatives of $f$ is bounded away from $0$, say $|\partial_1f|\gtrsim_N 1$. Then we integrate by parts $N$ times in $x_1$, and the desired bound \eqref{ln1} follows.  
\end{proof}

\subsection{Analysis of the functions $\Phi^{\sigma;\mu,\nu}$ and $\Xi^{\mu,\nu}$}\label{ra}
For $\sigma\in\{1,\ldots,d\}$ and $\mu,\nu\in\mathcal{I}_d$, 
\begin{equation}\label{jb1}
\mu=(\sigma_1\iota_1),\qquad \nu=(\sigma_2\iota_2),\qquad \sigma_1,\sigma_2\in\{1,\ldots,d\},\qquad \iota_1,\iota_2\in\{+,-\},
\end{equation}
recall the definitions of the smooth functions $\Lambda_\sigma:\mathbb{R}^3\to(0,\infty)$, $\Phi^{\sigma;\mu,\nu}:\mathbb{R}^3\times\mathbb{R}^3\to\mathbb{R}$ and $\Xi^{\mu,\nu}:\mathbb{R}^3\times\mathbb{R}^3\to\mathbb{R}^3$,
\begin{equation}\label{jb2}
\begin{split}
&\Lambda_\sigma(\xi)=(b_\sigma^2+c_\sigma^2|\xi|^2)^{1/2},\\
&\Phi^{\sigma;\mu,\nu}(\xi,\eta)=\Lambda_\sigma(\xi)-\widetilde{\Lambda}_{\mu}(\xi-\eta)-\widetilde{\Lambda}_{\nu}(\eta),\\
&\Xi^{\mu,\nu}(\xi,\eta)=(\nabla_\eta\Phi^{\sigma;\mu,\nu})(\xi,\eta)=-\iota_1\frac{c_{\sigma_1}^2(\eta-\xi)}{(b_{\sigma_1}^2+c_{\sigma_1}^2|\eta-\xi|^2)^{1/2}}-\iota_2\frac{c_{\sigma_2}^2\eta}{(b_{\sigma_2}^2+c^2_{\sigma_2}|\eta|^2)^{1/2}}.
\end{split}
\end{equation}
In this subsection we prove several lemmas describing the structure of almost resonant sets, which are the sets where both $|\Phi^{\sigma;\mu,\nu}(\xi,\eta)|$ and $|\Xi^{\mu,\nu}(\xi,\eta)|$ are small. 
These lemmas are used at several key places in the proof of Proposition \ref{reduced1}. Recall the sets
\begin{equation}\label{jb3}
\begin{split}
\mathcal{L}^{\sigma;\mu,\nu}_{k,k_1,k_2;\delta_1,\delta_2}=\{(\xi,\eta)\in\mathbb{R}^3\times\mathbb{R}^3:&\,|\xi|\in[2^{k-4},2^{k+4}],\,|\xi-\eta|\in[2^{k_1-4},2^{k_1+4}],\,|\eta|\in[2^{k_2-4},2^{k_2+4}],\\
&\,|\Xi^{\mu,\nu}(\xi,\eta)|\leq\delta_1,\,|\Phi^{\sigma;\mu,\nu}(\xi,\eta)|\leq\delta_2\}.
\end{split}
\end{equation}
defined for $\sigma\in\{1,\ldots,d\}$, $\mu,\nu\in\mathcal{I}_d$, $k,k_1,k_2\in\mathbb{Z}$, $\delta_1,\delta_2\in(0,\infty)$.

\begin{lemma}\label{desc1}
(i) Assume that 
\begin{equation}\label{jb3.5}
k\leq-D/100,\qquad \delta_12^{k_2}\leq 2^{-D/100},\qquad\delta_2\leq 2^{-D/100}.
\end{equation}
Then 
\begin{equation}\label{jb4}
\mathcal{L}^{\sigma;\mu,\nu}_{k,k_1,k_2;\delta_1,\delta_2}=\emptyset.
\end{equation}

(ii) Alternatively, assume that 
\begin{equation}\label{tik1}
\max(k_1,k_2)\geq D/2,\qquad \delta_1\leq 2^{-D}2^{-4\max(k_1,k_2)},\qquad\delta_2\leq 2^{-D}2^{-\max(k_1,k_2)}.
\end{equation}
Then 
\begin{equation}\label{tik2}
\mathcal{L}^{\sigma;\mu,\nu}_{k,k_1,k_2;\delta_1,\delta_2}=\emptyset.
\end{equation}
\end{lemma}

\begin{proof}[Proof of Lemma \ref{desc1}] (i) Assume that there is a point $(\xi,\eta)\in \mathcal{L}^{\sigma;\mu,\nu}_{k,k_1,k_2;\delta_1,\delta_2}$. 
Since $k\leq -D/100$ and $|\Phi^{\sigma;\mu,\nu}(\xi,\eta)|\leq 2^{-D/100}$, using the assumption $|b_{\sigma}\pm b_{\sigma_1}\pm b_{\sigma_2}|\geq 1/A$ 
(see \eqref{abcond2}) it follows that
\begin{equation}\label{jb6.5}
k_1,k_2\geq -C_A,
\end{equation}
where, in this proof, we let $C_A$ denote constants in $[1,\infty)$ that may depend only on $A$. Moreover,
\begin{equation*}
\big|(b^2_\sigma+c^2_\sigma|\xi|^2)^{1/2}-\iota_1(b^2_{\sigma_1}+c^2_{\sigma_1}|\eta-\xi|^2)^{1/2}-\iota_2(b^2_{\sigma_2}+c^2_{\sigma_2}|\eta|^2)^{1/2}\big|\leq 2^{-D/100}.
\end{equation*}
Since 
\begin{equation*}
\big|(b^2_\sigma+c^2_\sigma|\xi|^2)^{1/2}-b_\sigma\big|+\big|(b^2_{\sigma_1}+c^2_{\sigma_1}|\eta-\xi|^2)^{1/2}-(b^2_{\sigma_1}+c^2_{\sigma_1}|\eta|^2)^{1/2}\big|\leq C_A2^{-D/100},
\end{equation*}
it follows that
\begin{equation}\label{jb7}
\big|-b_{\sigma}+\iota_1(b^2_{\sigma_1}+c^2_{\sigma_1}|\eta|^2)^{1/2}+\iota_2(b_{\sigma_2}^2+c^2_{\sigma_2}|\eta|^2)^{1/2}\big|\leq C_A2^{-D/100}.
\end{equation}

Using the definitions \eqref{jb1}--\eqref{jb3}, we see that 
\begin{equation*}
\Big|\iota_1\frac{c^2_{\sigma_1}(\eta-\xi)}{(b^2_{\sigma_1}+c^2_{\sigma_1}|\eta-\xi|^2)^{1/2}}+\iota_2\frac{c^2_{\sigma_2}\eta}{(b^2_{\sigma_2}+c^2_{\sigma_2}|\eta|^2)^{1/2}}\Big|\leq C_A \delta_1.
\end{equation*}
Since
\begin{equation*}
\Big|\frac{c^2_{\sigma_1}(\eta-\xi)}{(b^2_{\sigma_1}+c^2_{\sigma_1}|\eta-\xi|^2)^{1/2}}-\frac{c^2_{\sigma_1}\eta}{(b^2_{\sigma_1}+c^2_{\sigma_1}|\eta|^2)^{1/2}}\Big|\leq C_A 2^{k-k_2},
\end{equation*}
it follows that $\iota_1\cdot\iota_2=-1$ and
\begin{equation*}
\Big|\frac{c^2_{\sigma_1}\eta}{(b^2_{\sigma_1}+c^2_{\sigma_1}|\eta|^2)^{1/2}}-\frac{c^2_{\sigma_2}\eta}{(b^2_{\sigma_2}+c^2_{\sigma_2}|\eta|^2)^{1/2}}\Big|\leq C_A (\delta_1+2^{k-k_2}).
\end{equation*}
Therefore
\begin{equation*}
\big|(c_{\sigma_2}^4c^2_{\sigma_1}-c_{\sigma_1}^4c^2_{\sigma_2})|\eta|^2+(b_{\sigma_1}^2c_{\sigma_2}^4-b_{\sigma_2}^2c_{\sigma_1}^4)\big|\leq C_A(\delta_12^{2k_2}+2^{k+k_2}).
\end{equation*}
In view of the assumption in the second line of \eqref{abcond2}, this implies that
\begin{equation*}
\begin{split}
&\big|c_{\sigma_2}^4c^2_{\sigma_1}-c_{\sigma_1}^4c^2_{\sigma_2}\big||\eta|^2\leq C_A(\delta_12^{2k_2}+2^{k+k_2}),\\
&\big|b_{\sigma_1}^2c_{\sigma_2}^4-b_{\sigma_2}^2c_{\sigma_1}^4\big|\leq C_A(\delta_12^{2k_2}+2^{k+k_2}).
\end{split}
\end{equation*}
Therefore
\begin{equation*}
|c_{\sigma_1}-c_{\sigma_2}|\leq C_A(\delta_1+2^{k-k_2})\,\,\,\text{ and }\,\,\, |b_{\sigma_1}-b_{\sigma_2}|\leq C_A(\delta_12^{2k_2}+2^{k+k_2}),
\end{equation*}
which shows that
\begin{equation*}
\big|(b^2_{\sigma_1}+c^2_{\sigma_1}|\eta|^2)^{1/2}-(b_{\sigma_2}^2+c^2_{\sigma_2}|\eta|^2)^{1/2}\big|\leq C_A(\delta_12^{k_2}+2^{k}).
\end{equation*}
This is in contradiction with \eqref{jb7}, since $\iota_1\cdot\iota_2=-1$ and $2^k+\delta_12^{k_2}\leq C_A2^{-D/100}$.

(ii) As before, assume that there is a point $(\xi,\eta)\in \mathcal{L}^{\sigma;\mu,\nu}_{k,k_1,k_2;\delta_1,\delta_2}$. Assume that $\eta=re$, $\xi=se+v$, $r\in[2^{k_2-4},2^{k_2+4}]$, $e\in\mathbb{S}^2$, $s\in\mathbb{R}$, $v\cdot e=0$. The condition $|\Xi(\xi,\eta)|\leq\delta_1$ gives
\begin{equation*}
\Big|\iota_1\frac{c^2_{\sigma_1}(r-s)}{(b^2_{\sigma_1}+c^2_{\sigma_1}((r-s)^2+|v|^2)^{1/2}}+\iota_2\frac{c^2_{\sigma_2}r}{(b^2_{\sigma_2}+c^2_{\sigma_2}r^2)^{1/2}}\Big|+\frac{c_{\sigma_1}^2|v|}{(b^2_{\sigma_1}+c^2_{\sigma_1}((r-s)^2+|v|^2)^{1/2}}\leq C_A \delta_1.
\end{equation*}
Therefore 
\begin{equation}\label{tik3}
\begin{split}
&2^{\min(k_1,k_2)}\geq C_A^{-1},\quad |v|\leq C_A\delta_12^{\max(k_1,k_2)},\\
&r\in[2^{k_2-6},2^{k_2+6}],\quad |s|\in [2^{k-6},2^{k+6}],\quad |r-s|\in [2^{k_1-6},2^{k_1+6}],
\end{split}
\end{equation}
and
\begin{equation}\label{tik4}
\Big|\iota_1\frac{c^2_{\sigma_1}(r-s)}{(b^2_{\sigma_1}+c^2_{\sigma_1}(r-s)^2)^{1/2}}+\iota_2\frac{c^2_{\sigma_2}r}{(b^2_{\sigma_2}+c^2_{\sigma_2}r^2)^{1/2}}\Big|\leq C_A \delta_1.
\end{equation}

Assume first that 
\begin{equation}\label{tik5}
\min(k_1,k_2)\geq \max(k_1,k_2)-D/10.
\end{equation}
Using \eqref{tik3}--\eqref{tik4} and the assumption \eqref{tik1}, and recalling that $|c_{\sigma_1}-c_{\sigma_2}|\in\{0\}\cup [1/A,\infty)$, see \eqref{abcond2}, it follows that
\begin{equation}\label{tik6}
c_{\sigma_1}=c_{\sigma_2},\qquad\iota_1\iota_2(r-s)<0,\qquad \big|b_{\sigma_2}|r-s|-b_{\sigma_1}r\big|\leq C_A\delta_12^{3\max(k_1,k_2)}.
\end{equation}
As a consequence of the last inequality and the assumption $|b_{\sigma_1}-b_{\sigma_2}|\in\{0\}\cup [1/A,\infty)$, 
\begin{equation}\label{tik7}
\text{ either }|s|\geq 2^{\max(k_1,k_2)-D/10}\qquad\text{ or }\qquad b_{\sigma_1}=b_{\sigma_2}\text{ and }|s|\leq C_A\delta_12^{3\max(k_1,k_2)}. 
\end{equation}
To use the condition $|\Phi^{\sigma;\mu,\nu}(\xi,\eta)|\leq\delta_2$, we estimate first, using \eqref{tik3} and \eqref{tik5},
\begin{equation*}
\begin{split}
&\sqrt{b_{\sigma_1}^2+c_{\sigma_1}^2|\eta-\xi|^2}=c_{\sigma_1}|r-s|+\frac{b_{\sigma_1}^2}{2c_{\sigma_1}|r-s|}+O_A(2^{-3\min(k_1,k_2)}),\\
&\sqrt{b_{\sigma_2}^2+c_{\sigma_2}^2|\eta|^2}=c_{\sigma_2}r+\frac{b_{\sigma_2}^2}{2c_{\sigma_2}r}+O_A(2^{-3\min(k_1,k_2)}).
\end{split}
\end{equation*}
Therefore, using again \eqref{tik3} and \eqref{tik6},
\begin{equation}\label{tik9}
\begin{split}
|\Phi^{\sigma;\mu,\nu}(\xi,\eta)|&=\big|\sqrt{b_{\sigma}^2+c_{\sigma}^2|\xi|^2}-\iota_1\sqrt{b_{\sigma_1}^2+c_{\sigma_1}^2|\eta-\xi|^2}-\iota_2\sqrt{b_{\sigma_2}^2+c_{\sigma_2}^2|\eta|^2}\big|\\
&=\Big|\iota_2\sqrt{b_{\sigma}^2+c_{\sigma}^2s^2}-\iota_1\iota_2\Big(c_{\sigma_1}|r-s|+\frac{b_{\sigma_1}^2}{2c_{\sigma_1}|r-s|}\Big)-\Big(c_{\sigma_2}r+\frac{b_{\sigma_2}^2}{2c_{\sigma_2}r}\Big)\Big|+O_A(2^{-3\min(k_1,k_2)})\\
&=\Big|\iota_2\sqrt{b_{\sigma}^2+c_{\sigma}^2s^2}-c_{\sigma_1}s+\frac{b_{\sigma_1}^2}{2c_{\sigma_1}(r-s)}-\frac{b_{\sigma_2}^2}{2c_{\sigma_1}r}\Big|+O_A(2^{-3\min(k_1,k_2)}).
\end{split}
\end{equation}
We examine now the alternatives in \eqref{tik7}. Clearly, if $|s|\leq C_A\delta_12^{3\max(k_1,k_2)}$ then $|\Phi^{\sigma;\mu,\nu}(\xi,\eta)|\geq C_A^{-1}$, in contradictions with the assumption $|\Phi^{\sigma;\mu,\nu}(\xi,\eta)|\leq\delta_2$. On the other hand, if $|s|\geq 2^{\max(k_1,k_2)-D/10}$, the using \eqref{tik9} and the assumption $|\Phi^{\sigma;\mu,\nu}(\xi,\eta)|\leq\delta_2$, it follows that
\begin{equation}\label{tik10}
c_{\sigma}=c_{\sigma_1},\qquad\iota_2|s|=s,\qquad\Big|\frac{b_{\sigma}^2}{s}+\frac{b_{\sigma_1}^2}{r-s}-\frac{b_{\sigma_2}^2}{r}\Big|\leq C_A2^{-D}2^{-\max(k_1,k_2)}.
\end{equation}

We compare now with the last inequality in \eqref{tik6}, written in the form 
\begin{equation*}
\Big|\frac{b_{\sigma_2}}{r}-\frac{b_{\sigma_1}}{|r-s|}\Big|\leq C_A2^{-D}2^{-\max(k_1,k_2)}.
\end{equation*}
Letting $\lambda:=b_{\sigma_2}/r\in[C_A^{-1}2^{-k_2},C_A2^{-k_2}]$, it follows that $|b_{\sigma_1}-\lambda|r-s||\leq C_A2^{-D}$. Using the last inequality in \eqref{tik10} it follows that $|b_{\sigma}^2-\lambda^2s^2|\leq C_A2^{-D}$. Therefore
\begin{equation*}
\big|b_{\sigma_2}-\lambda r\big|+\big|b_{\sigma_1}-\lambda|r-s|\big|+\big|b_{\sigma}-\lambda|s|\big|\leq C_A2^{-D},
\end{equation*}
which is in contradiction with the assumption in the first line of \eqref{abcond}.

Assume now that 
\begin{equation}\label{tik15}
\min(k_1,k_2)\leq \max(k_1,k_2)-D/10\qquad \text{ and }\qquad k_1\leq k_2.
\end{equation}
Using \eqref{tik3}--\eqref{tik4} and the assumption \eqref{tik1} it follows that
\begin{equation}\label{tik16}
\iota_1\iota_2(r-s)<0,\qquad \Big|\frac{c_{\sigma_1}^2|r-s|}{\sqrt{b_{\sigma_1}^2+c_{\sigma_1}^2|r-s|^2}}-c_{\sigma_2}\Big|\leq C_A2^{-2\max(k_1,k_2)}.
\end{equation}
Since $|r-s|\leq 2^{k_1+6}\leq C_A2^{-D/10}2^{\max(k_1,k_2)}$ it follows from the inequality above that $c_{\sigma_1}>c_{\sigma_2}$, therefore $c_{\sigma_1}\geq c_{\sigma_2}+1/A$. Using again the last inequality in \eqref{tik6}, it follows that $|r-s|\leq C_A$ and $s\geq 2^{k_2-10}$. Therefore we can write
\begin{equation}\label{tik17}
\begin{split}
|\Phi^{\sigma;\mu,\nu}(\xi,\eta)|&=\big|\sqrt{b_{\sigma}^2+c_{\sigma}^2|\xi|^2}-\iota_1\sqrt{b_{\sigma_1}^2+c_{\sigma_1}^2|\eta-\xi|^2}-\iota_2\sqrt{b_{\sigma_2}^2+c_{\sigma_2}^2|\eta|^2}\big|\\
&=\big|c_{\sigma}s-\iota_1\sqrt{b_{\sigma_1}^2+c_{\sigma_1}^2|r-s|^2}-\iota_2c_{\sigma_2}r\big|+O_A(2^{-k_2}).
\end{split}
\end{equation}
Using the assumption $|\Phi^{\sigma;\mu,\nu}(\xi,\eta)|\leq \delta_2$ and the inequalities $|r-s|\leq C_A$ and $s,r\geq 2^{k_2-10}$ proved earlier, it follows that $c_{\sigma}=c_{\sigma_2}$, $\iota_2=1$, and
\begin{equation*}
\big|c_{\sigma_2}|r-s|-\sqrt{b_{\sigma_1}^2+c_{\sigma_1}^2|r-s|^2}\big|\leq C_A2^{-k_2}.
\end{equation*}
It is easy to see that this is in contradiction with the last inequality in \eqref{tik16} and the inequality $c_{\sigma_1}\geq c_{\sigma_2}+1/A$ proved earlier.

The proof in the remaining case 
\begin{equation*}
\min(k_1,k_2)\leq \max(k_1,k_2)-D/10\qquad \text{ and }\qquad k_1\geq k_2
\end{equation*}
is similar. This completes the proof of the lemma.
\end{proof}

To deal with the spacetime resonant region we need a more precise description of the sub-level sets of the functions $\Phi^{\sigma;\mu,\nu}$ and $|\Xi^{\mu,\nu}|$. The estimates in Lemma \ref{desc2} and Lemma \ref{desc4} below are used only in the proof of Proposition \ref{reduced3}. 

We define the functions $r^{\mu,\nu}:(0,\infty)\to\mathbb{R}$, $\mu=(\sigma_1\iota_1)$, $\nu=(\sigma_2\iota_2)$, in the following way:

(a) if $\iota_1\cdot\iota_2=1$ then $r^{\mu,\nu}(s)$ is defined, for any $s>0$, as the unique solution $r\in [0,s]$ of the equation 
\begin{equation}\label{kx12}
\frac{c_{\sigma_1}^4(s-r)^2}{b_{\sigma_1}^2+c^2_{\sigma_1}(s-r)^2}
-\frac{c_{\sigma_2}^4r^2}{b_{\sigma_2}^2+c^2_{\sigma_2}r^2}=0.
\end{equation}

(b) if $\{\iota_1\cdot\iota_2=-1,\,c_{\sigma_1}> c_{\sigma_2}\}$ or if $\{\iota_1\cdot\iota_2=-1, c_{\sigma_1}=c_{\sigma_2},\,b_{\sigma_2}>b_{\sigma_1}\}$ then $r^{\mu,\nu}(s)$ is defined, for any $s>0$, as the unique solution $r\in [s,\infty)$ of the equation 
\begin{equation}\label{kx13}
(c_{\sigma_1}^4c^2_{\sigma_2}-c_{\sigma_2}^4c^2_{\sigma_1})(r-s)^2+c_{\sigma_1}^4b_{\sigma_2}^2(1-s/r)^2-c_{\sigma_2}^4b_{\sigma_1}^2=0.
\end{equation}

(c) if $\{\iota_1\cdot\iota_2=-1,\,c_{\sigma_1}< c_{\sigma_2}\}$ or if $\{\iota_1\cdot\iota_2=-1,\,c_{\sigma_1}=c_{\sigma_2},\,b_{\sigma_1}>b_{\sigma_2}\}$ then $r^{\mu,\nu}(s)$ is defined, for any $s>0$, as the unique solution $r\in (-\infty,0]$ of the equation 
\begin{equation}\label{kx14}
(c_{\sigma_2}^4c^2_{\sigma_1}-c^4_{\sigma_1}c^2_{\sigma_2})r^2+c_{\sigma_2}^4b_{\sigma_1}^2r^2/(r-s)^2-c_{\sigma_1}^4b_{\sigma_2}^2=0.
\end{equation}

The function $r^{\mu,\nu}$ is not defined (nor needed) when $\{\iota_1\cdot\iota_2=-1,\,c_{\sigma_1}=c_{\sigma_2},\,b_{\sigma_1}=b_{\sigma_2}\}$. Notice that $r^{\mu,\nu}$ is well-defined since the functions in \eqref{kx12}--\eqref{kx14} are strictly monotonic (as functions in $r$) and change sign in the respective ranges.

\begin{lemma}\label{desc2}
Assume that $\sigma\in\{1,\ldots,d\}$, $\mu=(\sigma_1\iota_1),\nu=(\sigma_2\iota_2)\in\mathcal{I}_d$, $k,k_1,k_2\in [-D,2D]\cap\mathbb{Z}$, $\delta\in[0,2^{-10D}]$, and assume that there is a point $(\xi,\eta)\in\mathbb{R}^3\times\mathbb{R}^3$ satisfying 
\begin{equation}\label{kx10}
|\xi|\in[2^{k-4},2^{k+4}],\quad|\eta|\in[2^{k_2-4},2^{k_2+4}],\quad|\xi-\eta|\in[2^{k_1-4},2^{k_1+4}],\quad|\Xi^{\mu,\nu}(\xi,\eta)|\leq\delta.
\end{equation}
Then, with $r^{\mu,\nu}$ defined as above and letting $p^{\mu,\nu}(\xi):=r^{\mu,\nu}(|\xi|)\xi/|\xi|$,
\begin{equation}\label{kx11}
\big|\eta-p^{\mu,\nu}(\xi)\big|\leq 2^{8D}\delta\quad\text{ and }\quad\Xi^{\mu,\nu}(\xi,p^{\mu,\nu}(\xi))=0.
\end{equation}
Moreover, for any $s\in[2^{k-6},2^{k+6}]$,
\begin{equation}\label{r1to1}
\begin{split}
&\min\big(|(\partial_sr^{\mu,\nu})(s)|,|1-(\partial_sr^{\mu,\nu})(s)|\big)\geq 2^{-4D},\\
&|(D^\rho_sr^{\mu,\nu})(s)|\leq 2^{20D},\qquad \rho=0,1,\ldots 4.
\end{split}
\end{equation}
\end{lemma}

\begin{proof}[Proof of Lemma \ref{desc2}] We remark first that the existence of a point $(\xi,\eta)$ satisfying \eqref{kx10} implies nontrivial assumptions on $k,k_1,k_2$ and the coefficients $\iota_1,\iota_2,c_{\sigma_1},c_{\sigma_2},b_{\sigma_1},b_{\sigma_2}$. The conclusions of the lemma depend, of course, on the existence of a point $(\xi,\eta)$ satisfying \eqref{kx10}.

We examine the formula \eqref{jb2} and assume that $\xi=|\xi|e$ for some unit 
vector $e\in\mathbb{S}^2$. If $\eta=\rho e+v$ with $\rho\in\mathbb{R}$, $v\in\mathbb{R}^3$, and $v\cdot e=0$, then the condition $|\Xi^{\mu,\nu}(\xi,\eta)|\leq\delta$ shows that
\begin{equation}\label{kx2}
\begin{split}
&\Big|\frac{\iota_1c^2_{\sigma_1}v}{\sqrt{b^2_{\sigma_1}+c^2_{\sigma_1}(|v|^2+(\rho-|\xi|)^2)}}+\frac{\iota_2c^2_{\sigma_2}v}{\sqrt{b^2_{\sigma_2}+c^2_{\sigma_2}(|v|^2+\rho^2)}}\Big|\leq\delta,\\
&\Big|\frac{\iota_1c^2_{\sigma_1}(\rho-|\xi|)}{\sqrt{b^2_{\sigma_1}+c^2_{\sigma_1}(|v|^2+(\rho-|\xi|)^2)}}+\frac{\iota_2c^2_{\sigma_2}\rho}{\sqrt{b^2_{\sigma_2}+c^2_{\sigma_2}(|v|^2+\rho^2)}}\Big|\leq\delta.
\end{split}
\end{equation}
In particular, using the second equation in \eqref{kx2},
\begin{equation*}
\Big|\frac{\iota_1c^2_{\sigma_1}}{\sqrt{b^2_{\sigma_1}+c^2_{\sigma_1}(|v|^2+(\rho-|\xi|)^2)}}+
\frac{\iota_2c^2_{\sigma_2}}{\sqrt{b^2_{\sigma_2}+c^2_{\sigma_2}(|v|^2+\rho^2)}}\Big||\rho|\geq C_A^{-1}2^{k-k_1},
\end{equation*}
where, in this proof, the constants $C_A\in[1,\infty)$ may depend only on the parameter $A$. Since $|\rho|\leq C_A2^{k_2}$ it follows that
\begin{equation*}
\Big|\frac{\iota_1c^2_{\sigma_1}}{\sqrt{b^2_{\sigma_1}+c^2_{\sigma_1}(|v|^2+(\rho-|\xi|)^2)}}+
\frac{\iota_2c^2_{\sigma_2}}{\sqrt{b^2_{\sigma_2}+c^2_{\sigma_2}(|v|^2+\rho^2)}}\Big|\geq C_A^{-1}2^{k-k_1-k_2}.
\end{equation*}
Using now the inequality in the first line of \eqref{kx2} it follows that
\begin{equation}\label{kx3}
 |v|\leq C_A2^{k_1+k_2-k}\delta,\qquad |\rho|\in [2^{k_2-6},2^{k_2+6}],\qquad \big|\rho-|\xi|\big|\in [2^{k_1-6},2^{k_1+6}].
\end{equation}

We analyze now more carefully the inequality in the second line of \eqref{kx2}. Using \eqref{kx3} we see that 
\begin{equation*}
\Big|\frac{c^2_{\sigma_1}(\rho-|\xi|)}{\sqrt{b^2_{\sigma_1}+c^2_{\sigma_1}(|v|^2+(\rho-|\xi|)^2)}}-\frac{c^2_{\sigma_1}(\rho-|\xi|)}{\sqrt{b^2_{\sigma_1}+c^2_{\sigma_1}(\rho-|\xi|)^2}}\Big|
+\Big|\frac{c^2_{\sigma_2}\rho}{\sqrt{b^2_{\sigma_2}+c^2_{\sigma_2}(|v|^2+\rho^2)}}-\frac{c^2_{\sigma_2}\rho}{\sqrt{b^2_{\sigma_2}+c^2_{\sigma_2}\rho^2}}\Big|\leq \delta.
\end{equation*}
Therefore
\begin{equation}\label{kx4}
\Big|\frac{\iota_1c^2_{\sigma_1}(\rho-|\xi|)}{\sqrt{b^2_{\sigma_1}+c^2_{\sigma_1}(\rho-|\xi|)^2}}
+\frac{\iota_2c^2_{\sigma_2}\rho}{\sqrt{b^2_{\sigma_2}+c^2_{\sigma_2}\rho^2}}\Big|\leq 4\delta.
\end{equation}

We consider two cases: if $\iota_1\cdot\iota_2=1$ then $\rho\in[0,|\xi|]$ and the equation \eqref{kx4} shows that
\begin{equation*}
\Big|\frac{c^4_{\sigma_1}(|\xi|-\rho)^2}{b^2_{\sigma_1}+c^2_{\sigma_1}(|\xi|-\rho)^2}
-\frac{c^4_{\sigma_2}\rho^2}{b^2_{\sigma_2}+c^2_{\sigma_2}\rho^2}\Big|\leq C_A\delta.
\end{equation*}
In this case we let $s:=|\xi|$ and use the definition \eqref{kx12}. Using also \eqref{kx3} it follows that $|\rho-r^{\mu,\nu}(s)|\leq C_A2^{6D}\delta$, and the desired conclusion \eqref{kx11} follows in this case.

Assume now $\iota_1\cdot\iota_2=-1$ and either $c_{\sigma_1}> c_{\sigma_2}$ or $\{c_{\sigma_1}=c_{\sigma_2},\,b_{\sigma_2}>b_{\sigma_1}\}$. Using \eqref{kx3}, \eqref{kx4}, 
and the assumption \eqref{abcond2}, it follows that $c_{\sigma_1}^4b_{\sigma_2}^2\geq c_{\sigma_2}^4b_{\sigma_1}^2$, $\rho\in[|\xi|,\infty)$, 
$k_1\leq k_2+10$, and
\begin{equation*}
 \Big|\frac{c^4_{\sigma_1}(\rho-|\xi|)^2}{b^2_{\sigma_1}+c^2_{\sigma_1}(\rho-|\xi|)^2}
-\frac{c_{\sigma_2}^4\rho^2}{b^2_{\sigma_2}+c^2_{\sigma_2}\rho^2}\Big|\leq C_A\delta.
\end{equation*}
Therefore
\begin{equation*}
\big|(c_{\sigma_1}^4c^2_{\sigma_2}-c_{\sigma_2}^4c^2_{\sigma_1})(\rho-|\xi|)^2+c_{\sigma_1}^4b_{\sigma_2}^2(1-|\xi|/\rho)^2-
c_{\sigma_2}^4b_{\sigma_1}^2\big|\leq C_A\delta (1+2^{2k_1}).
\end{equation*}
Recall that either $(c^2_{\sigma_1}-c^2_{\sigma_2})\geq C_A^{-1}$ or $c_{\sigma_1}^4b_{\sigma_2}^2-
c_{\sigma_2}^4b_{\sigma_1}^2\geq C_A^{-1}$. Then we let, as before, $s:=|\xi|$ and use the definition \eqref{kx13}. The conclusion \eqref{kx11} follows, using also \eqref{kx3}. 

The argument is similar if $\iota_1\cdot\iota_2=-1$ and either $c_{\sigma_1}<c_{\sigma_2}$ or $\{c_{\sigma_1}=c_{\sigma_2},\,b_{\sigma_2}<b_{\sigma_1}\}$. Using \eqref{kx3}, \eqref{kx4}, 
and the assumption \eqref{abcond2}, it follows that $c_{\sigma_2}^4b_{\sigma_1}^2\geq c_{\sigma_1}^4b_{\sigma_2}^2$, $\rho\in(-\infty,0]$, 
$k_2\leq k_1+10$, and
\begin{equation*}
 \Big|\frac{c^4_{\sigma_1}(\rho-|\xi|)^2}{b^2_{\sigma_1}+c^2_{\sigma_1}(\rho-|\xi|)^2}
-\frac{c_{\sigma_2}^4\rho^2}{b^2_{\sigma_2}+c^2_{\sigma_2}\rho^2}\Big|\leq C_A\delta.
\end{equation*}
Therefore
\begin{equation*}
\big|(c_{\sigma_2}^4c^2_{\sigma_1}-c_{\sigma_1}^4c^2_{\sigma_2})\rho^2+c_{\sigma_2}^4b_{\sigma_1}^2\rho^2/(\rho-|\xi|)^2-
c_{\sigma_1}^4b_{\sigma_2}^2\big|\leq C_A\delta (1+2^{2k_2}).
\end{equation*}
Then we let $s:=|\xi|$ and use the definition \eqref{kx14}. The conclusion \eqref{kx11} follows, using also \eqref{kx3} and the fact that either $(c^2_{\sigma_2}-c^2_{\sigma_1})\geq C_A^{-1}$ or $c_{\sigma_2}^4b_{\sigma_1}^2-
c_{\sigma_1}^4b_{\sigma_2}^2\geq C_A^{-1}$.

To prove \eqref{r1to1} we let, for simplicity of notation, $r(s)=r^{\mu,\nu}(s)$. We differentiate \eqref{kx12}, so
\begin{equation*}
\left[\frac{c_{\sigma_1}^4b_{\sigma_1}^2(s-r)}{\left[b_{\sigma_1}^2+c_{\sigma_1}^2(s-r)^2\right]^2}+\frac{c_{\sigma_2}^4b_{\sigma_2}^2r}{\left[b_{\sigma_2}^2+c_{\sigma_2}^2r^2\right]^2}\right]\cdot r'(s)=\frac{c_{\sigma_1}^4b_{\sigma_1}^2(s-r)}{\left[b_{\sigma_1}^2+c_{\sigma_1}^2(s-r)^2\right]^2}.
\end{equation*}
Using again the equation \eqref{kx12} it follows that
\begin{equation*}
r'(s)=\frac{b_{\sigma_1}^2c_{\sigma_2}^4r^3}{b_{\sigma_1}^2c_{\sigma_2}^4r^3+b_{\sigma_2}^2c_{\sigma_1}^4(s-r)^3}.
\end{equation*}
The desired bounds in \eqref{r1to1} follow easily in this case since $r(s)\approx 2^{k_2}$, $s-r(s)\approx 2^{k_1}$.

Similarly, we differentiate \eqref{kx13} to get
\begin{equation*}
\big[r^2(c_{\sigma_1}^4c_{\sigma_2}^2-c_{\sigma_2}^4c_{\sigma_1}^2)+c_{\sigma_1}^4b_{\sigma_2}^2s/r\big]\cdot r'(s)=\big[r^2(c_{\sigma_1}^4c_{\sigma_2}^2-c_{\sigma_2}^4c_{\sigma_1}^2)+c_{\sigma_1}^4b_{\sigma_2}^2\big],
\end{equation*}
which gives 
\begin{equation*}
r'(s)=1+\frac{c_{\sigma_1}^4b_{\sigma_2}^2(r-s)}{(c_{\sigma_1}^4c_{\sigma_2}^2-c_{\sigma_2}^4c_{\sigma_1}^2)r^3+c_{\sigma_1}^4b_{\sigma_2}^2s}.
\end{equation*}
The desired bounds in \eqref{r1to1} follow easily in this case as well.

Finally, we differentiate \eqref{kx14} to get
\begin{equation*}
\big[(c_{\sigma_2}^4c_{\sigma_1}^2-c_{\sigma_1}^4c_{\sigma_2}^2)(s-r)^3+c_{\sigma_2}^4b_{\sigma_1}^2s\big]\cdot r'(s)=c_{\sigma_2}^4b_{\sigma_1}^2r,
\end{equation*}
which gives 
\begin{equation*}
r'(s)=\frac{c_{\sigma_2}^4b_{\sigma_1}^2r}{(c_{\sigma_2}^4c_{\sigma_1}^2-c_{\sigma_1}^4c_{\sigma_2}^2)(s-r)^3+c_{\sigma_2}^4b_{\sigma_1}^2s},
\end{equation*}
and the desired bounds in \eqref{r1to1} follow easily.
\end{proof}

\begin{remark}\label{tik100}
The conclusions of Lemma \ref{desc2} hold, in a suitable sense, without making the assumption $k,k_1,k_2\leq 2D$. More precisely, to prove the bound \eqref{nxc8}, we need the following slightly stronger version: assume that $\sigma\in\{1,\ldots,d\}$, $\mu=(\sigma_1\iota_1),\nu=(\sigma_2\iota_2)\in\mathcal{I}_d$, $k,k_1,k_2\in [-D,\infty)\cap\mathbb{Z}$, $\delta\in[0,2^{-8D}2^{-4\max(k_1,k_2)}]$, and assume that there is a point $(\xi,\eta)\in\mathbb{R}^3\times\mathbb{R}^3$ satisfying 
\begin{equation}\label{tik101}
|\xi|\in[2^{k-4},2^{k+4}],\quad|\eta|\in[2^{k_2-4},2^{k_2+4}],\quad|\xi-\eta|\in[2^{k_1-4},2^{k_1+4}],\quad|\Xi^{\mu,\nu}(\xi,\eta)|\leq\delta.
\end{equation}
Then, with $r^{\mu,\nu}$ defined as in \eqref{kx12}--\eqref{kx14}, and letting $p^{\mu,\nu}(\xi)=r^{\mu,\nu}(|\xi|)\xi/|\xi|$,
\begin{equation}\label{tik102}
\big|\eta-p^{\mu,\nu}(\xi)\big|\lesssim 2^{4\max(k_1,k_2)}\delta\quad\text{ and }\quad\Xi^{\mu,\nu}(\xi,p^{\mu,\nu}(\xi))=0.
\end{equation}
The proof of \eqref{tik102} is similar to the proof of \eqref{kx11} given above.
\end{remark}

\begin{lemma}\label{desc4}
As in Lemma \ref{desc2}, assume that $\sigma\in\{1,\ldots,d\}$, $\mu=(\sigma_1\iota_1),\nu=(\sigma_2\iota_2)\in\mathcal{I}_d$, $k,k_1,k_2\in [-D,2D]\cap\mathbb{Z}$, and assume that there is a point $(\xi,\eta)\in\mathbb{R}^3\times\mathbb{R}^3$ satisfying 
\begin{equation}\label{kxz10}
|\xi|\in[2^{k-4},2^{k+4}],\quad|\eta|\in[2^{k_2-4},2^{k_2+4}],\quad|\xi-\eta|\in[2^{k_1-4},2^{k_1+4}],\quad|\Xi^{\mu,\nu}(\xi,\eta)|\leq 2^{-10D}.
\end{equation}
We define the function $\Psi^{\si;\mu,\nu}:[2^{k-4},2^{k+4}]\to\mathbb{R}$
\begin{equation}\label{kxz11}
\begin{split}
\Psi^{\si;\mu,\nu}(s)&:=\Phi^{\sigma;\mu,\nu}(se,r^{\mu,\nu}(s)e)\\
&=(b_\sigma^2+c^2_\sigma s^2)^{1/2}-\iota_1\big[b_{\sigma_1}^2+c^2_{\sigma_1}(r^{\mu,\nu}(s)-s)^2\big]^{1/2}
-\iota_2\big[b_{\sigma_2}^2+c^2_{\sigma_2}r^{\mu,\nu}(s)^2\big]^{1/2},
\end{split}
\end{equation}
for some $e\in\mathbb{S}^2$ (the definition, of course, does not depend on the choice of $e$). Then there is some constant $\widetilde{c}=\widetilde{c}(\iota_1,\iota_2,c_{\sigma},b_{\sigma},c_{\sigma_1},b_{\sigma_1},c_{\sigma_2},b_{\sigma_2})\in\{-1,1\}$ with the property that
\begin{equation}\label{kx20}
\text{ if }\,s\in[2^{k-4},2^{k+4}]\,\text{ and }\,|\Psi^{\si;\mu,\nu}(s)|\leq 2^{-20D}\,\text{ then }\,\widetilde{c}(\partial_s\Psi^{\si;\mu,\nu})(s)\geq 2^{-20D}.
\end{equation}
\end{lemma}

\begin{proof}[Proof of Lemma \ref{desc4}] For simplicity of notation, let $\Psi(s):=\Psi^{\si;\mu,\nu}(s)$ and $r(s):=r^{\mu,\nu}(s)$ in the rest of the proof. Recalling that $\Xi^{\mu,\nu}(\xi,r^{\mu,\nu}(|\xi|)\xi/|\xi|)=0$, it follows that
\begin{equation}\label{kx19}
\Psi'(s)=\frac{c_{\sigma}^2s}{\sqrt{b_{\sigma}^2+c_{\sigma}^2s^2}}+\frac{\iota_1c_{\sigma_1}^2(r(s)-s)}{\sqrt{b_{\sigma_1}^2+c_{\sigma_1}^2(r(s)-s)^2}}.
\end{equation}

Recall the identity, see \eqref{kx4}, 
\begin{equation}\label{kx24}
\frac{\iota_1c^2_{\sigma_1}(r(s)-s)}{\sqrt{b^2_{\sigma_1}+c^2_{\sigma_1}(r(s)-s)^2}}
+\frac{\iota_2c^2_{\sigma_2}r(s)}{\sqrt{b^2_{\sigma_2}+c^2_{\sigma_2}r(s)^2}}=0.
\end{equation}
Recalling \eqref{kx12}--\eqref{kx14}, in proving \eqref{kx20} we need to consider five cases,
\begin{equation}\label{kx41}
(\iota_1,\iota_2)=(1,1)\quad\text{ and }\quad r(s)\in[0,s],
\end{equation}
or
\begin{equation}\label{kx42}
(\iota_1,\iota_2)=(-1,1),\quad c_{\sigma_1}\geq c_{\sigma_2},\quad \text{ and }\quad r(s)\in[s,\infty),
\end{equation}
or
\begin{equation}\label{kx43}
(\iota_1,\iota_2)=(1,-1),\quad c_{\sigma_1}\geq c_{\sigma_2},\quad \text{ and }\quad r(s)\in[s,\infty),
\end{equation}
or
\begin{equation}\label{kx44}
(\iota_1,\iota_2)=(1,-1),\quad c_{\sigma_1}\leq c_{\sigma_2},\quad \text{ and }\quad r(s)\in(-\infty,0],
\end{equation}
or
\begin{equation}\label{kx45}
(\iota_1,\iota_2)=(-1,1),\quad c_{\sigma_1}\leq c_{\sigma_2},\quad \text{ and }\quad r(s)\in(-\infty,0].
\end{equation}
The desired lower bound in \eqref{kx20} follows easily from the identities \eqref{kx19} and \eqref{kx24}, with $\widetilde{c}:=1$, in the cases \eqref{kx43} and \eqref{kx45}.

We consider now the case described in \eqref{kx41} and rewite, using \eqref{kx19} and \eqref{kx24},
\begin{equation}\label{kx51}
\begin{split}
&\Psi'(s)=\frac{c_{\sigma}^2s}{\sqrt{b_{\sigma}^2+c_{\sigma}^2s^2}}-\frac{c_{\sigma_1}^2(s-r(s))}{\sqrt{b_{\sigma_1}^2+c_{\sigma_1}^2(s-r(s))^2}}=\frac{c_{\sigma}^2s}{\sqrt{b_{\sigma}^2+c_{\sigma}^2s^2}}-\frac{c^2_{\sigma_2}r(s)}{\sqrt{b^2_{\sigma_2}+c^2_{\sigma_2}r(s)^2}},\\
&\Psi(s)=\sqrt{b_\sigma^2+c^2_\sigma s^2}-\sqrt{b_{\sigma_1}^2+c^2_{\sigma_1}(s-r(s))^2}
-\sqrt{b_{\sigma_2}^2+c^2_{\sigma_2}r(s)^2}.
\end{split}
\end{equation}
If $c_{\sigma}> c_{\sigma_1}$ then 
$c_{\sigma}^2b_{\sigma_1}\geq c_{\sigma_1}^2b_{\sigma}$ (see \eqref{abcond2}) and the inequality $\Psi'(s)\geq 2^{-10D}$ follows 
easily from \eqref{kx51}, since $|s|\approx_A 2^{k},|r(s)|\approx_A 2^{k_2},|s-r(s)|\approx_A 2^{k_1}$. Similarly, if $c_{\sigma}> c_{\sigma_2}$ then 
$c_{\sigma}^2b_{\sigma_2}\geq c_{\sigma_2}^2b_{\sigma}$ (see \eqref{abcond2}) and the inequality $\Psi'(s)\gtrsim 2^{-10D}$ follows 
easily from \eqref{kx51}. 

On the other hand, if $c_\sigma\leq\min(c_{\sigma_1},c_{\sigma_2})$, we consider two cases: assume first that
\begin{equation*}
\max(c_{\sigma_1},c_{\sigma_2})\geq c_\sigma+1/A,\qquad\min(c_{\sigma_1},c_{\sigma_2})\geq c_\sigma.
\end{equation*}
In this case we estimate, using \eqref{kx51} and the assumption $|\Psi(s)|\leq 2^{-20D}$,
\begin{equation*}
\begin{split}
-\Psi'(s)&=\frac{c_{\sigma_1}^2(s-r(s))+c^2_{\sigma_2}r(s)}{\sqrt{b_{\sigma_1}^2+c_{\sigma_1}^2(s-r(s))^2}+\sqrt{b^2_{\sigma_2}+c^2_{\sigma_2}r(s)^2}}-\frac{c_{\sigma}^2s}{\sqrt{b_{\sigma}^2+c_{\sigma}^2s^2}}\\
&\geq \frac{c_{\sigma_1}^2(s-r(s))+c^2_{\sigma_2}r(s)-c_{\sigma}^2s}{\sqrt{b_{\sigma}^2+c_{\sigma}^2s^2}}-2^{-10D}\\
&\geq 2^{-10D}.
\end{split}
\end{equation*}
The desired bound \eqref{kx20} follows. 

In the remaining case
\begin{equation*}
c_\sigma=c_{\sigma_1}=c_{\sigma_2},
\end{equation*}
we show that $|\Psi(s)|\geq 2^{-10D}$, which would suffice to prove \eqref{kx20} (since the hypothesis in \eqref{kx20} does not hold). Indeed, the identity \eqref{kx24} shows that
\begin{equation*}
b_{\sigma_1}^2r(s)^2-b_{\sigma_2}^2(s-r(s))^2=0.
\end{equation*}
Letting $\kappa:=b_{\sigma_1}/b_{\sigma_2}=(s-r(s))/r(s)\in[1/A^2,A^2]$ and using also the assumption $|b_{\sigma}-b_{\sigma_1}-b_{\sigma_2}|\geq 1/A$ (see \eqref{abcond2}), we estimate
\begin{equation*}
\begin{split}
|\Psi(s)|&=\Big|\sqrt{b_\sigma^2+c^2_\sigma s^2}-\sqrt{\kappa^2b_{\sigma_2}^2+c^2_{\sigma}\kappa^2r(s)^2}
-\sqrt{b_{\sigma_2}^2+c^2_{\sigma}r(s)^2}\Big|\\
&\geq 2^{-3D}\Big|(b_\sigma^2+c^2_\sigma s^2)-(\kappa+1)^2(b_{\sigma_2}^2+c^2_{\sigma}r(s)^2)\Big|\\
&\geq 2^{-3D}C_A^{-1}|b_\sigma-(\ka+1)b_{\sigma_2}|\\
&\geq 2^{-3D}C_A^{-1},
\end{split}
\end{equation*}
as desired.

We consider now the case described in \eqref{kx42} and rewite, using \eqref{kx19} and \eqref{kx24},
\begin{equation}\label{kx55}
\begin{split}
&\Psi'(s)=\frac{c_{\sigma}^2s}{\sqrt{b_{\sigma}^2+c_{\sigma}^2s^2}}-\frac{c_{\sigma_1}^2(r(s)-s)}{\sqrt{b_{\sigma_1}^2+c_{\sigma_1}^2(r(s)-s)^2}}=\frac{c_{\sigma}^2s}{\sqrt{b_{\sigma}^2+c_{\sigma}^2s^2}}-\frac{c^2_{\sigma_2}r(s)}{\sqrt{b^2_{\sigma_2}+c^2_{\sigma_2}r(s)^2}},\\
&\Psi(s)=\sqrt{b_\sigma^2+c^2_\sigma s^2}+\sqrt{b_{\sigma_1}^2+c^2_{\sigma_1}(r(s)-s)^2}
-\sqrt{b_{\sigma_2}^2+c^2_{\sigma_2}(r(s)^2)}.
\end{split}
\end{equation}
If $c_{\sigma_2}> c_{\sigma}$ then 
$c_{\sigma_2}^2b_{\sigma}\geq c_{\sigma}^2b_{\sigma_2}$ (see \eqref{abcond2}) and the inequality $-\Psi'(s)\geq 2^{-10D}$ follows 
easily from \eqref{kx55}, since $|s|\approx_A 2^{k},|r(s)|\approx_A 2^{k_2},|s-r(s)|\approx_A 2^{k_1}$. On the other hand, if $c_{\sigma_2}\leq \min(c_{\sigma},c_{\sigma_1})$ then, as before, we consider two cases. If
\begin{equation*}
\max(c_{\sigma},c_{\sigma_1})\geq c_{\sigma_2}+1/A,\qquad\min(c_{\sigma},c_{\sigma_1})\geq c_{\sigma_2},
\end{equation*}
then, using \eqref{kx55} and the assumption $|\Psi(s)|\leq 2^{-20D}$, we estimate
\begin{equation*}
\begin{split}
\Psi'(s)&=\frac{c_{\sigma}^2s}{\sqrt{b_{\sigma}^2+c_{\sigma}^2s^2}}-\frac{c^2_{\sigma_2}r(s)-c_{\sigma_1}^2(r(s)-s)}{\sqrt{b^2_{\sigma_2}+c^2_{\sigma_2}r(s)^2}-\sqrt{b_{\sigma_1}^2+c_{\sigma_1}^2(r(s)-s)^2}}\\
&\geq \frac{c_{\sigma}^2s-c^2_{\sigma_2}r(s)+c_{\sigma_1}^2(r(s)-s)}{\sqrt{b_{\sigma}^2+c_{\sigma}^2s^2}}-2^{-10D}\\
&\geq 2^{-10 D},
\end{split}
\end{equation*}
as desired. 

On the other hand, if
\begin{equation*}
c_{\sigma_2}=c_{\sigma}=c_{\sigma_1}
\end{equation*}
we show that $|\Psi(s)|\geq 2^{-10D}$, which would suffice to prove \eqref{kx20} (since the hypothesis in \eqref{kx20} does not hold). Indeed, arguing as before, the identity \eqref{kx24} shows that
\begin{equation*}
b_{\sigma_1}^2r(s)^2-b_{\sigma_2}^2(r(s)-s)^2=0.
\end{equation*}
Letting $\kappa:=b_{\sigma_1}/b_{\sigma_2}=(r(s)-s)/r(s)\in[1/A^2,1]$ and using the assumption $|b_{\sigma}+b_{\sigma_1}-b_{\sigma_2}|\geq 1/A$ (see \eqref{abcond2}), we estimate
\begin{equation*}
\begin{split}
|\Psi(s)|&=\Big|\sqrt{b_\sigma^2+c^2_{\sigma} s^2}+\sqrt{\kappa^2b_{\sigma_2}^2+c^2_{\sigma}\kappa^2r(s)^2}
-\sqrt{b_{\sigma_2}^2+c^2_{\sigma}r(s)^2}\Big|\\
&\geq 2^{-3D}\Big|(b_\sigma^2+c^2_{\sigma} s^2)-(1-\kappa)^2(b_{\sigma_2}^2+c^2_{\sigma}r(s)^2)\Big|\\
&\geq 2^{-3D}C_A^{-1}|b_\sigma-(1-\ka)b_{\sigma_2}|\\
&\geq 2^{-3D}C_A^{-1},
\end{split}
\end{equation*}
as desired.

The analysis in the case described in \eqref{kx44} is similar. This completes the proof of the lemma. 
\end{proof}

\end{document}